\documentclass[11pt, reqno]{amsart}

\usepackage{graphicx}
\usepackage{epstopdf}
\usepackage{subcaption}
\usepackage[utf8]{inputenc}
\usepackage[T1]{fontenc}
\usepackage[english]{babel}
\usepackage{amsmath, amsthm}
\usepackage{ae}
\usepackage{icomma}
\usepackage{units}
\usepackage{color}
\usepackage{graphicx}
\usepackage{bbm}
\usepackage{caption}
\usepackage{array}
\usepackage[hmarginratio=1:1]{geometry}
\usepackage[hyphens]{url}
\usepackage[pdfpagelabels=false, hidelinks]{hyperref}
\usepackage{mathrsfs}
\usepackage{amssymb, amsfonts}
\usepackage{placeins} 
\usepackage{tikz}
\usetikzlibrary{patterns}
\usepackage{float}
\usepackage[nodayofweek]{datetime}
\usepackage{enumitem}
\usepackage{mathtools}
\usepackage[numbers, sort]{natbib}
\usepackage{dsfont}
\usepackage{comment}

\newcommand{\R}{\ensuremath{\mathbb{R}}}
\newcommand{\N}{\ensuremath{\mathbb{N}}}

\newcommand{\Haus}{\ensuremath{\mathcal{H}}}

\newcommand{\1}{\ensuremath{\mathds{1}}}

\DeclareMathOperator{\Tr}{Tr}

\newtheorem{theorem}{Theorem}[section]

\newtheorem{lemma}[theorem]{Lemma}

\newtheorem{corollary}[theorem]{Corollary}
\newtheorem{conjecture}[theorem]{Conjecture}
\numberwithin{theorem}{section}
\theoremstyle{remark}
\newtheorem{remark}[theorem]{Remark}

\newcommand{\limplus}{{\mathchoice{\vcenter{\hbox{$\scriptstyle +$}}}
  {\vcenter{\hbox{$\scriptstyle +$}}}
  {\vcenter{\hbox{$\scriptscriptstyle +$}}}
  {\vcenter{\hbox{$\scriptscriptstyle +$}}}
}}
\newcommand{\limminus}{{\mathchoice{\vcenter{\hbox{$\scriptstyle -$}}}
  {\vcenter{\hbox{$\scriptstyle -$}}}
  {\vcenter{\hbox{$\scriptscriptstyle -$}}}
  {\vcenter{\hbox{$\scriptscriptstyle -$}}}
}}
\newcommand{\limpm}{{\mathchoice{\vcenter{\hbox{$\scriptstyle \pm$}}}
  {\vcenter{\hbox{$\scriptstyle \pm$}}}
  {\vcenter{\hbox{$\scriptscriptstyle \pm$}}}
  {\vcenter{\hbox{$\scriptscriptstyle \pm$}}}
}}

\begin{document}

\title[Optimizing Riesz means of Robin Laplace operators]{Optimizing Riesz means of Robin Laplace operators\\ on cuboids in a semiclassical limit}

\author{Matthias Baur}
\address{\textnormal{(M. Baur)} Institute of Analysis, Dynamics and Modeling, Department of Mathematics, University of
Stuttgart, Pfaffenwaldring 57, 70569 Stuttgart, Germany.}
\email{\href{mailto:matthias.baur@mathematik.uni-stuttgart.de}{matthias.baur@mathematik.uni-stuttgart.de}}

\author{Simon Larson}
\address{\textnormal{(S. Larson)} Mathematical Sciences, Chalmers University of Technology and the University of Gothenburg, SE-412 96 G\"{o}teborg, Sweden. }
\email{\href{mailto:larsons@chalmers.se}{larsons@chalmers.se}}

\begin{abstract}
    We study asymptotic shape optimization for Riesz means of Robin Laplacian eigenvalues among cuboids of fixed measure. Our focus is the regime where the Robin parameter is proportional to the square root of the spectral parameter defining the Riesz means. Here, a transition emerges based on the precise ratio between the two parameters: as the spectral parameter tends to infinity, sequences of maximizers shift from converging to the unit cube to lacking convergent subsequences entirely. Key tools include two-term spectral asymptotics and uniform inequalities for the Riesz means. Notably, the transition point governing the behavior of optimizers may differ from the point at which the second asymptotic term changes sign. This shows that heuristics based solely on asymptotics for a fixed domain fail to accurately predict the asymptotic behavior of maximizers.
\end{abstract}

\maketitle


\section{Introduction and main results}

In this paper, we consider a family of spectral shape optimization problems for Laplace operators with Robin boundary conditions. More specifically, we consider Riesz means of the eigenvalues of these operators in a joint limit where the spectral parameter defining the Riesz mean and the Robin parameter tend to infinity simultaneously. 

Let $\Omega \subset \R^d$ be a bounded open set with Lipschitz regular boundary. For $\beta\in \R$ we define the operator
$-\Delta_\Omega^\beta$
on $L^2(\Omega)$ as the unique self-adjoint operator associated to the quadratic form 
$$
  u \mapsto\int_\Omega |\nabla u(x)|^2\, dx + \beta\int_{\partial \Omega} |u(x)|^2\, d\Haus^{d-1}(x)
$$ 
with form domain $H^1(\Omega)$ (see, e.g., \cite[Section~3.1]{FrankLaptevWeidl}). The assumptions on $\partial\Omega$ ensure that this quadratic form is lower semibounded and closed. Moreover, the quadratic form $u \mapsto \beta\int_{\partial\Omega}|u|^2\, d\Haus^{d-1}(x)$ is compact in $H^1(\Omega)$. Therefore, Rellich's compactness theorem ensures that the operator $-\Delta_\Omega^\beta$ has discrete spectrum. We denote its eigenvalues in non-decreasing order by $\{\lambda_k(-\Delta_\Omega^{\beta})\}_{k\geq 1}$, so that
\begin{equation*}
    \lambda_1(-\Delta_\Omega^\beta)\leq \lambda_2(-\Delta_\Omega^\beta)\leq \lambda_3(-\Delta_\Omega^\beta)\leq \ldots \to \infty\, .
\end{equation*}

In this paper we only consider $\beta>0$. In this case the Robin Laplace operators naturally lie in between the corresponding Dirichlet and Neumann operators, and in a certain sense form a continuous family interpolating between the two. We denote the Dirichlet and Neumann Laplace operators in $L^2(\Omega)$ by $-\Delta_\Omega^{\rm D}$ and $-\Delta_{\Omega}^{\rm N}$, respectively. The Neumann case is obtained by taking $\beta = 0$ in the definition above. The Dirichlet Laplacian is defined through the quadratic form 
$
  u \mapsto\int_\Omega |\nabla u(x)|^2\, dx
$
with form domain $H_0^1(\Omega)$. Formally, $-\Delta_\Omega^{\rm D}$ can be thought of as corresponding to the limiting case $\beta=\infty$. 

Our interest in this paper is primarily focused on the asymptotic regime of large eigenvalues. We study this regime by considering so-called Riesz means. Let $H$ be a self-adjoint operator whose spectrum consists of eigenvalues which we denote by $\{\lambda_k(H)\}_{k\geq 1}$. For $\gamma >0, \lambda\in \R$ the Riesz mean of order $\gamma$ of $H$ is given by
\begin{equation*}
    \Tr(H-\lambda)_\limminus^\gamma := \sum_{k\geq 1}(\lambda-\lambda_k(H))_\limplus^\gamma\, , 
\end{equation*}
here and in what follows we write $x_\limpm := (|x|\pm x)/2$.
The Riesz mean of order $0$ is defined to be the classical counting function. To unify notation, we write
\begin{equation*}
    \Tr(H-\lambda)_\limminus^0 := \#\{k:\lambda_k(H)\leq \lambda\}\, .
\end{equation*}
Note that we define the counting function to count eigenvalues $\leq \lambda$ and not $<\lambda$. The latter convention is perhaps more commonly used in the literature, but in most settings the choice of convention has little or no effects. However, we shall be considering optimization problems where the aim is to maximize either Riesz means or a counting function. In this setting, the convention for the counting function chosen here has the advantage of being upper semicontinuous as a function of both $\lambda$ and the eigenvalues of $H$.

\subsection{An asymptotic shape optimization problem} The main topic of this paper is the asymptotic behavior of maximizers of the Riesz mean $\Tr(-\Delta_\Omega^\beta-\lambda)_\limminus^\gamma$ among cuboids of unit measure in limits where the parameters $\lambda$ and $\beta$ both tend to infinity. Here and in what follows, a cuboid $R\subset \R^d$ is an open set defined by $R = \prod_{i=1}^d (0, l_i)$ with $l_1, \ldots, l_d >0$. Before stating our results, we need to introduce the following notation. For $d\geq 2, \gamma\geq0, \beta>0, \lambda \geq 0$, define
\begin{equation*}
    M_{\gamma, d}(\lambda, \beta) := \sup\bigl\{\Tr(-\Delta_R^\beta-\lambda)_\limminus^\gamma: R\subset \R^d \mbox{ a cuboid with }|R|=1\bigr\}\, .
\end{equation*}
For any $d\geq 2, \gamma \geq 0, \beta>0, \lambda \geq 0$ there exists at least one cuboid of unit measure that realizes the above supremum (see Lemma~\ref{lem: existence of optimizer}).
By considering a sequence of collapsing cuboids one finds that for any $\lambda>0, \gamma\geq 0$ the corresponding infimum is zero. If the Robin parameter $\beta$ is non-positive the situation would be reversed; the minimization problem is well posed and the corresponding supremum is infinite. 

Our main result is contained in the following theorem and describes the asymptotic behavior of cuboids that realize the supremum defining $M_{\gamma, d}$.
\begin{theorem}\label{thm: main theorem intro}
    Fix $d\geq 2, \gamma>0$, and $\beta>0$. Let $\{\lambda_j\}_{j\geq 1}$ be a sequence of positive numbers and $\{R_j\}_{j\geq 1}$ a sequence of cuboids in $\R^d$. Assume that $\lim_{j\to \infty}\lambda_j=\infty$, and, for each~$j\geq 1$, 
    \begin{equation*}
        |R_j|=1 \, , \quad \mbox{and}\quad \Tr(-\Delta_{R_j}^{\beta\sqrt{\lambda_j}}-\lambda_j)_\limminus^\gamma =M_{\gamma, d}(\lambda_j, \beta \sqrt{\lambda_j})\, .
    \end{equation*}
    There exists a $\beta^*>0$, depending only on $\gamma$ and $d$, such that
    \begin{enumerate}
        \item if $\beta<\beta^*$, then the sequence $\{R_j\}_{j\geq 1}$ has no converging subsequences.
        \item if $\beta>\beta^*$, then the sequence $\{R_j\}_{j\geq 1}$ converges to the unit cube as $j\to \infty$.
    \end{enumerate}
\end{theorem}
In Theorem~\ref{thm: geometric convergence general} we provide a slightly more general version of Theorem~\ref{thm: main theorem intro} where we also give a characterization of the transition point $\beta^*$ in terms of $\gamma$ and $d$.

Shape optimization problems in spectral theory have a long history, for an extensive overview of this classic topic we refer to \cite{Henrot_06, Henrot_17} and references therein. Asymptotic problems in spectral shape optimization have been the topic of a number of papers in recent years. The Dirichlet and Neumann analogues of the shape optimization among cuboids in Theorem~\ref{thm: main theorem intro} were obtained in \cite{GittinsLarson_17} in arbitrary dimensions, after initially having been solved in low dimensions in \cite{AntunesFreitas_13, vdBergBucurGittins_16, vdBergGittins_17}. These results also cover the case $\gamma=0$, which corresponds to optimizing individual eigenvalues. For Riesz means and optimizing among all convex sets of fixed measure, analogous results were obtained in \cite{Larson_JST19, FrankLarson_Crelle20, FrankLarson_CPAM26} for certain ranges of the parameter $\gamma>0$. In these results the limit of optimizing sets is given by the ball, whenever such a limit exists. In fact, the strategy we follow here is very close to that developed in \cite{FrankLarson_CPAM26}. For further results pertaining to asymptotic spectral shape optimization we refer the reader to \cite{BucurFreitas_13, vdBerg_15, AntunesFreitas_16, Freitas_17, LarsonAFM, FreitasKennedy_19, Lagace20, BuosoFreitas_20, Freitas_etal_21, Farrington_25}. 

Most of the existing results on the topic of asymptotic spectral shape optimization pertain to Dirichlet or Neumann Laplace operators. Two exceptions to this are \cite{AntunesFreitasKennedy_13, FreitasKennedy_19} which concerns minimizing $\lambda_k(-\Delta_\Omega^\beta)$ for a fixed Robin parameter $\beta>0$ in the limit as the eigenvalue index $k$ tends to infinity. In \cite{AntunesFreitasKennedy_13}, the authors show that $\inf_{\Omega: |\Omega|=1} \lambda_k(-\Delta_\Omega^\beta) \lesssim k^{1/d}$ and thus the asymptotic behavior of the minimal eigenvalues differs from that suggested by Weyl's law. Furthermore, the authors conjecture that there exists a $\beta_k\sim k^{1/d}$ so that for $\beta <\beta_k$ the optimizer of the $k$-th eigenvalue is given exactly by $k$ disjoint balls of equal size. The main results in~\cite{FreitasKennedy_19} concern the corresponding problem when the optimization takes place among either rectangles or disjoint unions of rectangles of fixed total area. Their results show that sequences of minimizers for these problems do not converge to any set as $k \to \infty$. In fact, among disjoint unions of rectangles they show that the optimum for $k$ sufficiently large, depending only on $\beta$, consist of $k$ squares of equal size, thus confirming the analogue of the conjecture in \cite{AntunesFreitasKennedy_13} in this restricted geometric setting. The purpose of the current paper is to show that the nature of these problems changes drastically in the regime where the Robin parameter goes to infinity as a multiple of $\sqrt{\lambda}$.

\subsection{Two-term asymptotics for Robin Riesz means}
The two most crucial ingredients in the proof of Theorem~\ref{thm: main theorem intro} are the validity of a two-term Weyl law for $\Tr(-\Delta_{R}^{\beta \sqrt{\lambda}}-\lambda)_\limminus^\gamma$ and uniform inequalities for these Riesz means. In this and the following subsection, we introduce these ingredients and explain how they relate to our main result. In order to state the results, we need some additional notation. 

For $d\geq 0, \gamma \geq 0$ define
\begin{equation*}
    L_{\gamma, d}^{\rm sc} := \frac{\Gamma(1+\gamma)}{(4\pi)^{d/2}\Gamma(1+\gamma+d/2)}\, , 
\end{equation*}
and for $\beta \in (0, \infty)$ let
\begin{equation*}
    L_{\gamma, d}(\beta) := L_{\gamma, d}^{\rm sc}\biggl(\frac{4}{\pi}\int_0^1(1-s^2)^{\gamma+d/2}\frac{\beta}{\beta^2+s^2}\, ds -1\biggr)\, .
\end{equation*}
As will be shown below (see Lemma \ref{lem: properties of L}) the function $\beta \mapsto L_{\gamma, d}(\beta)$ is smooth, decreasing, and satisfies
\begin{equation*}
    \lim_{\beta \to 0^\limplus} L_{\gamma, d}(\beta) = L_{\gamma, d}^{\rm sc} \quad \mbox{and} \quad \lim_{\beta \to \infty} L_{\gamma, d}(\beta) = -L_{\gamma, d}^{\rm sc}\, .
\end{equation*}
In particular, the equation $L_{\gamma, d}(\beta)=0$ has a unique solution which we denote by $\beta_W(\gamma, d)$.

The two-term asymptotic expansion for Riesz means that we shall prove can now be stated as follows. 
\begin{theorem}\label{thm: sc two-term Weyl intro}
    Let $d\geq 1, \gamma>0$, and $\beta>0$. If $R \subset \R^d$ is a cuboid, then
    \begin{align*}
        \Tr(-\Delta_{R}^{\beta\sqrt{\lambda}}-\lambda)_\limminus^\gamma = 
        L_{\gamma, d}^{\rm sc}|R|\lambda^{\gamma+d/2}+\frac{1}{4}L_{\gamma, d-1}(\beta) \Haus^{d-1}(\partial R)\lambda^{\gamma+(d-1)/2}+ o(\lambda^{\gamma+(d-1)/2})
    \end{align*}
    as $\lambda \to \infty$.
\end{theorem}
In proving Theorem~\ref{thm: main theorem intro} we shall need a more precise statement which provides explicit control of the little-$o$ remainder term in terms of the geometric properties of the cuboid $R$. A precise statement appears in Section~\ref{sec: d-dim spectral asymptotics}.

For $\gamma=1$, the analogue two-term expansions for $\Tr(-\Delta_{\Omega}^{\beta\sqrt{\lambda}}-\lambda)_\limminus$ were obtained by Frank and Geisinger in~\cite{FrankGeisinger_12} for any $\Omega \subset \R^d$ with $C^1$-regular boundary. To our knowledge Theorem~\ref{thm: sc two-term Weyl intro} is the first result where such an expansion is obtained for $\gamma <1$, although only for the class of cuboids. The result in \cite{FrankGeisinger_12} also covers Laplace operators with Robin boundary condition parametrized by a function $\beta \colon \partial\Omega \to \R$ which is not necessarily positive and satisfies some modest continuity assumptions. When the Laplace operator is independent of the spectral parameter, more is known about Riesz means asymptotics. For the Dirichlet and Neumann Laplace operators corresponding two-term asymptotics in arbitrary Lipschitz sets and any $\gamma>0$ were recently obtained in \cite{FrankLarson_Inventiones25}. In this setting $L_{\gamma, d-1}(\beta)$ should be replaced by $L_{\gamma, d-1}^{\rm sc}$ in the Neumann case and $-L_{\gamma, d-1}^{\rm sc}$ for Dirichlet. In \cite{FrankLarson_FAA25}, the corresponding result was extended to the case of a fixed Robin Laplacian under mild conditions on the function parameterizing the boundary condition, in this case the second term matches that obtained for the Neumann Laplacian. For particular geometries (for instance, rectangles or cuboids) two-term spectral asymptotics for the case $\gamma=0$ go back as far as to Weyl himself. In larger generality, such two-term asymptotics were famously obtained by Ivrii~\cite{Ivrii_80} under appropriate assumptions on the regularity of the underlying set $\Omega$, the function parameterizing the Robin boundary condition, and the billiard dynamics in $\Omega$. We conjecture that analogues of the asymptotics in Theorem~\ref{thm: sc two-term Weyl intro} extend to the case $\gamma=0$ and also to more general geometries and Robin boundary conditions, but to our knowledge this is unknown.

A heuristic argument based only on Theorem~\ref{thm: sc two-term Weyl intro} suggests that if one wants to maximize $\Tr(-\Delta_{R}^{\beta\sqrt{\lambda}}-\lambda)_\limminus^\gamma$ for large $\lambda$ among cuboids $R$ of fixed measure, then, as the first order term is always the same by the volume constraint, one should choose $R$ to make the second order term, $L_{\gamma, d-1}(\beta)\Haus^{d-1}(\partial R)$, large. If $L_{\gamma, d-1}(\beta)>0$ it becomes unreasonable to expect maximizers to converge since it appears favorable to have as large a perimeter as possible. However, if instead $L_{\gamma, d-1}(\beta)<0$, then it appears to be favorable to asymptotically minimize perimeter and one expects convergence to the unit cube, since this is the unique solution of the isoperimetric problem among cuboids. Consequently, it is tempting to conjecture that the transition point $\beta^*$ in Theorem~\ref{thm: main theorem intro} should coincide with $\beta_W(\gamma, d-1)$ as this is the point at which $L_{\gamma, d-1}(\beta)$ changes sign. However, this expectation based on heuristics from two-term asymptotics turns out to be too naive. Indeed, we shall show in Section~\ref{sec: the critical beta are different} that, in general, the transition point $\beta^*$ does not coincide with $\beta_W(\gamma, d-1)$. The failure of the heuristic argument is a phenomenon which we find interesting and to our knowledge has not previously been observed. In~\cite{Freitas_etal_21, FrankLarson_CPAM26}, corresponding heuristics for optimization problems in the setting of Dirichlet and Neumann Laplace operators are shown to fail, but only under the assumption that certain inequalities conjectured by P\'olya are false.

\subsection{Semiclassical inequalities for Robin Riesz means} In order to explain what actually determines the transition point in Theorem~\ref{thm: main theorem intro}, it is necessary to discuss the second key ingredient in our proof. Namely, uniform upper bounds for Robin Riesz means with Robin parameter growing with $\sqrt{\lambda}$.

\begin{theorem}\label{thm: Robin BLY intro}
    Let $d\geq 1, \gamma> 0$. There exists a $\beta(\gamma, d)>0$ such that for $\beta\geq\beta(\gamma, d)$, any cuboid $R \subset \R^d$, and $\lambda \geq 0$, 
    \begin{equation*}
        \Tr(-\Delta_{R}^{\beta\sqrt{\lambda}}-\lambda)_\limminus^\gamma \leq L_{\gamma, d}^{\rm sc}|R|\lambda^{\gamma+d/2}\, .
    \end{equation*}
    Moreover, if $\beta>\beta(\gamma, d)$ then the inequality is strict for any cuboid $R$ and $\lambda>0$.
\end{theorem}

For Riesz means of the Dirichlet Laplacian $-\Delta_\Omega^{\rm D}$, bounds of this form were obtained independently by Berezin~\cite{Berezin} and Li--Yau \cite{LiYau_83} for any open $\Omega\subset \R^d$ of finite measure as soon as $\gamma \geq 1$. Also for $\gamma \geq 1$, Kr\"oger~\cite{Kroger} proved that the reverse inequality holds for Riesz means of the Neumann Laplacian $-\Delta_\Omega^{\rm N}$. A well-known, and unresolved, conjecture of P\'olya~\cite{Polya1954} posits that these inequalities remain valid for $\gamma=0$ and any $\Omega$. For domains which tile $\R^d$, and thus in particular for cuboids, P\'olya himself proved the conjecture~\cite{Polya1961}. 

As far as we are aware, Theorem~\ref{thm: Robin BLY intro} is the first result concerning such uniform inequalities when the Robin parameter is coupled to the spectral parameter. By the monotonicity of Riesz means with respect to the Robin parameter, the validity of these inequalities for a fixed $\beta>0$ imply their validity for all larger values of $\beta$, and in particular the validity for the Dirichlet Laplacian. We believe that the inequalities of Theorem~\ref{thm: Robin BLY intro} extend to much larger classes of sets, but in what geometric generality and what range of $\gamma$ they remain valid is far from clear. The most ambitious conjecture would be that these inequalities remain valid for all open $\Omega \subset \R^d, d\geq 2, $ of finite measure and all $\gamma\geq 0$ as long as $\beta$ is sufficiently large. Note that independently of $\beta$ the stated bounds fail when $d=1$ and $\gamma=0$, in contrast to the case of the Dirichlet Laplace operator.

The two-term asymptotics in Theorem~\ref{thm: sc two-term Weyl intro} imply that $\beta(\gamma, d)\geq \beta_W(\gamma, d-1)$. Furthermore, we shall show below that it is the numbers $\beta(\gamma, d)$ from Theorem~\ref{thm: Robin BLY intro} that determine the transition point $\beta^*$ in Theorem~\ref{thm: main theorem intro} and in fact, $\beta^* = \beta(\gamma+1/2, d-1)$. This might appear a bit surprising, however, this naturally emerges from the way that the uniform inequalities in Theorem~\ref{thm: Robin BLY intro} enter the proof of our main result. Specifically, they enter in order to exclude that sequences of maximizers collapse to a lower dimensional cuboid after appropriate rescaling. In the relevant collapsing regime, the collapse results in an effective Robin spectral problem where the dimension is reduced and the order of Riesz mean increased. A similar `dimension drop' is observed in \cite{FrankLarson_CPAM26} in the setting of Dirichlet and Neumann Laplace operators on convex sets. 

From the definition of $L_{\gamma, d}(\beta)$ it is clear that $\beta_W(\gamma, d)$ only depends on the parameters $\gamma$ and $d$ through the quantity $\gamma+d/2$. Thus, the claim that $\beta^*$ differs from the point where the second term in the asymptotics changes sign, i.e.~$\beta(\gamma+1/2, d-1) >\beta_W(\gamma, d-1)$, is equivalent to $\beta(\gamma+1/2, d-1)>\beta_W(\gamma+1/2, d-2)$. In other words, the transition point $\beta^*$ is correctly predicted by the heuristic argument based on two-term asymptotics if and only if the inequality in Theorem~\ref{thm: Robin BLY intro} is valid for all $\beta\geq \beta_W(\gamma, d-1)$, i.e.\ as soon as the inequality does not fail due to the sign of the second term in the two-term asymptotics.

\subsection{Outline of the strategy and structure of the paper}
We end this introduction by outlining the structure of the paper and our overall strategy. As mentioned earlier, the strategy runs parallel to that developed in \cite{FrankLarson_CPAM26}. An idea introduced in that paper is to divide the analysis into two regimes: (i) a \emph{semiclassical regime} where uniform two-term spectral asymptotics are valid and 
(ii) a \emph{collapsing regime} in which the shortest side length of a sequence of cuboids is not larger than the natural wavelength $1/\sqrt{\lambda}$ as $\lambda \to \infty$. In this second regime, semiclassical approximations are not valid, but, as was observed in \cite{FrankLarson_CPAM26}, one can still understand the limiting behavior as a partial semiclassical limit (see also \cite{CarlenFrankLarson_LMP25}). The conclusion of Theorem~\ref{thm: main theorem intro} is then obtained by analyzing which of the two regimes is favorable for sequences of cuboids that asymptotically maximize the relevant Riesz means.

A large part of our results rests on understanding the Robin eigenvalue problem on intervals and lifting the analysis to cuboids by using their product structure. As such, a key element of our analysis is to understand spectral properties of the one-dimensional Robin Laplacian. The required analysis of the one-dimensional problem is the topic of Section~\ref{sec: 1D problem}. The main result is a two-term asymptotic expansion of Riesz means of the Robin problem on an interval with a quantitatively controlled remainder estimate. In Section~\ref{sec: d-dim spectral asymptotics}, the results from Section~\ref{sec: 1D problem} are used to prove a version of Theorem~\ref{thm: sc two-term Weyl intro} which is valid in the semiclassical regime explained above. The topic of Section~\ref{sec: collapsing cuboids} is to study the asymptotically collapsing regime. In Section~\ref{sec: uniform inequalities} we prove Theorem~\ref{thm: Robin BLY intro}. With Theorems~\ref{thm: sc two-term Weyl intro} and~\ref{thm: Robin BLY intro} in hand we turn in Section~\ref{sec: shape optimization} to the applications of these and analyze the asymptotic behavior of cuboids that maximize $M_{\gamma, d}$. We close the paper by showing that for some $\gamma$ and $d$, the transition point $\beta^*$ in Theorem~\ref{thm: main theorem intro} indeed differs from the transition point $\beta_W(\gamma, d-1)$ where the second term in the Weyl asymptotics changes sign. We also show numerical evidence that supports the conjecture that they never coincide.

\subsection*{Acknowledgment}
Financial support through the Swedish Research Council grant no.~2023-03985 (S.~L.) is acknowledged. S. L. would like to thank Timo Weidl for his hospitality during a visit to the University of Stuttgart during which this project was initiated. S.~L. would also like to thank the Isaac Newton Institute for Mathematical Sciences, Cambridge, for support and hospitality during the programme `Geometric spectral theory and applications', where work on this paper was undertaken. The programme was supported by EPSRC grant EP/Z000580/1. 

\section{One-dimensional considerations}
\label{sec: 1D problem}

In this section we consider the one-dimensional Robin eigenvalue problem. Specifically, for $l, \beta >0$ we are interested in the eigenvalue problem for $-\Delta_{(0, l)}^\beta$, which corresponds to 
\begin{equation}\label{eq: 1D eigenvalue equation}
  \begin{cases}
    -u''(x) = \lambda u(x) & \mbox{ for } x \in (0, l)\, , \\
    -u'(0)+ \beta u(0) =0\\
    u'(l)+\beta u(l) =0\, .
  \end{cases}
\end{equation}
In Section~\ref{sec: basic properties 1D problem} we derive transcendental equations for the eigenvalues in terms of $k, l, \beta$ which will allow us to deduce some basic properties of the eigenvalues. In Section~\ref{sec: 1D Riesz means} we use our conclusions for the eigenvalues to obtain various bounds for their Riesz means.

\subsection{Basic properties of Robin eigenvalues on an interval}
\label{sec: basic properties 1D problem}
From the scaling of equation \eqref{eq: 1D eigenvalue equation}, we see that $\lambda_k(-\Delta_{(0, l)}^\beta) = l^{-2}\lambda_k(-\Delta_{(0, 1)}^{\beta l})$ for each $k\geq 1$. As such, we can assume without loss of generality that $l=1$.

If $u$ is a non-trivial solution of \eqref{eq: 1D eigenvalue equation} with $l=1$ then there exist $a, b \in \R$ so that $u(x) = a \sin(\sqrt{\lambda}x)+ b \cos(\sqrt{\lambda}x)$. In terms of $a, b$ the boundary conditions read
\begin{equation*}
  -a \sqrt{\lambda}+b\beta = 0 \quad \mbox{and} \quad \sqrt{\lambda}(a\cos(\sqrt{\lambda})-b \sin(\sqrt{\lambda}))+\beta (a\sin(\sqrt{\lambda})+b\cos(\sqrt{\lambda}))=0\, .
\end{equation*}
From the first equation we see that neither $a$ nor $b$ can be zero (unless $\beta =0$ which we exclude), and the parameters are related by $b = \frac{a}{\beta}\sqrt{\lambda}$. Upon inserting this relation into the second equation and dividing by $a$, we find that $\lambda$ is an eigenvalue if and only if
\begin{equation}\label{eq: betalambda equation}
  2\sqrt{\lambda}\cos(\sqrt{\lambda})+\Bigl(\beta-\frac{\lambda}{\beta}\Bigr)\sin(\sqrt{\lambda})=0\, .
\end{equation}
Each positive solution of this equation yields a simple eigenvalue. Solving \eqref{eq: betalambda equation} for $\beta$ we find two possible solutions:
\begin{equation}\label{eq: transcendental equations}
  \beta = \sqrt{\lambda}\tan(\sqrt{\lambda}/2) \quad \mbox{and} \quad \beta = - \sqrt{\lambda}\tan(\sqrt{\lambda}/2)^{-1}\, .
\end{equation}

By the variational principle, the eigenvalues are monotone increasing with respect to $\beta$ and thus for any $\beta >0$ and each $k\geq 1$, 
\begin{equation*}
  \lambda_k(-\Delta_{(0, 1)}^{\rm N})=\pi^2 (k-1)^2  <\lambda_k(-\Delta_{(0, 1)}^\beta) < \lim_{\beta \to \infty}\lambda_k(-\Delta_{(0, 1)}^\beta)=\pi^2 k^2 = \lambda_k(-\Delta_{(0, 1)}^{\rm D})\, .
\end{equation*}
For $\lambda \in (\pi^2 (k-1)^2, \pi^2k^2)$, $k\in \N$, only one of the solutions $\beta$ in \eqref{eq: transcendental equations} is positive. Consequently, for $\beta>0$ we have that
\begin{equation}\label{eq: implicit equations odd/even}
\begin{aligned}
  \beta &= \begin{cases}
      \sqrt{\lambda_k(-\Delta_{(0, 1)}^\beta)}\tan\Bigl(\sqrt{\lambda_k(-\Delta_{(0, 1)}^\beta)}/2\Bigr) & \mbox{if } k \mbox{ is odd}\, , \\
  -\sqrt{\lambda_k(-\Delta_{(0, 1)}^\beta)}\tan\Bigl(\sqrt{\lambda_k(-\Delta_{(0, 1)}^\beta)}/2\Bigr)^{-1} & \mbox{if } k \mbox{ is even}\, .
  \end{cases} 
\end{aligned}
\end{equation}
From \eqref{eq: implicit equations odd/even} and the implicit function theorem it follows that $\beta \mapsto \lambda_k(-\Delta_{(0, 1)}^\beta)$ is a smooth function on $(0, \infty)$. In particular, by implicit differentiation and Taylor's theorem, one obtains the following lemmas which will be of use later on.
\begin{lemma}\label{lem: lambda1 asymptotics}
    As $\beta \to 0^\limplus$, we have
    \begin{equation*}
        \lambda_1(-\Delta_{(0, 1)}^{\beta}) = 2\beta +O(\beta^2)\, .
    \end{equation*}
\end{lemma}

\begin{proof}
    Define
    \begin{equation*}
        F_1: (0, \infty)\times (0, \pi^2)\to \R, \quad (\beta, \lambda)\mapsto \beta- \sqrt{\lambda}\tan(\sqrt{\lambda}/2)\, , 
    \end{equation*}
    so that $\lambda=\lambda_1(-\Delta_{(0, 1)}^\beta)$ is the unique solution to the equation $F_1(\beta, \lambda)=0$ for any $\beta>0$.
    Since $F_1$ is smooth and
    \begin{equation*}
        \frac{d}{d\lambda}F_1(\beta, \lambda) = -\frac{\sqrt{\lambda}+\sin(\sqrt{\lambda})}{2\sqrt{\lambda}(1+\cos(\sqrt{\lambda}))}<0
    \end{equation*}
    for all $(\beta, \lambda)\in (0, \infty)\times (0, \pi^2)$ the implicit function theorem implies that $\beta \mapsto \lambda_1(-\Delta_{(0, 1)}^\beta)$ is a smooth function on $(0, \infty)$.

    To simplify notation let $\lambda(\beta)= \lambda_1(-\Delta_{(0, 1)}^\beta)$. By implicitly differentiating the equation $F_1(\beta, \lambda(\beta))=0$ twice, 
    \begin{align*}
        \lambda'(\beta) &= \frac{2\sqrt{\lambda(\beta)}(1+\cos(\sqrt{\lambda(\beta})))}{\sqrt{\lambda(\beta)}+\sin(\sqrt{\lambda(\beta)})}\, , \\
        \lambda''(\beta) &= -\frac{2(1+\cos(\sqrt{\lambda(\beta)}))^2\bigl[\lambda(\beta)\tan(\sqrt{\lambda(\beta)}/2)+\sqrt{\lambda(\beta)}-\sin(\sqrt{\lambda(\beta)})\bigr]}{(\sqrt{\lambda(\beta)}+\sin(\sqrt{\lambda(\beta)}))^3}\, .
    \end{align*}
    As $\lim_{\beta \to 0^\limplus}\lambda(\beta) =0$ it follows that
    \begin{equation*}
        \lim_{\beta \to 0^\limplus}\lambda'(\beta) = 2\, .
    \end{equation*}
    Moreover, using that $\tan(x) = x+O(x^3)$ and $\sin(x) = x+ O(x^3)$ as $x \to 0$ one concludes that $\lambda''(\beta)$ extends continuously to $0$. Taylor's theorem thus yields the desired asymptotic expansion as $\beta \to 0^\limplus$.
\end{proof}

\begin{lemma}\label{lem: lambdak derivative bounds}
    For every $k\geq 1$ and $\beta >0$, 
    \begin{equation*}
        0\leq \frac{d}{d\beta}\lambda_k(-\Delta_{(0, 1)}^{\beta})\leq 4\, . 
    \end{equation*}
\end{lemma}
\begin{proof}
    The lower bound is a direct consequence of the variational principle. Therefore, it remains to prove the upper bound. To simplify notation, let $\lambda_k(\beta)=\lambda_k(-\Delta_{(0, 1)}^\beta)$ throughout this proof. 
    
    We define for $k\in\mathbb{N}$ the functions $F_k\colon (0, \infty)\times ((k-1)^2\pi^2, k^2\pi^2)\to \R$ by
    \begin{align*}
        F_k(\beta, \lambda) = \begin{cases}\beta-\sqrt{\lambda}\tan(\sqrt{\lambda}/2) & \mbox{if }k \mbox{ is odd\, , }\\
         \beta+\sqrt{\lambda}\tan(\sqrt{\lambda}/2)^{-1} & \mbox{if }k \mbox{ is even\, .}
         \end{cases}
    \end{align*}
    Then $\lambda_k(\beta)$ is the unique solution of the equation $F_k(\beta, \lambda)=0$.

    For $k$ odd we have that 
    \begin{equation*}
        \frac{d}{d\lambda}F_k(\beta, \lambda) = - \frac{\sqrt{\lambda}+\sin(\sqrt{\lambda})}{2\sqrt{\lambda}(1+\cos(\sqrt{\lambda}))}<0 \quad \mbox{for all }\lambda \in (\pi^2(k-1)^2, \pi^2k^2)\, , 
    \end{equation*}
    and similarly, for $k$ even it holds that
    \begin{equation*}
        \frac{d}{d\lambda}F_k(\beta, \lambda) = - \frac{\sqrt{\lambda}-\sin(\sqrt{\lambda})}{2\sqrt{\lambda}(1-\cos(\sqrt{\lambda}))}<0 \quad \mbox{for all }\lambda \in (\pi^2(k-1)^2, \pi^2k^2)\, .
    \end{equation*}
    Thus, the use of the implicit function theorem is justified. The smoothness of the functions $F_k$ implies that each mapping $\beta \mapsto \lambda_k(\beta)$ is a smooth function.
    
    By implicit differentiation of the equation $F_k(\beta, \lambda_k(\beta))=0$, we find that
    \begin{align*}
        \lambda_k'(\beta) = \begin{cases}
            \dfrac{2\sqrt{\lambda_k(\beta)}(1+\cos(\sqrt{\lambda_k(\beta})))}{\sqrt{\lambda_k(\beta)}+\sin(\sqrt{\lambda_k(\beta)})} & \mbox{for }k \mbox{ odd}\, , \\[2ex]
        \dfrac{2\sqrt{\lambda_k(\beta)}(1-\cos(\sqrt{\lambda_k(\beta})))}{\sqrt{\lambda_k(\beta)}-\sin(\sqrt{\lambda_k(\beta)})} & \mbox{for }k \mbox{ even}\, .
        \end{cases} 
    \end{align*}
    Using the fact that $\lambda_k(\beta) \in (\pi^2(k-1)^2, \pi^2k^2)$ and $(-1)^{k+1}\sin(x) = |\sin(x)|$ for $x\in[\pi(k-1), \pi k]$, we conclude that
    \begin{equation*}
        \lambda_k'(\beta) = \frac{2(1+(-1)^{k+1}\cos(\sqrt{\lambda_k(\beta)}))}{1+(-1)^{k+1}{\sin(\sqrt{\lambda_k(\beta)})}/{\sqrt{\lambda_k(\beta)}}} = \frac{2(1+(-1)^{k+1}\cos(\sqrt{\lambda_k(\beta)}))}{1+|{\sin(\sqrt{\lambda_k(\beta)})}/{\sqrt{\lambda_k(\beta)}}|} \leq 4\, .
    \end{equation*}
    This completes the proof of the lemma.
\end{proof}

Finally, in the next subsection, we shall rely heavily on the following two-sided bounds for the one-dimensional Robin eigenvalues.
\begin{lemma}\label{lem: arctan bounds lambdak}
    For each $k\geq 1$ and any $\beta>0$, 
    \begin{equation*}
        \Bigl(\pi k -2\arctan\Bigl(\frac{\pi k}{\beta}\Bigr)\Bigr)^2 < \lambda_k(-\Delta_{(0, 1)}^\beta)< \Bigl(\pi k- 2\arctan\Bigr(\frac{\pi (k-1)}{\beta}\Bigr)\Bigr)^2\, .
    \end{equation*}
\end{lemma}

\begin{proof}
    Since $\arctan(x)< \pi/2$ and $k\geq 1$, the expressions in the squares in the upper and lower bounds are both non-negative. Consequently, the claimed bounds are equivalent to that
    \begin{equation*}
        2\arctan\Bigl(\frac{\pi (k-1)}{\beta}\Bigr) < \pi k -\sqrt{\lambda_k(-\Delta_{(0, 1)}^\beta)}< 2\arctan\Bigl(\frac{\pi k}{\beta}\Bigr)\, .
    \end{equation*}

    To prove this inequality, we begin by rewriting the relations \eqref{eq: implicit equations odd/even} somewhat. Define $\delta_k(\beta) := \pi k -\sqrt{\lambda_k(-\Delta_{(0, 1)}^\beta)}$. Since $\lambda_k(-\Delta_{(0, 1)}^\beta) \in (\pi^2(k-1)^2, \pi^2k^2)$, we have that $\delta_k(\beta) \in (0, \pi)$ for each $k\geq 1$ and $\beta>0$. The relations in \eqref{eq: implicit equations odd/even} can be written in terms of $\delta_k$ as
    \begin{equation*}
        \beta = (\pi k-\delta_k(\beta))\tan(\delta_k(\beta)/2)^{-1}
    \end{equation*}
    which is equivalent to
    \begin{equation}\label{eq: implicit equation delta}
        \delta_k(\beta) = 2\arctan\Bigl(\frac{\pi k-\delta_k(\beta)}{\beta}\Bigr)\, .
    \end{equation}
    The desired bound follows from \eqref{eq: implicit equation delta} combined with the monotonicity of $x\mapsto \arctan(x)$ and that $\delta_k(\beta) \in (0, \pi)$.
\end{proof}

\subsection{Riesz means of Robin Laplace operators on an interval}
\label{sec: 1D Riesz means}

In this section we proceed to use the bounds of Lemma~\ref{lem: arctan bounds lambdak} to deduce bounds for Riesz means of $-\Delta_{(0, 1)}^\beta$. As we shall see, the bounds of Lemma~\ref{lem: arctan bounds lambdak} are sufficiently precise to imply two-term asymptotic expansions of $\Tr(-\Delta_{(0, 1)}^\beta-\lambda)_\limminus^\gamma$ as $\lambda \to \infty$ uniformly in the Robin parameter. When we study shape optimization problems later in this paper, such bounds will be crucial as we will then consider limits where $\lambda$ and $\beta$ tend to infinity simultaneously. 

The main result that we shall prove in this section is contained in the next theorem.
\begin{theorem}\label{thm: two-term Weyl 1D}
    Let $\gamma>0$ and $\kappa_{\gamma, 1}=\min\{\gamma, 1\}$. There exists a constant $C_\gamma>0$ so that, for any $\beta>0$ and any $ \lambda >0$, 
    \begin{align*}
        \biggl| \Tr(-\Delta_{(0, 1)}^{\beta}-\lambda)_\limminus^\gamma - L_{\gamma, 1}^{\rm sc}\lambda^{\gamma+1/2}- \frac{1}{2}L_{\gamma, 0}(\beta/\sqrt{\lambda}) \lambda^\gamma \biggr|\leq  C_\gamma \lambda^{\gamma-\kappa_{\gamma, 1}/2} \, .
    \end{align*}
\end{theorem}

The proof of Theorem~\ref{thm: two-term Weyl 1D} will occupy what remains of this section. We split the proof into a number of lemmas.
We first focus on proving the theorem for $\gamma  \leq 1$. The case $\gamma > 1$ is deduced by lifting $\gamma = 1$ with the Aizenman--Lieb identity \cite{AizenmanLieb}. Our method of proof for $\gamma \leq 1$ can be made to work also for $\gamma>1$, however the error estimates that we obtain in this manner are worse than those that we get via the Aizenman--Lieb argument.

To prove Theorem \ref{thm: two-term Weyl 1D}, we start from Lemma~\ref{lem: arctan bounds lambdak}. A direct consequence of Lemma~\ref{lem: arctan bounds lambdak} is that for any $\gamma \geq 0, \beta>0, $ and $\lambda \geq 0$ we have
\begin{align*}
    \sum_{k\geq 1}\Bigl(\lambda- \Bigl(\pi k-2\arctan\Bigl(\frac{\pi (k-1)}{\beta}\Bigr)\Bigr)^2\Bigr)_\limplus^\gamma
    &\leq
    \Tr(-\Delta_{(0, 1)}^\beta-\lambda)_\limminus^\gamma\\
    &\leq
    \sum_{k\geq 1}\Bigl(\lambda- \Bigl(\pi k-2\arctan\Bigl(\frac{\pi k}{\beta}\Bigr)\Bigr)^2\Bigr)_\limplus^\gamma\, .
\end{align*}
The next lemma quantifies that these upper and lower bounds are precise enough to deduce asymptotics as $\lambda \to \infty$ up to an acceptable remainder. 
\begin{lemma}\label{lem: unif bound Riesz arctan} 
    Let $\gamma \in (0, 1]$. There exists a constant $C_\gamma>0$ so that, for any $\beta>0$ and any $\lambda >0$, 
    \begin{equation*}
        \biggl|\sum_{k\geq 1}\biggl[\Bigl(\lambda- \Bigl(\pi k-2\arctan\Bigl(\frac{\pi k}{\beta}\Bigr)\Bigr)^2\Bigr)_\limplus^\gamma-\Bigl(\lambda- \Bigl(\pi k-2\arctan\Bigl(\frac{\pi (k-1)}{\beta}\Bigr)\Bigr)^2\Bigr)_\limplus^\gamma\biggr]\biggr| \leq C_\gamma \lambda^{\gamma/2}\, .
    \end{equation*}
\end{lemma}

As $\lambda^{\gamma/2}$ is of lower order than the second term in the two-term asymptotics claimed in Theorem~\ref{thm: two-term Weyl 1D}, the lemma implies that it suffices to prove asymptotics for the sum with the eigenvalues replaced by one of the expressions given in terms of $\arctan$.

\begin{proof}
  Since $[0, \infty)\ni y \mapsto \arctan(y)$ is increasing and satisfies $0<\arctan(y) <\pi/2$, it follows that each term in the sum is non-negative. Consequently, we can drop the the absolute values.

  Since $2\arctan(y)\in [0, \pi]$ for all $y\geq 0$ the expression
  $
    \lambda- (\pi k-2\arctan(y))^2
  $
  is non-negative if $k \leq \sqrt{\lambda}/\pi$ and non-positive if $k \geq \sqrt{\lambda}/\pi +1$. Therefore, any $k$ for which the corresponding term in the sum is non-zero satisfies $k \leq \sqrt{\lambda}/\pi +1$.

  Let $M\geq 2$ be a constant to be specified below. 

  We begin with proving the bound for $\lambda \leq 4\pi^2 M^2$.  As noted above the sum contains at most $\sqrt{\lambda}/\pi + 1$ non-zero terms. Each of these terms is bounded by $\lambda^\gamma$. Therefore, for all $\lambda \geq 0$ we have
   \begin{equation*}
    \sum_{k\geq 1}\Bigl(\lambda- \Bigl(\pi k-2\arctan\Bigl(\frac{\pi k}{\beta}\Bigr)\Bigr)^2\Bigr)_\limplus^\gamma-\sum_{k\geq 1}\Bigl(\lambda- \Bigl(\pi k-2\arctan\Bigl(\frac{\pi (k-1)}{\beta }\Bigr)\Bigr)^2\Bigr)_\limplus^\gamma \leq \lambda^\gamma\Bigl(\frac{\sqrt{\lambda}}\pi +1\Bigr)\, .
   \end{equation*}
   The claimed bound for $\lambda \leq 4\pi^2 M^2$ follows from noting that $\lambda^\gamma(\sqrt{\lambda}/\pi +1)\lesssim_{\gamma, M}\lambda^{\gamma/2}$ for all $\lambda \leq 4\pi^2 M^2$ \footnote{Throughout the paper, we use the notation $\lesssim / \gtrsim$ to mean that the left-hand side is bounded from above/below by a positive constant times the right-hand side, where the implied constant is independent of the relevant parameters. Subscripts are used to emphasize that the implicit constant only depends on the parameters that appear in the subscript.}. It thus remains to prove the bound for $\lambda > 4\pi^2 M^2$.
  
  For any $\sqrt{\lambda}/\pi-M \leq k \leq \sqrt{\lambda}/\pi+1$ we have that
  \begin{equation}\label{eq: bound Riesz summand large lambda}
  \begin{aligned}
    \lambda- (\pi k-2\arctan(y))^2
    &=\lambda- \pi^2 k^2+4k\pi \arctan(y)-4\arctan(y)^2\\
    &\leq \lambda - \pi^2 (\sqrt{\lambda}/\pi-M)^2+2(\sqrt{\lambda}/\pi+1)\pi^2\\
    &= 2\pi M\sqrt{\lambda}-\pi^2 M^2 +2\pi \sqrt{\lambda}+2\pi^2\\
    &\leq 4\pi M \sqrt{\lambda}
  \end{aligned}
  \end{equation}
  where the last inequality used that $M\geq 2$. 
  Consequently, for any $M\geq 2$ and $\lambda \geq 4\pi^2M^2$, 
  \begin{align}
    &\sum_{k\geq 1}\Bigl(\lambda- \Bigl(\pi k-2\arctan\Bigl(\frac{\pi k}{\beta}\Bigr)\Bigr)^2\Bigr)_\limplus^\gamma-\sum_{k\geq 1}\Bigl(\lambda- \Bigl(\pi k-2\arctan\Bigl(\frac{\pi (k-1)}{\beta}\Bigr)\Bigr)^2\Bigr)_\limplus^\gamma \notag\\ 
    &\ \leq
    \sum_{1 \leq k \leq \sqrt{\lambda}/\pi -M}\Biggl[\Bigl(\lambda- \Bigl(\pi k-2\arctan\Bigl(\frac{\pi k}{\beta}\Bigr)\Bigr)^2\Bigr)^\gamma-\Bigl(\lambda- \Bigl(\pi k-2\arctan\Bigl(\frac{\pi (k-1)}{\beta}\Bigr)\Bigr)^2\Bigr)^\gamma\Biggr] \notag\\
    &\qquad + (4\pi)^\gamma (M+1) M^{\gamma}\lambda^{\gamma/2}\, . \label{eq: first bound diff arctan sum}
\end{align}
To bound the terms in the remaining sum we note that
\begin{equation}\label{eq: diff of arctan as integral}
\begin{aligned}
    \Bigl(\lambda- \Bigl(\pi k-2&\arctan\Bigl(\frac{\pi k}{\beta}\Bigr)\Bigr)^2\Bigr)^\gamma-\Bigl(\lambda- \Bigl(\pi k-2\arctan\Bigl(\frac{\pi (k-1)}{\beta}\Bigr)\Bigr)^2\Bigr)^\gamma\\
    &= 
    \int_{k-1}^{k} \frac{d}{ds}\Bigl(\lambda- \Bigl(\pi k-2\arctan\Bigl(\frac{\pi s}{\beta}\Bigr)\Bigr)^2\Bigr)^\gamma\, ds\\
    &=
    4\pi \beta \gamma \int_{k-1}^{k} \frac{\bigl(\lambda- \bigl(\pi k-2\arctan\bigl(\frac{\pi s}{\beta }\bigr)\bigr)^2\bigr)^{\gamma-1}\bigl(\pi k-2\arctan\bigl(\frac{\pi s}{\beta}\bigr)\bigr)}{\beta^2 + \pi^2s^2}\, ds\, .
\end{aligned}
\end{equation}

Now, for any $1 \leq k \leq \sqrt{\lambda}/\pi-M$, 
\begin{align*}
    \lambda- (\pi k-2\arctan(y))^2
    &=\lambda- \pi^2 k^2+4k\pi \arctan(y)-4\arctan(y)^2\\
    &\geq \lambda - \pi^2 (\sqrt{\lambda}/\pi-M)^2-\pi^2\\
    &= 2\pi M\sqrt{\lambda}-\pi^2 M^2 -\pi^2\\
    &\geq \pi M \sqrt{\lambda}\, , 
\end{align*}
where the last step used that $M \geq 2$ and $\lambda \geq 4\pi^2M^2$.
Thus, we have for all $1 \leq k \leq \sqrt{\lambda}/\pi-M$ and $s \in [k-1, k]$, 
\begin{align*}
  \Bigl|\lambda- \Bigl(\pi k-2\arctan\Bigl(\frac{\pi s}{\beta}\Bigr)\Bigr)^2\Bigr| \geq \pi M \sqrt{\lambda}\quad \mbox{and}\quad
  \Bigl|\pi k- 2\arctan\Bigl(\frac{\pi s}{\beta}\Bigr)\Bigr| \leq \sqrt{\lambda}\, .
\end{align*}
Hence, 
\begin{equation*}
\begin{aligned}
    4\pi \beta \gamma &\int_{k-1}^{k} \frac{\bigl(\lambda- \bigl(\pi k-2\arctan\bigl(\frac{\pi s}{\beta }\bigr)\bigr)^2\bigr)^{\gamma-1}\bigl(\pi k-2\arctan\bigl(\frac{\pi s}{\beta}\bigr)\bigr)}{\beta^2 + \pi^2s^2}\, ds\\
    &\leq
    4\pi^{\gamma}\beta \gamma M^{\gamma-1}\lambda^{\gamma/2}  \int_{k-1}^{k} \frac{1}{\beta^2 + \pi^2s^2}\, ds\, .
  \end{aligned}
  \end{equation*}
Inserting this bound into \eqref{eq: diff of arctan as integral} and \eqref{eq: first bound diff arctan sum} yields
\begin{equation*}
\begin{aligned}
    \sum_{k\geq 1}\Bigl(\lambda- \Bigl(&\pi k-2\arctan\Bigl(\frac{\pi k}{\beta}\Bigr)\Bigr)^2\Bigr)_\limplus^\gamma-\Bigl(\lambda- \Bigl(\pi k-2\arctan\Bigl(\frac{\pi (k-1)}{\beta}\Bigr)\Bigr)^2\Bigr)_\limplus^\gamma\\
    &\leq 
    4\pi^{\gamma}\beta \gamma M^{\gamma-1}\lambda^{\gamma/2} \Biggl( \sum_{1 \leq k \leq \sqrt{\lambda}/\pi-M} \int_{k-1}^{k} \frac{1}{\beta^2 + \pi^2s^2}\, ds \Biggr)
    +2(4\pi)^\gamma M^{\gamma+1}\lambda^{\gamma/2}\\
    &\leq
    4\pi^{\gamma}\beta \gamma M^{\gamma-1}\lambda^{\gamma/2} \int_{0}^{\sqrt{\lambda}/\pi} \frac{1}{\beta^2 + \pi^2s^2}\, ds
    +(4\pi)^\gamma (M+1)M^{\gamma}\lambda^{\gamma/2}\\
    &=
    4\pi^{\gamma-1}\gamma M^{\gamma-1}\lambda^{\gamma/2} \arctan\Bigl(\frac{\sqrt{\lambda}}{\beta}\Bigr)
    +(4\pi)^\gamma (M+1)M^{\gamma}\lambda^{\gamma/2}\, .
\end{aligned}
\end{equation*}
Fixing $M \geq 2$ and using that $\arctan(y)\leq \pi/2$ for all $y$ completes the proof.
\end{proof}

The next lemma shows that the approximation of Riesz means of $-\Delta_{(0, 1)}^\beta$ deduced from Lemma~\ref{lem: arctan bounds lambdak} behaves essentially like the Riesz mean of $-\Delta_{(0, 1)}^{\rm D}$ plus an explicit sum.

\begin{lemma}\label{lem: Riesz arctan sum estimate 1}
    Let $\gamma \in (0, 1]$ and $M\geq 2$. There exists a constant $C_{\gamma, M}>0$ so that, for any $\beta>0 $ and $\lambda >0$, 
    \begin{align*}
        \biggl|\sum_{k\geq 1}\Bigl(\lambda- \Bigl(\pi k-&2\arctan\Bigl(\frac{\pi k}{\beta}\Bigr)\Bigr)^2\Bigr)_\limplus^\gamma-\Tr(-\Delta_{(0, 1)}^{\rm D}-\lambda)_\limminus^\gamma\\
        &- 4\pi \gamma \sum_{1\leq k \leq \sqrt{\lambda}/\pi-M}k(\lambda-\pi^2k^2)^{\gamma-1}\arctan\Bigl(\frac{\pi k}{\beta}\Bigr)\biggr|\leq C_{\gamma, M} \lambda^{\gamma/2}\, .
    \end{align*}
\end{lemma}

\begin{proof} 
    As in the previous proof, the proof is divided into two cases depending on $\lambda$ and $M$. 
    
    \medskip
    \noindent\emph{Case 1:} $\lambda \geq 4\pi^2 M^2$.

  For any $\gamma \in \R, x\in [0, 1)$ we have
  \begin{equation*}
    (1-x)^\gamma = 1+\sum_{j\geq 1}\frac{\prod_{i=0}^{j-1}(i-\gamma)}{j!}x^j\, .
  \end{equation*}
  For each $k \in [1, \sqrt{\lambda}/\pi-M]$ we thus have
  {\allowdisplaybreaks
  \begin{align*}
    &\Bigl(\lambda- \Bigl(\pi k-2\arctan\Bigl(\frac{\pi k}{\beta}\Bigr)\Bigr)^2\Bigr)^\gamma\\
    &=\lambda^\gamma\Bigl(1- \frac{1}{\lambda}\Bigl(\pi k-2\arctan\Bigl(\frac{\pi k}{\beta}\Bigr)\Bigr)^2\Bigr)^\gamma\\
    &=
    \lambda^\gamma\biggl(1+\sum_{j=1}^\infty\frac{\prod_{i=0}^{j-1}(i-\gamma)}{j!}\Bigl(\frac{\pi^{2}k^{2}}{\lambda}\Bigr)^j\Bigl(1-\frac{2}{\pi k}\arctan\Bigl(\frac{\pi k}{\beta}\Bigr)\Bigr)^{2j}\biggr)\\
     &=
    \lambda^\gamma\biggl(1+\sum_{j=1}^\infty\frac{\prod_{i=0}^{j-1}(i-\gamma)}{j!}\Bigl(\frac{\pi^{2}k^{2}}{\lambda}\Bigr)^j\biggr) -\frac{4\lambda^\gamma}{\pi k}\sum_{j=1}^\infty\frac{\prod_{i=0}^{j-1}(i-\gamma)}{(j-1)!}\Bigl(\frac{\pi^{2}k^{2}}{\lambda}\Bigr)^j\arctan\Bigl(\frac{\pi k}{\beta}\Bigr)\\
    &\quad +
   \lambda^\gamma \sum_{j=1}^\infty\frac{\prod_{i=0}^{j-1}(i-\gamma)}{j!}\Bigl(\frac{\pi^{2}k^{2}}{\lambda}\Bigr)^j\biggl(\Bigl(1-\frac{2}{\pi k}\arctan\Bigl(\frac{\pi k}{\beta}\Bigr)\Bigr)^{2j}-1+\frac{4j}{\pi k}\arctan\Bigl(\frac{\pi k}{\beta}\Bigr)\biggr)\\
    &=
    (\lambda-\pi^2k^2)^\gamma +4\pi \gamma k (\lambda-\pi^2k^2)^{\gamma-1}\arctan\Bigl(\frac{\pi k}{\beta}\Bigr)\\
    &\quad +
   \lambda^\gamma\sum_{j=1}^\infty\frac{\prod_{i=0}^{j-1}(i-\gamma)}{j!}\lambda^{-j}\pi^{2j}k^{2j}\biggl(\Bigl(1-\frac{2}{\pi k}\arctan\Bigl(\frac{\pi k}{\beta}\Bigr)\Bigr)^{2j}-1+\frac{4j}{\pi k}\arctan\Bigl(\frac{\pi k}{\beta}\Bigr)\biggr)\, .
  \end{align*}}
  Therefore, we obtain
  \begin{equation}\label{eq: lem26-divide}
  \begin{aligned}
       \biggl|&\sum_{k\geq 1}\Bigl(\lambda- \Bigl(\pi k -2\arctan\Bigl(\frac{ \pi k}{\beta}\Bigr)\Bigr)^2\Bigr)^\gamma_\limplus - \sum_{k\geq 1}(\lambda-\pi^2k^2)^\gamma_\limplus\\
    &\quad - 4\pi \gamma\sum_{1\leq k \leq \sqrt{\lambda}/\pi-M}k (\lambda-\pi^2k^2)_\limplus^{\gamma-1}\arctan\Bigl(\frac{\pi k}{\beta}\Bigr)\biggr|\\
    &\leq \sum_{k>\sqrt{\lambda}/\pi -M}\Bigl(\lambda- \Bigl(\pi k-2\arctan\Bigl(\frac{\pi k}{\beta}\Bigr)\Bigr)^2\Bigr)_\limplus^\gamma + \sum_{k>\sqrt{\lambda}/\pi-M}(\lambda-\pi^2k^2)_\limplus^\gamma\\
    &+
   \biggl|\lambda^\gamma\!\!\!\!\sum_{1\leq k\leq \sqrt{\lambda}/\pi-M}\sum_{j=1}^\infty\frac{\prod_{i=0}^{j-1}(i-\gamma)}{j!}\Bigl(\frac{\pi^2k^2}{\lambda}\Bigr)^j\biggl(\Bigl(1-\frac{2\arctan(\frac{\pi k}\beta)}{\pi k}\Bigr)^{2j}\!-1+\frac{4j\arctan(\frac{\pi k}\beta)}{\pi k}\biggr)\biggr|\, .
  \end{aligned}
  \end{equation}
  Since $\sum_{k\geq 1}(\lambda-\pi^2k^2)_\limplus^\gamma = \Tr(-\Delta_{(0, 1)}^{\rm D}-\lambda)_\limminus^\gamma$, the proof is completed if we can suitably bound the three terms on the right-hand side of \eqref{eq: lem26-divide}.

  For the first sum, recall from the proof of Lemma~\ref{lem: unif bound Riesz arctan} that any $k$ for which $\lambda-(\pi k -2\arctan(\pi k/\beta))^2>0$ satisfies $k \leq \sqrt{\lambda}/\pi +1$ and by \eqref{eq: bound Riesz summand large lambda}, it holds 
    \begin{equation*}
        \lambda - (\pi k-2\arctan(y))^2 \leq 4\pi M \sqrt{\lambda}\, , 
    \end{equation*}
    for all $y \geq 0$ and $\sqrt{\lambda}/\pi-M \leq k \leq \sqrt{\lambda}/\pi+1$. Consequently, it follows that
    \begin{align*}
        \sum_{k>\sqrt{\lambda}/\pi -M}\Bigl(\lambda- \Bigl(\pi k-2\arctan\Bigl(\frac{\pi k}{\beta}\Bigr)\Bigr)^2\Bigr)_\limplus^\gamma
        \leq (4\pi)^\gamma (M+1)M^{\gamma} \lambda^{\gamma/2}\, .
    \end{align*}

  The second sum on the right-hand side we can bound by
  \begin{align*}
    \sum_{k>\sqrt{\lambda}/\pi -M} (\lambda-\pi^2k^2)^\gamma_\limplus 
    &\leq (M+1)(\lambda-\pi^2(\sqrt{\lambda}/\pi-M)^2)^\gamma\\[-10pt]
    &= (M+1)(2\pi M \sqrt{\lambda} - \pi^2 M^2)^\gamma \\
    &\leq (M+1)(2\pi M)^\gamma \lambda^{\gamma/2}  , 
  \end{align*}
  where in the last step we used the assumption that $\lambda \geq 4\pi^2 M^2$. 
  
  It remains to bound the third sum on the right-hand side of \eqref{eq: lem26-divide}. By Taylor's theorem it holds that
    \begin{equation*}
        |(1-x)^{2j}-1+2jx| \leq 2j^2 x^2
    \end{equation*}
    for all $x \in [0, 1]$ and $j \in \mathbb{N}$. When combined with $\Bigl|\frac{\prod_{i=0}^{j-1}(i-\gamma)}{j!}\Bigr| = \Bigl|\frac{\gamma\Gamma(j-\gamma)}{\Gamma(j+1)\Gamma(1-\gamma)}\Bigr| \lesssim_\gamma j^{-\gamma-1}$, one finds
    {\allowdisplaybreaks
    \begin{align*}
        \biggl|\lambda^\gamma\sum_{1\leq k \leq \sqrt{\lambda}/\pi-M}&\sum_{j=1}^\infty\frac{\prod_{i=0}^{j-1}(i-\gamma)}{j!}\Bigl(\frac{\pi^2k^2}{\lambda}\Bigr)^j\biggl(\Bigl(1-\frac{2\arctan(\frac{\pi k}\beta)}{\pi k}\Bigr)^{2j}-1+\frac{4j\arctan(\frac{\pi k}\beta)}{\pi k}\biggr)\biggr|\\
        &\lesssim_\gamma
        \lambda^\gamma\sum_{1\leq k\leq \sqrt{\lambda}/\pi-M}\sum_{j=1}^\infty \frac{j^{1-\gamma}}{k^2}\Bigl(\frac{\pi^2k^2}{\lambda}\Bigr)^j\\
        &\leq
        \lambda^\gamma\sum_{j=1}^\infty {j^{1-\gamma}}\Bigl(\frac{\pi^{2}}{\lambda}\Bigr)^j\sum_{1\leq k\leq \sqrt{\lambda}/\pi-M}k^{2j-2}\\
        &\leq 
        \lambda^\gamma\sum_{j=1}^\infty j^{1-\gamma}\Bigl(\frac{\pi^{2}}{\lambda}\Bigr)^j\frac{\sqrt{\lambda}}{\pi}\Bigl(\frac{\sqrt{\lambda}}{\pi}-M\Bigr)^{2j-2}\\
        &=
        \pi \lambda^{\gamma-1/2}\sum_{j=1}^\infty j^{1-\gamma}\Bigl(1-\frac{\pi M}{\sqrt{\lambda}}\Bigr)^{2j-2}\, .
    \end{align*}
    }
    Since by assumption $\lambda \geq 4\pi^2M^2$ and $\gamma>0$, we have
    \begin{equation*}
        \sum_{j=1}^\infty j^{1-\gamma}\Bigl(1-\frac{\pi M}{\sqrt{\lambda}}\Bigr)^{2j-2}\leq \sum_{j=1}^\infty j 2^{2-2j} = \frac{1}{(1-1/4)^2} < \infty \, .
    \end{equation*}
    This proves the desired bound under the assumption that $\lambda \geq 4\pi^2 M^2$.

    \medskip
    \noindent\emph{Case 2:} $\lambda < 4\pi^2 M^2$. In this regime we simply show that in absolute value each of the three terms in Lemma \ref{lem: Riesz arctan sum estimate 1} are $\lesssim_{\gamma,M} \lambda^\gamma$, i.e.~we show
    \begin{align} \label{eq: lem2-6case2uppbnd1}
        \biggl|\sum_{k\geq 1}\Bigl(\lambda- \Bigl(\pi k-2\arctan\Bigl(\frac{\pi k}{\beta}\Bigr)\Bigr)^2\Bigr)_\limplus^\gamma \biggr| \lesssim_{\gamma,M} \lambda^\gamma \, ,  \qquad         \bigl|\Tr(-\Delta_{(0, 1)}^{\rm D}-\lambda)_\limminus^\gamma \bigr|  \lesssim_{\gamma,M} \lambda^\gamma \, 
    \end{align}
    and
    \begin{align}    \label{eq: lem2-6case2uppbnd2}
        \biggl|4\pi \gamma \sum_{1\leq k \leq \sqrt{\lambda}/\pi-M}k(\lambda-\pi^2k^2)^{\gamma-1}\arctan\Bigl(\frac{\pi k}{\beta}\Bigr) \biggr| &\lesssim_{\gamma,M} \lambda^\gamma \, .
    \end{align}
    Since $\lambda^\gamma\lesssim_{\gamma, M}\lambda^{\gamma/2}$ for all $0\leq \lambda \leq 4\pi^2 M^2$, the bound claimed in the lemma then follows by the triangle inequality.
    
    For the terms
    \begin{equation*}
        \sum_{k\geq 1}\Bigl(\lambda- \Bigl(\pi k-2\arctan\Bigl(\frac{\pi k}\beta\Bigr)\Bigr)^2\Bigr)^\gamma_\limplus \quad \mbox{and}\quad \Tr(-\Delta_{(0, 1)}^{\rm D}-\lambda)_\limminus^\gamma = \sum_{k\geq 1}(\lambda-\pi^2k^2)_\limplus^\gamma
    \end{equation*}
    we just note that each term in the respective sum is non-negative and bounded from above by $\lambda^\gamma$, and that there are no more than $\sqrt{\lambda}/\pi+1 \leq 2M +1\leq 3M$ non-zero terms in each sum. This shows \eqref{eq: lem2-6case2uppbnd1}.

    It remains to verify \eqref{eq: lem2-6case2uppbnd2}. Note that the sum is trivially zero if $\lambda \leq \pi^2 M^2$, thus we may restrict our attention to $\lambda \in (\pi^2 M^2, 4\pi^2M^2)$.
    Again we use that there are no more than $3M$ terms in the sum combined with a suitable bound for each term.

    Since $\gamma \in (0, 1]$, $\lambda\in (\pi^2 M^2, 4\pi^2M^2)$, and $k \in [1, \sqrt{\lambda}/\pi-M]$, we can bound
    \begin{align*}
        k(\lambda-\pi^2k^2)^{\gamma-1}\arctan\Bigl(\frac{\pi k}{\beta}\Bigr) 
        &\leq 
            \frac{\sqrt{\lambda}}{\pi}(\lambda- \pi^2(\sqrt{\lambda}/\pi-M)^2)^{\gamma-1} \cdot \frac{\pi}{2}\\
        &=\frac{\sqrt{\lambda}}{2} (2\pi M\sqrt{\lambda}- \pi^2 M^2)^{\gamma-1}\\
        &\leq
            \pi M  (3 \pi^2 M^2)^{\gamma-1}\\
        &\lesssim_\gamma  M^{2\gamma-1}\, .
    \end{align*}
    In particular, we conclude that
    \begin{align*}
        4\pi \gamma \sum_{1\leq k \leq \sqrt{\lambda}/\pi-M}k(\lambda-\pi^2k^2)^{\gamma-1}\arctan\Bigl(\frac{\pi k}{\beta}\Bigr) 
        \lesssim_\gamma M^{2\gamma} \lesssim_\gamma \lambda^{\gamma/2}\, .
    \end{align*}
    This completes the proof of Lemma~\ref{lem: Riesz arctan sum estimate 1}.
\end{proof}

The explicit sum from the previous lemma can be replaced by a suitable integral expression up to an acceptable error. This is shown in the next lemma.

\begin{lemma} \label{lem: arctan sum to integral estimate}
    Let $\gamma \in (0, 1]$ and $M\geq 2$. There exists a constant $C_{\gamma, M}>0$ so that, for any $\beta>0$ and $\lambda >0$, 
    \begin{equation*}
        \Biggl|\sum_{1 \leq k \leq \sqrt{\lambda}/\pi-M} \!\!\!\!\! k(\lambda-\pi^2k^2)^{\gamma-1}\arctan\Bigl(\frac{\pi k}{\beta}\Bigr) -  \frac{\lambda^\gamma}{\pi^2}\int_0^1 \!\! x(1-x^2)^{\gamma-1}\arctan\Bigl(\frac{x\sqrt{\lambda}}{\beta}\Bigr)\, dx\Biggr|\leq  C_{\gamma, M} \lambda^{\gamma/2}\, .
    \end{equation*}
\end{lemma}

\begin{proof}
For $\lambda < \pi^2 (M+1)^2$, the sum on the left-hand side is empty, hence the claimed bound follows by estimating 
\begin{align*} 
    \Biggl|  \frac{\lambda^\gamma}{\pi^2}\int_0^1 x(1-x^2)^{\gamma-1}\arctan\Bigl(\frac{x\sqrt{\lambda}}{\beta}\Bigr)\, dx\Biggr| 
        \lesssim_\gamma \lambda^\gamma\lesssim_{\gamma, M}   \lambda^{\gamma/2} \, .
\end{align*}
    
In the regime of $\lambda \geq \pi^2 (M+1)^2$, the sum can be rewritten as
\begin{align*}
    \sum_{1\leq k \leq \sqrt{\lambda}/\pi-M}&k (\lambda-\pi^2k^2)^{\gamma-1}\arctan\Bigl(\frac{\pi k}{\beta}\Bigr) \\
    &= \frac{\lambda^{\gamma-1/2}}{\pi} \sum_{1\leq k \leq \sqrt{\lambda}/\pi-M} \frac{\pi k}{ \sqrt{\lambda}} \biggl(1-\biggl( \frac{\pi k}{ \sqrt{\lambda}}\biggr)^2 \biggr)_\limplus^{\gamma-1}\arctan\Bigl(\frac{\pi k}{\beta}\Bigr) \\
    &= \frac{ \lambda^{\gamma}}{\pi^2} \cdot \frac{\pi }{ \sqrt{\lambda}} \sum_{1\leq k \leq \sqrt{\lambda}/\pi-M} f_{\lambda, \beta, \gamma}\biggl( \frac{\pi k}{ \sqrt{\lambda}}\biggr)
\end{align*}
where $f_{\lambda, \beta, \gamma}(x)=x (1-x^2)^{\gamma-1}_+ \arctan(x\sqrt{\lambda}/ \beta)$. The remaining sum is a quadrature of the integral 
\begin{align*}
    \int_0^1 f_{\lambda, \beta, \gamma}(x) \, dx\, , 
\end{align*}
up to $M$ missing terms on the right endpoint of the interval.
We claim that the quadrature expression converges to the integral as $\lambda \to \infty$ and that 
\begin{equation}
     \biggl|\frac{\pi }{ \sqrt{\lambda}} \sum_{1\leq k \leq \sqrt{\lambda}/\pi-M} f_{\lambda, \beta, \gamma}\biggl( \frac{\pi k}{ \sqrt{\lambda}}\biggr) -  \int_0^1 f_{\lambda, \beta, \gamma}(x) \, dx\biggr|\lesssim_{\gamma, M} \lambda^{-\gamma/2}
     \label{eq:quadrature_error_term}
\end{equation}
for $\lambda \geq \pi^2 (M+1)^2$. To complete the proof of the lemma, it remains to prove \eqref{eq:quadrature_error_term}.

For $\gamma \in (0, 1]$ the functions $x \mapsto x(1-x^2)^{\gamma-1}$ and $x\mapsto x \arctan(x\sqrt{\lambda}/\beta)$ are both non-negative and non-decreasing. Thus, as a product of non-negative and non-decreasing functions, $f_{\lambda, \beta, \gamma}$ is non-negative and non-decreasing. Since $M\geq 2$ and $f_{\lambda, \beta, \gamma}(0)=0$ it follows that
\begin{equation*}
    \int_0^{1-\pi(M+1)/\sqrt{\lambda}} f_{\lambda, \beta, \gamma}(x) \, dx\leq \frac{\pi }{ \sqrt{\lambda}} \sum_{1\leq k \leq \sqrt{\lambda}/\pi-M} f_{\lambda, \beta, \gamma}\biggl( \frac{\pi k}{ \sqrt{\lambda}}\biggr) \leq   \int_0^1 f_{\lambda, \beta, \gamma}(x) \, dx
\end{equation*}
provided that $\lambda \geq \pi^2(M+1)^2$. Since $0\leq f_{\lambda, \beta, \gamma}(x)\leq \frac{\pi}{2} (1-x^2)^{\gamma-1}$ for all $\beta>0, \lambda>0, \gamma>0, $ and $x\in [0, 1]$ it therefore holds that
\begin{align*}
    \biggl|\frac{\pi }{ \sqrt{\lambda}} \sum_{1\leq k \leq \sqrt{\lambda}/\pi-M} f_{\lambda, \beta, \gamma}\biggl( \frac{\pi k}{ \sqrt{\lambda}}\biggr)-\int_0^{1} f_{\lambda, \beta, \gamma}(x) \, dx\biggr|
    &\leq \frac{\pi}{2}\int_{1-\pi(M+1)/\sqrt{\lambda}}^1  (1-x^2)^{\gamma-1}\, dx\\
    &\leq \pi 2^{\gamma-2}\int_{1-\pi(M+1)/\sqrt{\lambda}}^1  (1-x)^{\gamma-1}\, dx\\
    &= \frac{\pi^{\gamma+1}2^{\gamma-2}}{\gamma}(M+1)^\gamma \lambda^{-\gamma/2}\, .
\end{align*}
This justifies \eqref{eq:quadrature_error_term} and hence completes the proof of Lemma~\ref{lem: arctan sum to integral estimate}.
\end{proof}

The next lemma shows that the integral that appears in Lemma~\ref{lem: arctan sum to integral estimate} can be rewritten to match that in the definition of $L_{\gamma, 0}(\beta)$. We also take the opportunity to prove the properties of the function $\beta \mapsto L_{\gamma, d}(\beta)$ mentioned in the introduction.
\begin{lemma}\label{lem: properties of L}
    Let $d\geq 0, \gamma \geq 0$. The function $(0, \infty) \ni\beta \mapsto L_{\gamma, d}(\beta)$ is smooth, bounded, decreasing, and satisfies
    \begin{equation*}
        \lim_{\beta \to 0^\limplus} L_{\gamma, d}(\beta)= L_{\gamma, d}^{\rm sc} \quad \mbox{and}\quad \lim_{\beta \to \infty}L_{\gamma, d}(\beta) =-L_{\gamma, d}^{\rm sc}\, .
    \end{equation*}
    Moreover, for any $\beta>0$, 
    \begin{equation*}
        L_{\gamma, d}(\beta) = L_{\gamma, d}^{\rm sc}\biggl[\frac{8}{\pi}\Bigl(\gamma+\frac{d}{2}\Bigr)\int_0^1 x(1-x^2)^{\gamma+d/2-1}\arctan\Bigl(\frac{x}{\beta}\Bigr)\, dx - 1\biggr]\, .
    \end{equation*}
\end{lemma}

\begin{proof}
    We begin by proving the alternative integral representation. To do this we must show that
    \begin{equation*}
        \frac{8}{\pi}\Bigl(\gamma+\frac{d}{2}\Bigr)\int_0^1 x(1-x^2)^{\gamma+d/2-1}\arctan\Bigl(\frac{x}{\beta}\Bigr)\, dx = \frac{4}{\pi}\int_0^1(1-s^2)^{\gamma+d/2}\frac{\beta}{\beta^2+s^2}\, ds\, .
    \end{equation*}
    By writing $\arctan$ as an integral and using Fubini's theorem we have
    \begin{align*}
        \int_0^1 x(1-x^2)^{\gamma+d/2-1}\arctan\Bigl(\frac{x}{\beta}\Bigr)\, dx 
        &= \int_0^1 x(1-x^2)^{\gamma+d/2-1}\biggl(\int_0^{x} \frac{\beta}{\beta^2+s^2}\, ds\biggr)\, dx\\
        &= \int_0^1\biggl(\int_s^1 x(1-x^2)^{\gamma+d/2-1}\, dx\biggr) \frac{\beta}{\beta^2+s^2}\, ds\\
        &= \frac{1}{2\gamma+d}\int_0^1 (1-s^2)^{\gamma+d/2} \frac{\beta}{\beta^2+s^2}\, ds\, .
    \end{align*}
    This proves the desired identity.

    The claimed boundedness, smoothness, and monotonicity follow from the proved representation of $L_{\gamma, d}(\beta)$ and the corresponding properties of $\arctan$. The claimed limiting properties follow from the dominated convergence theorem, $\lim_{\beta \to \infty}\arctan(x/\beta)=0$, $\lim_{\beta \to 0^\limplus}\arctan(x/\beta)=\pi/2$ for any $x\in (0, 1)$, and the fact that
    \begin{equation*}
        4\Bigl(\gamma+\frac{d}{2}\Bigr)\int_0^1 x(1-x^2)^{\gamma+d/2-1}\, dx = 2\, .
    \end{equation*}
    This completes the proof of Lemma~\ref{lem: properties of L}.
\end{proof}

The final ingredient in our proof of Theorem~\ref{thm: two-term Weyl 1D} is a two-term asymptotic expansion for Riesz means of the Dirichlet Laplacian.

\begin{lemma} \label{lem: Dirichlet Riesz mean asymptotics}
    Let $\gamma \in (0, 1]$. There exists a constant $C_\gamma>0$ such that, for all $\lambda >0$, 
    \begin{equation*}
        \Bigl|\Tr(-\Delta_{(0, 1)}^{\rm D}-\lambda)_\limminus^\gamma-L_{\gamma, 1}^{\rm sc}\lambda^{\gamma+1/2}+\frac{1}{2}\lambda^\gamma\Bigr| \leq C_\gamma \lambda^{\gamma/2}\, .
    \end{equation*}
\end{lemma}

\begin{proof}
    The claim is deduced as in the proof of \cite[Theorem 1.1]{FrankLarson_Inventiones25} from the estimates
    \begin{equation}\label{eq: 1D Dirichlet counting estimate}
        \Bigl|\Tr(-\Delta_{(0, 1)}^{\rm D}-\lambda)_\limminus^0 - \frac{1}\pi\sqrt{\lambda}\Bigr| \leq 1
    \end{equation}
    and
    \begin{equation}\label{eq: 1D Dirichlet Riesz 1 estimate}
        \Bigl|\Tr(-\Delta_{(0, 1)}^{\rm D}-\lambda)_\limminus - \frac{2}{3\pi}\lambda^{3/2}+ \frac{1}{2}\lambda\Bigr| \lesssim \lambda^{1/2}\, . 
    \end{equation}
    
    The estimate \eqref{eq: 1D Dirichlet counting estimate} follows directly from $\lambda_k(-\Delta_{(0, 1)}^{\rm D})=\pi^2k^2$. The estimate in \eqref{eq: 1D Dirichlet Riesz 1 estimate} follows almost as directly, by explicitly calculating the sum defining the Riesz mean: We have
    \begin{align*}
        \Tr(-\Delta_{(0, 1)}^{\rm D}-\lambda)_\limminus 
        &=
        \sum_{1 \leq k < \sqrt{\lambda}/\pi}(\lambda-\pi^2k^2)\\
        &= \frac{2}{3\pi}\lambda^{3/2}- \frac{1}{2}\lambda+\pi \sqrt{\lambda}\Bigl\{\frac{\sqrt{\lambda}}\pi\Bigr\}\Bigl(1-\Bigl\{\frac{\sqrt{\lambda}}\pi\Bigr\}\Bigr)\\
        &\quad \frac{\pi^2}{6}\Bigl\{\frac{\sqrt{\lambda}}{\pi}\Bigr\}\Bigl(1-3\Bigl\{\frac{\sqrt{\lambda}}{\pi}\Bigr\}+2\Bigl\{\frac{\sqrt{\lambda}}{\pi}\Bigr\}^2\Bigr)\, , 
    \end{align*}
    where $\{x\} := x - \lfloor x\rfloor$ denotes the fractional part of $x \in \R$.
    The claimed bound now follows from noting that $0\leq \{x\}\leq \min\{1, x\}$ for all $x\geq 0$.
\end{proof}

We are finally ready to prove Theorem~\ref{thm: two-term Weyl 1D}.
\begin{proof}[Proof of Theorem~\ref{thm: two-term Weyl 1D}]

We treat the cases $0<\gamma \leq 1$ and $\gamma > 1$ separately.

\medskip
\noindent\textit{Case 1: $0<\gamma \leq 1$.} Using the representation for $L_{\gamma, 0}(\beta)$ in Lemma~\ref{lem: properties of L} we have
{\allowdisplaybreaks
\begin{align*}
    \Tr(-\Delta_{(0, 1)}^{\beta}&-\lambda)_\limminus^\gamma - L_{\gamma, 1}^{\rm sc}\lambda^{\gamma+1/2}- \frac{1}{2}L_{\gamma, 0}(\beta/\sqrt{\lambda})\lambda^\gamma\\
    & =\Biggl( \Tr(-\Delta_{(0, 1)}^{\beta}-\lambda)_\limminus^\gamma - \sum_{k\geq 1}\Bigl(\lambda- \Bigl(\pi k-2\arctan\Bigl(\frac{\pi k}{\beta}\Bigr)\Bigr)^2\Bigr)_\limplus^\gamma \Biggr)\\
    & \quad +\Biggl(\, \sum_{k\geq 1}\Bigl(\lambda- \Bigl(\pi k-2\arctan\Bigl(\frac{\pi k}{\beta}\Bigr)\Bigr)^2\Bigr)_\limplus^\gamma-\Tr(-\Delta_{(0, 1)}^{\rm D}-\lambda)_\limminus^\gamma\\
  & \qquad \quad - 4\pi \gamma \sum_{1\leq k \leq \sqrt{\lambda}/\pi-M}k(\lambda-\pi^2k^2)^{\gamma-1}\arctan\Bigl(\frac{\pi k}{\beta}\Bigr)\Biggr) \\
  & \quad + 4\pi \gamma\Biggl(\, \sum_{1 \leq k \leq \sqrt{\lambda}/\pi-M}  k(\lambda-\pi^2k^2)^{\gamma-1}\arctan\Bigl(\frac{\pi k}{\beta}\Bigr)\\
  &\qquad \quad -  \frac{\lambda^\gamma}{\pi^2}\int_0^1 x(1-x^2)^{\gamma-1}\arctan\Bigl(\frac{x\sqrt{\lambda}}{\beta}\Bigr)\, dx\Biggr) \\
  & \quad + \biggl(\Tr(-\Delta_{(0, 1)}^{\rm D}-\lambda)_\limminus^\gamma - L_{\gamma, 1}^{\rm sc}\lambda^{\gamma+1/2} + \frac{1}{2} \lambda^\gamma \biggr)\, .
\end{align*}
}
By Lemmas~\ref{lem: Riesz arctan sum estimate 1}, \ref{lem: arctan sum to integral estimate} and \ref{lem: Dirichlet Riesz mean asymptotics}, respectively, the absolute values of the expressions in the last three parentheses can each be estimated $\lesssim_\gamma\lambda^{\gamma/2}$. Due to the two-sided bounds of Lemma~\ref{lem: arctan bounds lambdak} combined with Lemma~\ref{lem: unif bound Riesz arctan}, we have for the quantity in the first parentheses
\begin{align*}
    \biggl| \Tr(&-\Delta_{(0, 1)}^{\beta}-\lambda)_\limminus^\gamma - \sum_{k\geq 1}\Bigl(\lambda- \Bigl(\pi k-2\arctan\Bigl(\frac{\pi k}{\beta}\Bigr)\Bigr)^2\Bigr)_\limplus^\gamma \biggr| \\
    &\leq \biggl|\sum_{k\geq 1}\biggl[\Bigl(\lambda- \Bigl(\pi k-2\arctan\Bigl(\frac{\pi k}{\beta}\Bigr)\Bigr)^2\Bigr)_\limplus^\gamma-\Bigl(\lambda- \Bigl(\pi k-2\arctan\Bigl(\frac{\pi (k-1)}{\beta}\Bigr)\Bigr)^2\Bigr)_\limplus^\gamma\biggr]\biggr| \\
    &\lesssim_\gamma \lambda^{\gamma/2}\, .
\end{align*}
Collecting these estimates, one arrives at
\begin{align*}
    \biggl| \Tr(-\Delta_{(0, 1)}^{\beta}-\lambda)_\limminus^\gamma - L_{\gamma, 1}^{\rm sc}\lambda^{\gamma+1/2}- \frac{1}{2}L_{\gamma, 0}(\beta/\sqrt{\lambda})\lambda^\gamma \biggr| \lesssim_{\gamma}\lambda^{\gamma/2}\, , 
\end{align*}
which is the desired bound.

\medskip

\noindent\textit{Case 2: $\gamma > 1$.} We treat the case $\gamma>1$ with the Aizenman--Lieb identity. In the previous case, we have proven for $\gamma=1$ that
    \begin{align} \label{eq:Riesz_gamma1_restterm}
         \Tr(-\Delta_{(0, 1)}^{\beta}-\lambda)_\limminus = L_{1, 1}^{\rm sc}\lambda^{3/2}+ \frac{1}{2}L_{1, 0}(\beta/\sqrt{\lambda}) \lambda + R(\lambda)
    \end{align}
    where $|R(\lambda)| \lesssim \sqrt{\lambda}$ for all $\beta>0$ and $\lambda \geq 0$. A particular case of the Aizenman--Lieb identity \cite{AizenmanLieb} states that, for $\gamma>1$ and $\lambda >0$, 
\begin{align*}
     \Tr(-\Delta_{(0, 1)}^{\beta}-\lambda)_\limminus^\gamma = \gamma (\gamma-1) \int_0^\lambda (\lambda- \tau)^{\gamma-2} \Tr(-\Delta_{(0, 1)}^{\beta}-\tau)_\limminus \, d\tau\, .
\end{align*}
Inserting \eqref{eq:Riesz_gamma1_restterm} allows us to compute the corresponding bounds for $\gamma>1$.

First, a standard calculation shows
\begin{align*}
 \gamma (\gamma-1) \int_0^\lambda (\lambda- \tau)^{\gamma-2}   L_{1, 1}^{\rm sc}\tau^{3/2}  \, d\tau = L_{\gamma, 1}^{\rm sc}\lambda^{\gamma +1/2}
\end{align*}
for the first term. To lift the second term, we need to show that
\begin{equation}\label{eq: AL lift second term}
    \gamma (\gamma-1)\int_0^\lambda (\lambda- \tau)^{\gamma-2} L_{1, 0}(\beta/\sqrt{\tau}) \tau \, d\tau = L_{\gamma, 0}(\beta/\sqrt{\lambda})\lambda^\gamma\, .
\end{equation}
For this, we use the representation for $L_{\gamma, d}(\beta)$ provided by Lemma~\ref{lem: properties of L}, 
\begin{align*}
    \gamma (\gamma-1) \int_0^\lambda& (\lambda- \tau)^{\gamma-2} L_{1, 0}(\beta/\sqrt{\tau}) \tau \, d\tau\\
    &=
    \gamma (\gamma-1) \int_0^\lambda (\lambda- \tau)^{\gamma-2} \biggl[\frac{8}{\pi}\int_0^1 x\arctan\Bigl(\frac{x \sqrt{\tau}}{\beta}\Bigr)\, dx -1\biggr] \tau \, d\tau.
\end{align*}

Since $L_{\gamma, 0}^{\rm sc}= 1$ and
\begin{align*}
    \gamma (\gamma-1) \int_0^\lambda (\lambda- \tau)^{\gamma-2} \tau \, d\tau = \lambda^\gamma\, , 
\end{align*}
it remains to show that
\begin{equation}\label{eq: lifting of L integral}
    (\gamma-1)\int_0^\lambda \tau (\lambda- \tau)^{\gamma-2} \int_0^1 x \arctan\Bigl(\frac{x \sqrt{\tau}}{\beta} \Bigr) \, dx d\tau = \lambda^\gamma \int_0^1 x(1-x^2)^{\gamma-1}\arctan\Bigl(\frac{x\sqrt{\lambda}}{\beta}\Bigr)\, dx\, .
\end{equation}
Changing variables in the inner integral on the left-hand side by setting $s= x\sqrt{\tau}$ and defining 
\begin{equation*}
    G(\eta) = \int_0^\eta s\arctan\Bigl(\frac{s}{\beta}\Bigr)\, ds 
\end{equation*}
gives
\begin{align*}
    (\gamma-1)\int_0^\lambda \tau (\lambda- \tau)^{\gamma-2} \int_0^1 x \arctan\Bigl(\frac{x \sqrt{\tau}}{\beta} \Bigr) \, dx d\tau =(\gamma-1)\int_0^\lambda (\lambda- \tau)^{\gamma-2} G(\sqrt{\tau}) d\tau \, .
\end{align*}
Since $G(0)=0$ and $(\gamma-1)(\lambda-\tau)^{\gamma-2}=-\frac{d}{d\tau}(\lambda-\tau)^{\gamma-1}$ an integration by parts and the fundamental theorem of calculus yields
\begin{align*}
    (\gamma-1)\int_0^\lambda \tau (\lambda- \tau)^{\gamma-2} \int_0^1 x \arctan\Bigl(\frac{x \sqrt{\tau}}{\beta} \Bigr) \, dx d\tau 
    &=
        \frac{1}{2}\int_0^\lambda (\lambda- \tau)^{\gamma-1} \frac{G'(\sqrt{\tau})}{\sqrt{\tau}}\, d\tau\\
    &=
        \frac{1}{2}\int_0^\lambda (\lambda- \tau)^{\gamma-1} \arctan\Bigl(\frac{\sqrt{\tau}}{\beta}\Bigr)\, d\tau\, .
\end{align*}
Substituting $\tau=\lambda x^2$, yields \eqref{eq: lifting of L integral} and thus completes the proof of \eqref{eq: AL lift second term}.

Finally, having shown that the first two asymptotic terms transform correctly, we have
{\allowdisplaybreaks
\begin{align*}
    \biggl| \Tr(-\Delta_{(0, 1)}^{\beta}-\lambda)_\limminus^\gamma - L_{\gamma, 1}^{\rm sc}\lambda^{\gamma+1/2}- \frac{1}{2}L_{\gamma, 0}(\beta/\sqrt{\lambda}) \lambda^\gamma \biggr|
    & \leq \gamma (\gamma-1) \int_0^\lambda (\lambda- \tau)^{\gamma-2}  | R(\tau) |  \, d\tau \\
    & \lesssim_\gamma \int_0^\lambda (\lambda- \tau)^{\gamma-2}    \sqrt{\tau}  \, d\tau \\
    &  \lesssim_\gamma \lambda^{\gamma-1/2}\, , 
\end{align*}
}
for all $\gamma>1$ and $\lambda \geq 0$. This concludes the proof of Theorem~\ref{thm: two-term Weyl 1D}. 
\end{proof}

\section{Spectral asymptotics for Riesz means on cuboids}
\label{sec: d-dim spectral asymptotics}

In this section we argue that our estimate obtained in Theorem~\ref{thm: two-term Weyl 1D} implies corresponding estimates for Riesz means on cuboids in any dimension.

\begin{theorem}\label{thm: two-term Weyl}
    Let $d\geq 2, \gamma>0$, and $\kappa_{\gamma, d} = \min\bigl\{\frac{\gamma}{1+\gamma}, \frac{1}{d}\bigr\}$. There exists a constant $C_{\gamma, d}>0$ such that for all $\lambda \geq 0$, $\beta>0$, and any cuboid $R\subset \R^d$ with side lengths $l_1, \ldots, l_d>0$, 
    \begin{align*}
        \biggl|\Tr(-\Delta_{R}^{\beta}-\lambda)_\limminus^\gamma &- L_{\gamma, d}^{\rm sc}|R|\lambda^{\gamma+d/2} -\frac{1}{4}L_{\gamma, d-1}(\beta/\sqrt{\lambda}) \Haus^{d-1}(\partial R)\lambda^{\gamma+(d-1)/2}\biggr|\\
        &\leq C_{\gamma, d}\Haus^{d-1}(\partial R)\lambda^{\gamma+(d-1)/2}\Bigl((\min_i l_i\sqrt{\lambda})^{-\kappa_{\gamma, d}}+(\min_i l_i\sqrt{\lambda})^{1-d}\Bigr)\, .
    \end{align*}
\end{theorem}

For the Riesz means of Robin Laplace operators with Robin parameter proportional to $\sqrt{\lambda}$, Theorems \ref{thm: two-term Weyl 1D} and \ref{thm: two-term Weyl} imply the following quantitative version of Theorem~\ref{thm: sc two-term Weyl intro}. 
\begin{corollary}\label{cor: sc two-term Weyl}
    Let $d\geq 1, \gamma>0$. There exist constants $C_{\gamma, d}>0$ and $\kappa_{\gamma, d}>0$ such that for all $\lambda \geq 0$, $\beta>0$, and any cuboid $R \subset \R^d$ with side lengths $l_1, \ldots, l_d>0$, it holds that
    \begin{align*}
        \biggl|\Tr(-\Delta_{R}^{\beta\sqrt{\lambda}}-\lambda)_\limminus^\gamma &- L_{\gamma, d}^{\rm sc}|R|\lambda^{\gamma+d/2} -\frac{1}{4}L_{\gamma, d-1}(\beta) \Haus^{d-1}(\partial R)\lambda^{\gamma+(d-1)/2}\biggr|\\
        &\leq C_{\gamma, d}\Haus^{d-1}(\partial R)\lambda^{\gamma+(d-1)/2}\Bigl((\min_i l_i\sqrt{\lambda})^{-\kappa_{\gamma, d}}+\1_{d\geq 2}(\min_i l_i\sqrt{\lambda})^{1-d}\Bigr)\, .
    \end{align*}
\end{corollary}

In our proof of Theorem~\ref{thm: two-term Weyl}, we shall use the following a priori bound for the Riesz means of a cuboid.
\begin{lemma}\label{lem: Neumann Riesz mean bound}
    Let $d\geq 1, \gamma \geq 0$. There exists a constant $C_{d}>0$ so that for all $\lambda \geq 0$, $\sharp \in \{\rm D, \rm N, \beta\}$ with $\beta>0$, and any cuboid $R\subset \R^d$ with side lengths $l_1, \ldots, l_d>0$, 
    \begin{equation*}
        \Tr(-\Delta_{R}^\sharp-\lambda)_\limminus^\gamma\leq C_d|R|\lambda^{\gamma+d/2}(1+(\min_{i}l_i\sqrt{\lambda})^{-d})\, .
    \end{equation*}
\end{lemma}

\begin{proof}
    By the variational principle we have $$\Tr(-\Delta_R^{\rm D}-\lambda)_\limminus^\gamma\leq \Tr(-\Delta_R^\beta-\lambda)_\limminus^\gamma \leq \Tr(-\Delta_R^{\rm N}-\lambda)_\limminus^\gamma\, .$$ Thus it suffices to prove the claimed bound for the Neumann Laplacian.

    By invariance of the Laplacian under orthogonal transformations we may assume without loss of generality that $l_1 = \min_i l_i$. We shall prove that we can take $C_d=3^{d-1}$ and argue by induction in $d$.

    For $d=1$ we have 
    \begin{equation*}
        \Tr(-\Delta_{(0, l_1)}^{\rm N}-\lambda)_\limminus^\gamma = \sum_{k\geq 0}(\lambda-\pi^2k^2/l_1^2)^\gamma_\limplus \leq \lambda^\gamma \Bigl(\frac{l_1\sqrt{\lambda}}{\pi}+1\Bigr) \leq l_1 \lambda^{\gamma+1/2}(1+(l_1 \sqrt{\lambda})^{-1})\, , 
    \end{equation*}
    which proves the claim for $d=1$ with $C_1=1=3^0$.

    We now prove the claim in dimension $d$ assuming its validity in dimension $d-1$. Let $R'=\prod_{i=1}^{d-1}(0, l_i)\subset \R^{d-1}$. First, by the induction hypothesis in dimension $d-1$, it follows that
    \begin{align*}
        \Tr(-\Delta_{R}^{\rm N}-\lambda)_\limminus^\gamma 
        &=
        \sum_{k\geq 0}\Tr(-\Delta_{R'}^{\rm N}-(\lambda-\pi^2k^2/l_d^2))_\limminus^\gamma\\
        &\leq 3^{d-2} |R'|   \sum_{k\geq 0} (\lambda-\pi^2k^2/l_d^2)^{\gamma + (d-1)/2} (1+l_1^{-(d-1)} (\lambda-\pi^2k^2/l_d^2)^{-(d-1)/2})\\
        &=
        3^{d-2} |R'| \Bigl(\Tr(-\Delta_{(0, l_d)}^{\rm N}-\lambda)_\limminus^{\gamma+(d-1)/2}+
         l_1^{-d+1}\Tr(-\Delta_{(0, l_d)}^{\rm N}-\lambda)_\limminus^{\gamma} \Bigr)\,.
    \end{align*}
    We then apply the bound in dimension $d=1$ to both remaining Riesz means to obtain
     \begin{align*}    
        \Tr(-\Delta_{R}^{\rm N}-\lambda)_\limminus^\gamma &\leq 
        3^{d-2} |R'| \left(l_d \lambda^{\gamma+d/2}(1+(l_d\sqrt{\lambda})^{-1})+
         l_1^{-d+1}l_d\lambda^{\gamma+1/2}(1+(l_d\sqrt{\lambda})^{-1}) \right)\\
        &=
        3^{d-2} |R| \lambda^{\gamma+d/2}(1+(l_d\sqrt{\lambda})^{-1}+(l_1\sqrt{\lambda})^{-d+1}+(l_1\sqrt{\lambda})^{-d+1}(l_d\sqrt{\lambda})^{-1})\\
        &\leq
        3^{d-2} |R| \lambda^{\gamma+d/2}(1+(l_1\sqrt{\lambda})^{-1}+(l_1\sqrt{\lambda})^{-d+1}+(l_1\sqrt{\lambda})^{-d})\\
        &\leq
        3^{d-1} |R| \lambda^{\gamma+d/2}(1+(l_1\sqrt{\lambda})^{-d})
    \end{align*}
    where in the last two inequalities we have used that $l_1 \leq l_d$ and $(l_1\sqrt{\lambda})^{-1}+(l_1\sqrt{\lambda})^{-d+1}\leq 2 (1+(l_1\sqrt{\lambda})^{-d})$. This completes the proof of Lemma~\ref{lem: Neumann Riesz mean bound}. 
\end{proof}

\begin{proof}[Proof of Theorem~\ref{thm: two-term Weyl}]
    We divide the proof into two cases depening on the size of $\min_i l_i\sqrt{\lambda}$.

    \smallskip

    \noindent{\it Case 1: $\min_i l_i \sqrt{\lambda} \lesssim 1$.}
    In this case we estimate each of the terms on the left-hand side of the bound in Theorem~\ref{thm: two-term Weyl} individually. By Lemma~\ref{lem: Neumann Riesz mean bound} and since $L_{\gamma, d}(\beta)\lesssim_{\gamma, d}1, |R|\lesssim_d l_i\Haus^{d-1}(\partial R)$, 
    \begin{align*}
        \Tr(-\Delta_\Omega^{\beta}-\lambda)_\limminus^\gamma \lesssim_d |R|\lambda^{\gamma+d/2}(\min_i l_i\sqrt{\lambda})^{-d} &\lesssim_{\gamma, d}\Haus^{d-1}(\partial R)\lambda^{\gamma+(d-1)/2}(\min_i l_i\sqrt{\lambda})^{1-d}\, , \\
        L_{\gamma, d}^{\rm sc}|R|\lambda^{\gamma+d/2} \lesssim_d \Haus^{d-1}(\partial R)\lambda^{\gamma+(d-1)/2}(l_i\sqrt{\lambda})&\lesssim_{\gamma, d}\Haus^{d-1}(\partial R)\lambda^{\gamma+(d-1)/2}(\min_i l_i\sqrt{\lambda})^{1-d}\, , \\
        |L_{\gamma, d-1}(\beta/\sqrt{\lambda})|\Haus^{d-1}(\partial R)\lambda^{\gamma+(d-1)/2}&\lesssim_{\gamma, d}\Haus^{d-1}(\partial R)\lambda^{\gamma+(d-1)/2}(\min_i l_i\sqrt{\lambda})^{1-d}\, .
    \end{align*}
    Therefore, by the triangle inequality, 
    \begin{align*}
        \Bigl|\Tr(-\Delta_\Omega^{\beta}-\lambda)_\limminus^\gamma& -L_{\gamma, d}^{\rm sc}|R|\lambda^{\gamma+d/2} - \frac{1}{4}L_{\gamma, d-1}(\beta/\sqrt{\lambda})\Haus^{d-1}(\partial R)\lambda^{\gamma+(d-1)/2}\Bigr|\\
        &\lesssim_{\gamma, d}\Haus^{d-1}(\partial R)\lambda^{\gamma+(d-1)/2}(\min_i l_i\sqrt{\lambda})^{1-d}\, .
    \end{align*}
    This completes the proof of Theorem \ref{thm: two-term Weyl} when $\min_i l_i \sqrt{\lambda} \lesssim 1$.

    \smallskip

    \noindent{\it Case 2: $\min_i l_i \sqrt{\lambda} \gtrsim 1$.} In this case we argue by induction in $d$. The base case for our induction is the corresponding asymptotics in dimension $d=1$ which, up to rescaling, is the content of Theorem~\ref{thm: two-term Weyl 1D}. Specifically, Theorem~\ref{thm: two-term Weyl 1D} gives the bound
    \begin{equation} \label{eq: 1D induction step}  
        \biggl| \Tr(-\Delta_{(0, l_1)}^{\beta}-\lambda)_\limminus^\gamma - L_{\gamma, 1}^{\rm sc}l_1\lambda^{\gamma+1/2}- \frac{1}{4}L_{\gamma, 0}(\beta/\sqrt{\lambda})\Haus^{0}(\partial(0, l_1)) \lambda^\gamma \biggr|
        \lesssim_\gamma  \lambda^{\gamma}(l_1\sqrt{\lambda})^{-\kappa_{\gamma, 1}} \, , 
    \end{equation}
    for all $l_1>0, \beta>0, \lambda\geq 0$, with $\kappa_{\gamma, 1}= \min\{\gamma, 1\}$.

    Let us now consider the $d$-dimensional case and carry out the induction step. For $i \in \{1, \ldots, d\}$, let $R_i'= \prod_{j\neq i}^d (0, l_j) \subset \R^{d-1}$. For any choice of $i\in \{1, \ldots, d\}$, separation of variables yields that
    \begin{align*}
        \Tr(-\Delta_R^{\beta} -\lambda)_\limminus^\gamma 
        &=
        \sum_{k, m\geq 1} \bigl(\lambda-\lambda_k(-\Delta_{R'_i}^\beta)-\lambda_{m}(-\Delta_{(0, l_i)}^{\beta})\bigr)^\gamma_\limplus\\
        &=
        \sum_{k\geq 1} \Tr\bigl(-\Delta_{(0, l_i)}^{\beta}-\bigl(\lambda-\lambda_{k}(-\Delta_{R'_i}^{\beta})\bigr)_\limplus\bigr)^\gamma_\limminus\, \, .
    \end{align*}
    Therefore, by \eqref{eq: 1D induction step}, 
    \begin{align*}
        \Biggl|\Tr(-\Delta_R^{\beta} -\lambda)_\limminus^\gamma &-L_{\gamma, 1}^{\rm sc}l_i\Tr(-\Delta_{R'_i}^{\beta}-\lambda)_\limminus^{\gamma+1/2}\\
        &\quad -\frac{1}{2}\sum_{k \geq 1}L_{\gamma, 0}\bigl(\beta(\lambda-\lambda_k(-\Delta_{R'_i}^\beta))_\limplus^{-1/2}\bigr)\bigl(\lambda-\lambda_{k}(-\Delta_{R'_i}^{\beta})\bigr)_\limplus^\gamma\Biggr| \\
        &\lesssim_\gamma 
        l_i^{-\kappa_{\gamma, 1}}\Tr(-\Delta_{R'_i}^{\beta}-\lambda)_\limminus^{\gamma-\kappa_{\gamma, 1}/2}\, .
    \end{align*}
    We now apply the induction hypothesis to replace $\Tr(-\Delta_{R'_i}^{\beta}-\lambda)_\limminus^{\gamma}$ by its asymptotic expansion. Together with the a priori estimate of Lemma~\ref{lem: Neumann Riesz mean bound} and the fact that $L_{\gamma, 1}^{\rm sc}L_{\gamma+1/2, d-1}^{\rm sc}= L_{\gamma, d}^{\rm sc}$, this yields
    
    \begin{align}
        \Biggl|\Tr(-&\Delta_R^{\beta} -\lambda)_\limminus^\gamma -L_{\gamma, d}^{\rm sc}|R|\lambda^{\gamma+d/2}-\frac{1}{4}L_{\gamma, 1}^{\rm sc}L_{\gamma+1/2, d-2}(\beta/\sqrt{\lambda})l_i\Haus^{d-2}(\partial R_i')\lambda^{\gamma+(d-1)/2} \notag\\
        &\quad -\frac{1}{2}\sum_{k \geq 1}L_{\gamma, 0}(\beta (\lambda-\lambda_k(-\Delta_{R'_i}^\beta))_\limplus^{-1/2})\bigl(\lambda-\lambda_{k}(-\Delta_{R'_i}^{\beta})\bigr)_\limplus^\gamma\Biggr|\label{eq: induction step} \\
        &\lesssim_{\gamma, d} 
        l_i^{-1}|R|\lambda^{\gamma+(d-1)/2}(l_i\sqrt{\lambda})^{-\kappa_{\gamma, 1}} +  (\min_{j\neq i}l_j)^{-1}|R|\lambda^{\gamma+(d-1)/2} (\min_{j\neq i}l_j\sqrt{\lambda})^{-\kappa_{\gamma+1/2, d-1}}\, , \notag
    \end{align}
    where we used that $\min_j l_j \sqrt{\lambda} \gtrsim 1$ and that $d-2\geq 1 \geq \kappa_{\gamma+1/2, d-1}$ if $d\geq 3$.
    
    Note that
    \begin{equation*}
        \Haus^{d-1}(\partial R) = 2 \sum_{j=1}^d \frac{|R|}{l_j} \,, \qquad \Haus^{d-2}(\partial R_i') = 2\sum_{j\neq i} \frac{|R_i'|}{l_j} = l_i^{-1}\Haus^{d-1}(\partial R) - 2 \frac{|R|}{l_i} \,,
    \end{equation*}
    and that $L_{\gamma, 1}^{\rm sc}L_{\gamma+1/2, d-2}(\beta/\sqrt{\lambda}) = L_{\gamma, d-1}(\beta/\sqrt{\lambda})$, since $L_{\gamma, 1}^{\rm sc}L_{\gamma+1/2, d-2}^{\rm sc}= L_{\gamma, d-1}^{\rm sc}$. Thus~\eqref{eq: induction step} is equivalent to the estimate
    {\allowdisplaybreaks
    \begin{align}
        \Biggl|&\!\Tr(-\Delta_R^{\beta} -\lambda)_\limminus^\gamma -L_{\gamma, d}^{\rm sc}|R|\lambda^{\gamma+d/2}-\frac{1}{4}L_{\gamma, d-1}(\beta/\sqrt{ \lambda})\Haus^{d-1}(\partial R)\lambda^{\gamma+(d-1)/2} \notag\\
        &\!+\frac{1}{2}\biggl(\!L_{\gamma, d-1}(\beta/\sqrt{\lambda})|R|l_i^{-1}\lambda^{\gamma+(d-1)/2}
        -\sum_{k \geq 1}L_{\gamma, 0}\bigl(\beta(\lambda-\lambda_k(-\Delta_{R'_i}^\beta))_\limplus^{-1/2}\bigr)\bigl(\lambda-\lambda_{k}(-\Delta_{R'_i}^{\beta})\bigr)_\limplus^\gamma\biggr)\!\Biggr|  \notag\\
        &\!\lesssim_{\gamma, d} 
        l_i^{-1}|R|\lambda^{\gamma+(d-1)/2}(l_i\sqrt{\lambda})^{-\kappa_{\gamma, 1}} +  (\min_{j\neq i}l_j)^{-1}|R|\lambda^{\gamma+(d-1)/2} (\min_{j\neq i}l_j\sqrt{\lambda})^{-\kappa_{\gamma+1/2, d-1}} \, . \label{eq: induction step 2-term Weyl d dim 2}
    \end{align}
    }
    
    To complete the proof we wish to show that the terms in the parentheses in the second line of \eqref{eq: induction step 2-term Weyl d dim 2} are small. Note that $|R|/l_i = \prod_{j\neq i}l_j$, and therefore the terms in these parentheses are independent of $l_i$. We claim that, for any $i = 1, \ldots, d$ and for $\min_j l_j \sqrt{\lambda}\gtrsim 1$, 
    \begin{align}
        \Biggl|\sum_{k \geq 1}&L_{\gamma, 0}\bigl(\beta(\lambda-\lambda_k(-\Delta_{R'_i}^\beta))_\limplus^{-1/2}\bigr)\bigl(\lambda-\lambda_{k}(-\Delta_{R'_i}^{\beta})\bigr)_\limplus^\gamma -L_{\gamma, d-1}(\beta/\sqrt{\lambda})\biggl(\prod_{j\neq i}l_j\biggr)\lambda^{\gamma+(d-1)/2}\Biggr| \notag \\
        &\lesssim_{\gamma, d}
        \biggl(\prod_{j\neq i}l_j\biggr)\lambda^{\gamma+(d-1)/2}(\min_{j\neq i}l_j \sqrt{\lambda})^{-\kappa_{\gamma, d}} \, .\label{eq: second term representations agree}
    \end{align}

    Before proving the validity of \eqref{eq: second term representations agree}, we show how to complete the proof of Theorem~\ref{thm: two-term Weyl} with this estimate in hand. Combining \eqref{eq: induction step 2-term Weyl d dim 2}, \eqref{eq: second term representations agree} (for any choice of $i$), and the triangle inequality, it follows that
    \begin{align*}
        \biggl|&\Tr(-\Delta_R^{\beta} -\lambda)_\limminus^\gamma -L_{\gamma, d}^{\rm sc}|R|\lambda^{\gamma+d/2}-\frac{1}{4}L_{\gamma, d-1}(\beta/\sqrt{\lambda})\Haus^{d-1}(\partial\Omega)\lambda^{\gamma+(d-1)/2}\biggr| \\
        &\lesssim_{\gamma, d} \Haus^{d-1}(\partial R)\lambda^{\gamma+(d-1)/2}\Bigl(
        (l_i\sqrt{\lambda})^{-\kappa_{\gamma, 1}} +  (\min_{j\neq i}l_j\sqrt{\lambda})^{-\kappa_{\gamma+1/2, d-1}} + (\min_{j\neq i}l_j \sqrt{\lambda})^{-\kappa_{\gamma, d}}\Bigr)\\
        &\lesssim_{\gamma, d} \Haus^{d-1}(\partial R)\lambda^{\gamma+(d-1)/2}(\min_{j}l_j \sqrt{\lambda})^{-\kappa_{\gamma, d}}\, , 
    \end{align*}
    where we used that $l_j^{-1} |R|\leq \frac{1}{2}\Haus^{d-1}(\partial R)$ for each $j= 1, \ldots, d$, the assumption $\min_j l_j \sqrt{\lambda}\gtrsim 1$, and that $\kappa_{\gamma, d} \leq \min\{\kappa_{\gamma, 1}, \kappa_{\gamma+1/2, d-1}\}$. This is the claimed estimate and, hence, to complete the proof of Theorem~\ref{thm: two-term Weyl}, it only remains to justify \eqref{eq: second term representations agree}.
    
    \smallskip

    In proving \eqref{eq: second term representations agree} we may assume without loss of generality that $i=1$ and that $l_2 \leq l_3 \leq \ldots \leq l_d$. 
    
    Note that we have proved \eqref{eq: induction step 2-term Weyl d dim 2} above for any $l_1, \ldots, l_d>0$ under the assumption that $\min_j l_j \sqrt{\lambda}\gtrsim 1$. By applying \eqref{eq: induction step 2-term Weyl d dim 2} with $i=1$ and $i=2$, we deduce from the triangle inequality that for any $\lambda>0$ and $\beta>0$, 
    \begin{align*}
        \Biggl|&\sum_{k \geq 1}L_{\gamma, 0}\bigl(\beta(\lambda-\lambda_k(-\Delta_{R'_1}^\beta))_\limplus^{-1/2}\bigr)\bigl(\lambda-\lambda_{k}(-\Delta_{R'_1}^{\beta})\bigr)_\limplus^\gamma
        -L_{\gamma, d-1}(\beta/\sqrt{\lambda})\biggl(\prod_{j\neq 1}l_j\biggr)\lambda^{\gamma+(d-1)/2}\Biggr| \\
        &\lesssim_{\gamma, d}\sum_{k \geq 1}\bigl|L_{\gamma, 0}\bigl(\beta(\lambda-\lambda_k(-\Delta_{R'_2}^\beta))_\limplus^{-1/2}\bigr)\bigr|\bigl(\lambda-\lambda_{k}(-\Delta_{R'_2}^{\beta})\bigr)_\limplus^\gamma+|L_{\gamma, d-1}(\beta/\sqrt{\lambda})||R_2'|\lambda^{\gamma+(d-1)/2}\\
        &\quad +
        l_1^{-1}|R|\lambda^{\gamma+(d-1)/2}(l_1\sqrt{\lambda})^{-\kappa_{\gamma, 1}} +  l_2^{-1}|R|\lambda^{\gamma+(d-1)/2} (l_2\sqrt{\lambda})^{-\kappa_{\gamma+1/2, d-1}}\\
        &\quad +
        l_2^{-1}|R|\lambda^{\gamma+(d-1)/2}(l_2\sqrt{\lambda})^{-\kappa_{\gamma, 1}} +  \min\{l_1, l_3\}^{-1}|R|\lambda^{\gamma+(d-1)/2} (\min\{l_1, l_3\}\sqrt{\lambda})^{-\kappa_{\gamma+1/2, d-1}}\, .
    \end{align*}
    if $\min_j l_j \sqrt{\lambda}\gtrsim 1$. Since $L_{\gamma, 0}(\beta)\lesssim_{\gamma, d} 1$ and $\min_{j}l_j\sqrt{\lambda}\gtrsim 1$, Lemma~\ref{lem: Neumann Riesz mean bound} implies that
    \begin{align*}
        \sum_{k \geq 1}\bigl|L_{\gamma, 0}\bigl(\beta(\lambda-\lambda_k(-\Delta_{R'_2}^\beta))_\limplus^{-1/2}\bigr)\bigr|\bigl(\lambda-\lambda_{k}(-\Delta_{R'_2}^{\beta})\bigr)_\limplus^\gamma &\lesssim_{\gamma, d} \Tr(-\Delta_{R'_2}^\beta-\lambda)_\limminus^\gamma\\
        &\lesssim_{\gamma, d} |R_2'|\lambda^{\gamma+(d-1)/2}\, .
    \end{align*}
    Therefore, using again that $L_{\gamma, 0}(\beta)\lesssim_{\gamma, d} 1$, we have shown that
    \begin{align}
        \Biggl|\sum_{k \geq 1}L_{\gamma, 0}&\bigl(\beta(\lambda-\lambda_k(-\Delta_{R'_1}^\beta))_\limplus^{-1/2}\bigr)\bigl(\lambda-\lambda_{k}(-\Delta_{R'_1}^{\beta})\bigr)_\limplus^\gamma
        -L_{\gamma, d-1}(\beta/\sqrt{\lambda})\biggl(\prod_{j\neq 1}l_j\biggr)\lambda^{\gamma+(d-1)/2}\Biggr|\notag \\
        &\lesssim_{\gamma, d} \biggl(\prod_{j\neq 1}l_j\biggr)\lambda^{\gamma+(d-1)/2}\Bigl(l_1l_2^{-1} +
        (l_1\sqrt{\lambda})^{-\kappa_{\gamma, 1}} +  l_1l_2^{-1} (l_2\sqrt{\lambda})^{-\kappa_{\gamma+1/2, d-1}} \label{eq: R1 bound}\\
        &\quad +
        l_1l_2^{-1}(l_2\sqrt{\lambda})^{-\kappa_{\gamma, 1}} +  l_1\min\{l_1, l_3\}^{-1} (\min\{l_1, l_3\}\sqrt{\lambda})^{-\kappa_{\gamma+1/2, d-1}}\Bigr)\, .\notag
    \end{align}
    
    Note that the expression on the left-hand side of \eqref{eq: R1 bound} is independent of $l_1$. Since the proof of this bound is valid for any $l_1>0$, we are free to choose $l_1$ in order to make the right-hand side small. Specifically, we can take $l_1 = l_2 /(l_2 \sqrt{\lambda})^\alpha$ for some $\alpha \in (0, 1)$ to be chosen later. By the assumptions $l_2\sqrt{\lambda}\gtrsim 1$, $l_2\leq l_3$, we have $l_2 /(l_2 \sqrt{\lambda})^\alpha\lesssim l_3$. This yields
    \begin{align*}
        \Biggl|\sum_{k \geq 1}L_{\gamma, 0}&\bigl(\beta(\lambda-\lambda_k(-\Delta_{R'_1}^\beta))_\limplus^{-1/2}\bigr)\bigl(\lambda-\lambda_{k}(-\Delta_{R'_1}^{\beta})\bigr)_\limplus^\gamma
        -L_{\gamma, d-1}(\beta/\sqrt{\lambda})\biggl(\prod_{j\neq 1}l_j\biggr)\lambda^{\gamma+(d-1)/2}\Biggr| \\
        &\lesssim_{\gamma, d} \biggl(\prod_{j\neq 1}l_j\biggr)\lambda^{\gamma+(d-1)/2}\Bigl((l_2\sqrt{\lambda})^{-\alpha} +
        (l_2\sqrt{\lambda})^{-(1-\alpha)\kappa_{\gamma, 1}} +  (l_2\sqrt{\lambda})^{-\kappa_{\gamma+1/2, d-1}-\alpha}\\
        &\quad +
        (l_2\sqrt{\lambda})^{-\kappa_{\gamma, 1}-\alpha} +   (l_2\sqrt{\lambda})^{-(1-\alpha)\kappa_{\gamma+1/2, d-1}}\Bigr)\, .
    \end{align*}
    Since $l_2\sqrt{\lambda}\gtrsim 1$, we arrive at
    \begin{align*}
        \Biggl|\sum_{j \geq 1}&L_{\gamma, 0}\bigl(\beta(\lambda-\lambda_j(-\Delta_{R'_1}^\beta))_\limplus^{-1/2}\bigr)\bigl(\lambda-\lambda_{j}(-\Delta_{R'_1}^{\beta})\bigr)_\limplus^\gamma
        -L_{\gamma, d-1}(\beta/\sqrt{\lambda})\biggl(\prod_{i\neq 1}l_i\biggr)\lambda^{\gamma+(d-1)/2}\Biggr|\\
        &\lesssim_{\gamma, d}
        \biggl(\prod_{j \neq 1 }l_j\biggr)\lambda^{\gamma+(d-1)/2}(l_2\sqrt{\lambda})^{-\tilde\kappa}\, 
    \end{align*}
    where
    \begin{align*}
        \tilde\kappa &= \min\bigl\{\alpha, (1-\alpha)\kappa_{\gamma, 1}, \kappa_{\gamma+1/2, d-1}+\alpha, \kappa_{\gamma, 1}+\alpha, (1-\alpha)\kappa_{\gamma+1/2, d-1}\bigr\}\\
        &= \min\bigl\{\alpha, (1-\alpha)\min\{\kappa_{\gamma, 1}, \kappa_{\gamma+1/2, d-1}\}\bigr\}.
    \end{align*}
    We claim that $\min\{\kappa_{\gamma, 1}, \kappa_{\gamma+1/2, d-1}\} =\min\bigl\{\gamma, \frac{1}{d-1}\bigr\}$ for $d\geq 2$. Indeed, if $d=2$, then $$\min\{\kappa_{\gamma, 1}, \kappa_{\gamma+1/2, d-1}\} =\min\{\gamma, \gamma+1/2, 1\}=\min\{\gamma, 1\}$$ as claimed. If instead $d\geq 3$, we find that 
    $$\min\{\kappa_{\gamma, 1}, \kappa_{\gamma+1/2, d-1}\}= \min\Bigl\{1, \gamma, \frac{\gamma+1/2}{\gamma+3/2}, \frac{1}{d-1}\Bigr\} = \min\Bigl\{\gamma, \frac{1}{d-1}\Bigr\}\, , $$ 
    where the last step used that $\gamma \leq (\gamma+1/2)/(\gamma+3/2)$ for $\gamma\leq 1/(d-1)\leq 1/2$. Therefore, 
    \begin{align*}
    \tilde\kappa= \min\Bigl\{\alpha, (1-\alpha)\min\Bigl\{\gamma, \frac{1}{d-1}\Bigr\}\Bigr\}\, .
    \end{align*}
    Optimizing $\tilde \kappa$ with respect to the choice of $\alpha \in (0, 1)$ leads to
    \begin{equation*}
        \tilde \kappa = \alpha = \frac{\min\{\gamma, \frac{1}{d-1}\}}{1+\min\{\gamma, \frac{1}{d-1}\}} = \min\Bigr\{\frac{\gamma}{1+\gamma}, \frac{1}{d}\Bigr\}=\kappa_{\gamma, d}\, .
    \end{equation*}
    This completes the proof of \eqref{eq: second term representations agree} and thus also the proof of Theorem~\ref{thm: two-term Weyl}.
\end{proof}

Theorem~\ref{thm: two-term Weyl} provides a quantitative error estimate in a two-term Weyl law as soon as $\gamma>0$. Later in the paper we require a corresponding error estimate for the leading order Weyl law also when $\gamma=0$. This is the content of the next lemma.
\begin{lemma}\label{lem: quantitative Weyl law}
    Let $d\geq 1, \gamma \geq 0$. There exists a constant $C_{\gamma, d}>0$ so that for any $\lambda \geq 0$, $\sharp \in \{\rm D, \rm N, \beta\}$ with $\beta>0$, and any cuboid $R\subset \R^d$ with side lengths $l_1, \ldots, l_d>0$, 
    \begin{equation*}
        |\Tr(-\Delta_R^{\sharp}-\lambda)_\limminus^\gamma-L_{\gamma, d}^{\rm sc}|R|\lambda^{\gamma+d/2}|\leq C_{ \gamma, d}\Haus^{d-1}(\partial R)\lambda^{\gamma+(d-1)/2}(1+(\min_i l_i \sqrt{\lambda})^{1-d})\, .
    \end{equation*}
\end{lemma}

\begin{proof}
    Since 
    \begin{equation*}
        \Tr(-\Delta_{R}^{\rm D}-\lambda)_\limminus^\gamma \leq \Tr(-\Delta_{R}^\beta-\lambda)_\limminus^\gamma \leq \Tr(-\Delta_R^{\rm N}-\lambda)_\limminus^\gamma\, , 
    \end{equation*}
    it suffices to prove the claimed statement for $\sharp\in \{\rm D, \rm N\}$.

    For $d=1$, let $R=(0, l_1)$ with $l_1 >0$. We know from the explicit formulas for the Dirichlet and Neumann eigenvalues that
    \begin{equation}\label{eq: d=1 Weyl counting bound}
        |\Tr(-\Delta_{(0, l_1)}^{\sharp}-\lambda)_\limminus^0 -L_{0, 1}^{\rm sc} l_1 \lambda^{1/2}| \leq 1\, .
    \end{equation}
    By the Aizenman--Lieb principle, we have for $\gamma >0$ that
    \begin{align*}
        \Tr(-\Delta_{(0, l_1)}^\sharp-\lambda)_\limminus^\gamma= L_{\gamma, 1}^{\rm sc}l_1 \lambda^{\gamma+1/2}+ \gamma\int_0^\lambda (\lambda-\tau)^{\gamma-1}\big(\Tr(-\Delta_{(0, l_1)}^\sharp-\tau)_\limminus^0-L_{0, 1}^{\rm sc} l _ 1\tau^{1/2}\big)\, d\tau\, .
    \end{align*}
    It thus follows from \eqref{eq: d=1 Weyl counting bound} that 
    \begin{equation}\label{eq: d=1 Weyl Riesz bound}
        |\Tr(-\Delta_{(0, l_1)}^{\sharp}-\lambda)_\limminus^\gamma-L_{\gamma, 1}^{\rm sc}l_1 \lambda^{\gamma+1/2}|\lesssim_{\gamma}\lambda^{\gamma}
    \end{equation}
    for all $\lambda \geq 0, \gamma \geq 0$. This completes the proof for $d=1$.

    To prove the claim for $d\geq 2$, we argue by induction in $d$, the base case being the bounds for $d=1$ that we have just shown. Fix $d\geq 2, \gamma \geq 0$ and assume that we have proved the lemma in all lower dimensions and all orders of Riesz means. We first use the product structure of $R = (0, l_1)\times R'$ and the triangle inequality to obtain the estimate
    \begin{equation} \label{eq: lem3-4_split}
    \begin{aligned}
             |\Tr(-\Delta_R^{\sharp}-\lambda)_\limminus^\gamma-L_{\gamma, d}^{\rm sc}|R|\lambda^{\gamma+d/2}| &\leq  |\Tr(-\Delta_R^\sharp-\lambda)_\limminus^\gamma-L_{\gamma, 1}^{\rm sc}l_1 \Tr(-\Delta_{R'}^\sharp-\lambda)_\limminus^{\gamma+1/2}|\\
         & \qquad +|L_{\gamma, 1}^{\rm sc}l_1 \Tr(-\Delta_{R'}^\sharp-\lambda)_\limminus^{\gamma+1/2}-L_{\gamma, d}^{\rm sc}|R|\lambda^{\gamma+d/2}|
    \end{aligned}
    \end{equation}
    With the triangle inequality and \eqref{eq: d=1 Weyl Riesz bound}, we can bound the first term now by
    \begin{align*}
        |\Tr(-&\Delta_R^\sharp-\lambda)_\limminus^\gamma-L_{\gamma, 1}^{\rm sc}l_1 \Tr(-\Delta_{R'}^\sharp-\lambda)_\limminus^{\gamma+1/2}|\\
        &= \biggl|\sum_{k \geq 1}\Tr(-\Delta_{(0, l_1)}^\sharp-(\lambda-\lambda_k(-\Delta_{R'}^\sharp))_\limplus)_\limminus^\gamma-L_{\gamma, 1}^{\rm sc}l_1 \Tr(-\Delta_{R'}^\sharp-\lambda)_\limminus^{\gamma+1/2}\biggr|\\
        &\lesssim_{\gamma}\Tr(-\Delta_{R'}^\sharp-\lambda)_\limminus^\gamma\, .
    \end{align*}
    For the Riesz mean appearing on the right-hand side Lemma~\ref{lem: Neumann Riesz mean bound} yields
    \begin{align*}
        \Tr(-\Delta_{R'}^\sharp-\lambda)_\limminus^\gamma&\lesssim_d |R'|\lambda^{\gamma+(d-1)/2}(1+(\min_i l_i \sqrt{\lambda})^{1-d})\,.
    \end{align*}
    By the induction hypothesis, the second term on the right-hand side of \eqref{eq: lem3-4_split} satisfies
    \begin{align*}
        |\Tr(-\Delta_{R'}^\sharp&-\lambda)_\limminus^{\gamma+1/2}-L_{\gamma+1/2, d-1}^{\rm sc}|R'|\lambda^{\gamma+d/2}| \\
        &\lesssim_{\gamma, d} \Haus^{d-2}(\partial R') \lambda^{\gamma+(d-1)/2}(1+(\min_i l_i \sqrt{\lambda})^{2-d})\, .
    \end{align*}
    Combining the two previous estimates with the facts that $L_{\gamma, 1}^{\rm sc}L_{\gamma+1/2, d-1}^{\rm sc}=L_{\gamma, d}^{\rm sc}$ and $|R|=l_1|R'|$, we conclude that
     \begin{align*}
        |\Tr(-\Delta_R^\sharp-\lambda)_\limminus^\gamma-L_{\gamma, d}^{\rm sc}|R|\lambda^{\gamma+d/2}| &\lesssim_{\gamma, d} |R'|\lambda^{\gamma+(d-1)/2}(1+(\min_i l_i \sqrt{\lambda})^{1-d}) \\
        & \qquad + l_1 \Haus^{d-2}(\partial R')\lambda^{\gamma+(d-1)/2}(1+(\min_i l_i \sqrt{\lambda})^{2-d})\, .
    \end{align*} 
    We finally note that $(\min_{i}l_i \sqrt{\lambda})^{2-d}\leq 1+ (\min_{i}l_i \sqrt{\lambda})^{1-d}$ and $\Haus^{d-1}(\partial R) = l_1 \Haus^{d-2}(\partial R') + 2 |R'|$, thus giving the estimates $|R'| \leq \Haus^{d-1}(\partial R)$ and $l_1 \Haus^{d-2}(\partial R')\leq \Haus^{d-1}(\partial R) $. Using these estimates completes the proof of the lemma.
\end{proof}

\section{Asymptotics along collapsing sequences of cuboids}
\label{sec: collapsing cuboids}

The next lemma will be an important ingredient in our analysis of the asymptotic shape optimization problem. Specifically, the lemma will allow us to understand the asymptotic behavior of Riesz means along sequences of $\lambda_j$ tending to infinity and cuboids $R_j$ that fail to be precompact. 

\begin{lemma}\label{lem: collapsing Weyl law}
    Let $d\geq 2$, $\gamma \geq 0$. Let $\{\lambda_j\}_{j\geq 1}$, $\{\beta_j\}_{j\geq 1}$ be positive sequences and $\{R_j\}_{j\geq 1}$ be a sequence of cuboids in $\R^d$ with side lengths $l_1^{(j)}, \ldots, l_d^{(j)}>0$. Assume that
    \begin{equation*}
       \lim_{j\to \infty}|R_j|^{2/d}\lambda_j=\infty\, , \qquad \lim_{j\to \infty} \beta_j /\sqrt{\lambda_j}=\beta'\geq 0\, , 
    \end{equation*}
    and
    \begin{equation*}
        0 < \inf_{j\geq 1}\min_i l_i^{(j)} \sqrt{\lambda_j} \leq \sup_{j\geq 1} \min_i l_i^{(j)}\sqrt{\lambda_j}<\infty\, .
    \end{equation*}
    Then there exists an integer $m \in \{1, \ldots, d-1\}$ and a cuboid $R' \subset \R^m$ such that
    \begin{equation*}
        \limsup_{j\to \infty} \frac{\Tr(-\Delta_{R_j}^{\beta_j}-\lambda_j)_\limminus^\gamma}{L_{\gamma, d}^{\rm sc}|R_j|\lambda_j^{\gamma+d/2}} = \frac{\Tr(-\Delta_{R'}^{\beta'}-1)^{\gamma+(d-m)/2}_\limminus}{L_{\gamma+(d-m)/2, m}^{\rm sc}|R'|}\, .
    \end{equation*}
    If $\lim_{j\to \infty}\beta_j/\sqrt{\lambda_j}=\infty$ then the corresponding statement holds with $-\Delta_{R'}^{\beta'}$ on the right-hand side replaced by $-\Delta_{R'}^{\rm D}$. 
    
    Furthermore, if there is an integer $m' \in \{1, \ldots, d-1\}$ such that 
    \begin{equation*}
        \lim_{j\to \infty}l_i^{(j)}\sqrt{\lambda_j} = l_i' \quad \mbox{for }i \leq m' \quad \mbox{and} \quad \lim_{j\to \infty}l_i^{(j)}\sqrt{\lambda_j} = \infty \quad \mbox{for }i >m'\, , 
    \end{equation*}
    then the claim holds with $m=m'$ and $R' = \prod_{i=1}^m (0, l_i')$.
\end{lemma}

\begin{proof}
    By the invariance of the Laplacian under orthogonal transformations we may without loss of generality assume that $0<l_1^{(j)}\leq l_2^{(j)} \leq \ldots \leq l_d^{(j)}$, for each $j\geq 1$. 
    
    We next reduce to the case $|R_j|=1$ for all $j\geq 1$. If $\lambda_j, \beta_j, R_j$ satisfy the assumptions of the lemma then defining $\tilde \lambda_j=|R_j|^{2/d}\lambda_j, \tilde \beta_j = \beta_j |R_j|^{1/d}, $ and $\tilde R_j = |R_j|^{-1/d}R_j$ these new sequences also satisfy the assumptions with the same values of $\beta', l_i'$, since for each $j\geq 1$ we have
    \begin{equation*}
        |\tilde R_j|^{2/d}\tilde\lambda_j=|R_j|^{2/d}\lambda_j\, , \quad \tilde\beta_j /\sqrt{\tilde\lambda_j}=\beta_j/\sqrt{\lambda_j}\, , \quad \tilde l_i^{(j)}\sqrt{\tilde \lambda_j} =l_i^{(j)}\sqrt{\lambda_j} \, .
    \end{equation*}
    Furthermore, by the behavior of Robin eigenvalues under rescaling it holds that
    \begin{equation*}
        \frac{\Tr(-\Delta_{R_j}^{\beta_j}-\lambda_j)_\limminus^\gamma}{L_{\gamma, d}^{\rm sc}|R_j|\lambda_j^{\gamma+d/2}}= \frac{\Tr(-\Delta_{\tilde R_j}^{\tilde \beta_j}-\tilde \lambda_j)_\limminus^\gamma}{L_{\gamma, d}^{\rm sc}|\tilde R_j|\tilde\lambda_j^{\gamma+d/2}}\, .
    \end{equation*}
    Consequently, we may assume without loss of generality that $|R_j|=1$ for each $j\geq 1$.
    
    By passing to a subsequence we may assume without loss of generality that
    \begin{align*}
        \limsup_{j\to \infty} \frac{\Tr(-\Delta_{R_j}^{\beta_j}-\lambda_j)_\limminus^\gamma}{L_{\gamma, d}^{\rm sc}\lambda_j^{\gamma+d/2}} &= \lim_{j\to \infty} \frac{\Tr(-\Delta_{R_j}^{\beta_j}-\lambda_j)_\limminus^\gamma}{L_{\gamma, d}^{\rm sc}\lambda_j^{\gamma+d/2}}\, , 
    \end{align*}
    and that there exists an integer $m\in \{1, \ldots, d\}$ such that
    \begin{equation*}
        \lim_{j\to \infty}l_i^{(j)}\sqrt{\lambda_j} = l_i'>0 \quad \mbox{for }i \leq m
    \qquad
    \mbox{and}
    \qquad 
        \lim_{j\to \infty}l_i^{(j)}\sqrt{\lambda_j} = \infty \quad \mbox{for }i > m\, .
    \end{equation*}
    Note that since $|R_j| = \prod_{i=1}^d l_i^{(j)}=1$ for each $j\geq 1$, it follows from the assumption that $\sup_{j\geq 1}l_{1}^{(j)}\sqrt{\lambda_j}<\infty$ that we must have $\lim_{j\to \infty}l_d^{(j)}\sqrt{\lambda_j}=\infty$. We shall prove that the claimed limit holds with $R'=\prod_{i=1}^m (0, l_i')$. Note that this will justify also the second statement of the lemma.

    We define $R_j^\parallel=\prod_{i=1}^m(0, l_i^{(j)})$ and $R_j^\perp=\prod_{i=m+1}^d(0, l_i^{(j)})$.
    By the product structure of the cuboid, we have
    \begin{align*}
        \Tr(-\Delta_{R_j}^{\beta_j}-\lambda_j)_\limminus^\gamma = \sum_{k\geq 1}\Tr(-\Delta_{R_j^\perp}^{\beta_j}-(\lambda_j -\lambda_k(-\Delta_{R_j^{\parallel}}))_\limplus)^\gamma_\limminus \, .
    \end{align*}
    Applying Lemma~\ref{lem: quantitative Weyl law} to the Riesz means on the right-hand side, we obtain
    \begin{align*}
        &\Bigl| \Tr(-\Delta_{R_j}^{\beta_j}-\lambda_j)_\limminus^\gamma - L_{\gamma, d-m}^{\rm sc} |R_j^\perp| \Tr(-\Delta_{R_j^\parallel}^{\beta_j}-\lambda_j)_\limminus^{\gamma+(d-m)/2}  \Bigr| \\
        &\quad \lesssim_{\gamma, d, m} \Haus^{d-m-1}(\partial R_j^\perp) \Big( \Tr(-\Delta_{R_j^\parallel}^{\beta_j}-\lambda_j)_\limminus^{\gamma+(d-m-1)/2} + (l_{m+1}^{(j)})^{-(d-m-1)} \Tr(-\Delta_{R_j^\parallel}^{\beta_j}-\lambda_j)_\limminus^{\gamma} \Big) \, . 
    \end{align*}
    Dividing the above inequality by $L_{\gamma, d}^{\rm sc} |R_j| \lambda_j^{\gamma+d/2}$ and using $L_{\gamma, d}^{\rm sc}=L_{\gamma, d-m}^{\rm sc}L_{\gamma+(d-m)/2, m}^{\rm sc}$ and $|R_j| = |R_j^\parallel| \, |R_j^\perp|$ gives 
    \begin{equation} \label{eq: proof4-1estim}
    \begin{aligned}
         &\Biggl| \frac{ \Tr(-\Delta_{R_j}^{\beta_j}-\lambda_j)_\limminus^\gamma }{L_{\gamma, d}^{\rm sc} |R_j| \lambda_j^{\gamma+d/2}} - \frac{  \Tr(-\Delta_{R_j^\parallel}^{\beta_j}-\lambda_j)_\limminus^{\gamma+(d-m)/2}}{L_{\gamma+(d-m)/2, m}^{\rm sc} |R_j^\parallel| \lambda_j^{\gamma+d/2}}  \Biggr| \\
        &\hspace{50pt} \lesssim_{\gamma, d, m} \frac{\Haus^{d-m-1}(\partial R_j^\perp)}{|R_j| \lambda_j^{\gamma+d/2}} \Bigl( \Tr(-\Delta_{R_j^\parallel}^{\beta_j}-\lambda_j)_\limminus^{\gamma+(d-m-1)/2} \\
        & \hspace{165pt}+ (l_{m+1}^{(j)})^{-(d-m-1)} \Tr(-\Delta_{R_j^\parallel}^{\beta_j}-\lambda_j)_\limminus^{\gamma} \Bigr) \, .
    \end{aligned}
    \end{equation}
    Let us first consider the second term on the left-hand side. Defining $ R_j ' = \sqrt{\lambda_j}R_j^\parallel$ and $\beta_j' = \beta_j/\sqrt{\lambda_j}$ and using the scaling of Robin eigenvalues, we have for any $\gamma>0$
    \begin{align*}
        \Tr(-\Delta_{R_j^\parallel}^{\beta_j}-\lambda_j)_\limminus^{\gamma} = \lambda_j^\gamma \, \Tr(-\Delta_{R_j'}^{\beta_j'}-1)_\limminus^{\gamma}, 
    \end{align*}
    therefore, 
    \begin{align*}
        \frac{  \Tr(-\Delta_{R_j^\parallel}^{\beta_j}-\lambda_j)_\limminus^{\gamma+(d-m)/2}}{L_{\gamma+(d-m)/2, m}^{\rm sc} |R_j^\parallel| \lambda_j^{\gamma+d/2}}  = \frac{\Tr(-\Delta_{R_j'}^{\beta_j '}-1)_\limminus^{\gamma+(d-m)/2}}{L_{\gamma+(d-m)/2, m}|R_j'|}.   
    \end{align*}
    The side lengths of $R_j'$ are given by $l_i^{(j)}\sqrt{\lambda_j}$, for $i=1, \ldots, m$. By our definition of $m$ and $l_i'$, $R_j'$ converges to the $m$-dimensional cuboid $R' = \prod_{i=1}^m (0, l_i')$. Since furthermore $\gamma+(d-m)/2>0$, the continuity of the individual eigenvalues with respect to the cuboid and the Robin parameter yields 
    \begin{equation} \label{eq: proof4-1_rescaledlimit}
    \begin{aligned}
                \lim_{j\to \infty}\frac{\Tr(-\Delta_{R_j'}^{\beta_j '}-1)_\limminus^{\gamma+(d-m)/2}}{L_{\gamma+(d-m)/2, m}|R_j'|}= 
        \begin{cases}
        \displaystyle\frac{\Tr(-\Delta_{R'}^{\beta'}-1)_\limminus^{\gamma+(d-m)/2}}{L_{\gamma+(d-m)/2, m}^{\rm sc}|R'|} &\mbox{if }\lim_{j\to \infty}\beta_j' = \beta'\geq 0\, , \\[12pt]
        \displaystyle\frac{\Tr(-\Delta_{R'}^{\rm D}-1)_\limminus^{\gamma+(d-m)/2}}{L_{\gamma+(d-m)/2, m}^{\rm sc}|R'|} &\mbox{if }\lim_{j\to \infty}\beta_j' = \infty\, , \\
        \end{cases}
    \end{aligned}
    \end{equation}
    Hence, the statement of the lemma follows, if we can show that the right-hand side of \eqref{eq: proof4-1estim} goes to zero as $j \to \infty$. To prove that this is the case we rewrite the expression as
    \begin{align*}
        \frac{ \Haus^{d-m-1}(\partial R_j^\perp)}{ |R_j| \lambda_j^{\gamma+d/2}} \Bigl( \Tr(-\Delta_{R_j^\parallel}^{\beta_j}-\lambda_j)_\limminus^{\gamma+(d-m-1)/2} + (l_{m+1}^{(j)})^{-(d-m-1)} \Tr(-\Delta_{R_j^\parallel}^{\beta_j}-\lambda_j)_\limminus^{\gamma} \Bigr) \\
        =\frac{ \Haus^{d-m-1}(\partial R_j^\perp)}{ |R_j^\perp|} \Biggl(  \frac{\Tr(-\Delta_{R_j'}^{\beta_j'}-1)_\limminus^{\gamma+(d-m-1)/2}  }{|R_j'|  \lambda_j^{1/2} }+ \frac{ \Tr(-\Delta_{R_j'}^{\beta_j'}-1)_\limminus^{\gamma} }{|R_j'|(l_{m+1}^{(j)})^{d-m-1}  \lambda_j^{(d-m)/2}} \Biggr) \, .
    \end{align*}
    From the estimate
    \begin{align*}
        \frac{ \Haus^{d-m-1}(\partial R_j^\perp)}{ |R_j^\perp|} = 2 \sum_{i=m+1}^d \frac{1}{l_{i}^{(j)}} \leq \frac{2(d-m)}{l_{m+1}^{(j)}}
    \end{align*}
    one concludes that
    \begin{align*}
        &\frac{ \Haus^{d-m-1}(\partial R_j^\perp)}{ |R_j| \lambda_j^{\gamma+d/2}} \Bigl( \Tr(-\Delta_{R_j^\parallel}^{\beta_j}-\lambda_j)_\limminus^{\gamma+(d-m-1)/2} + (l_{m+1}^{(j)})^{-(d-m-1)} \Tr(-\Delta_{R_j^\parallel}^{\beta_j}-\lambda_j)_\limminus^{\gamma} \Bigr) \\
       \qquad &\leq 2(d-m) \Biggl(  \frac{\Tr(-\Delta_{R_j'}^{\beta_j'}-1)_\limminus^{\gamma+(d-m-1)/2}  }{|R_j'| l_{m+1}^{(j)} \lambda_j^{1/2} }+ \frac{ \Tr(-\Delta_{R_j'}^{\beta_j'}-1)_\limminus^{\gamma} }{|R_j'| (l_{m+1}^{(j)}  \lambda_j^{1/2} )^{(d-m)}} \Biggr) \, .
    \end{align*}
    Since $\lim_{j\to \infty} l_{m+1}^{(j)}  \sqrt{\lambda_j} =\infty$ and the terms of the form $\Tr(-\Delta_{R_j'}^{\beta_j'}-1)_\limminus^{\gamma} / |R_j'|$ converge as in~\eqref{eq: proof4-1_rescaledlimit}, we conclude that this expression tends to zero as $j \to \infty$. This completes the proof.
\end{proof}

While Lemma~\ref{lem: collapsing Weyl law} allows us to understand the behavior of Riesz means along a sequence of cuboids that collapses at a rate proportional to the characteristic wave length $1/\sqrt{\lambda}$, we also need to understand what happens along sequences collapsing at an even faster rate. To this end we shall prove the following lemma.
\begin{lemma}\label{lem: super-critical collapse}
     Let $d\geq 2$, $\gamma \geq 0$. Let $\{\lambda_j\}_{j\geq 1}$, $\{\beta_j\}_{j\geq 1}$ be positive sequences, and $\{R_j\}_{j\geq 1}$ be a sequence of cuboids in $\R^d$ with side lengths $l_1^{(j)}, \ldots, l_d^{(j)}>0$. Assume that
    \begin{equation*}
       \lim_{j\to \infty}|R_j|^{2/d}\lambda_j=\infty\, , \quad \lim_{j\to \infty}\min_i l_i^{(j)} \sqrt{\lambda_j} =0\, , \qquad \mbox{and}\qquad \liminf_{j\to \infty} \frac{\beta_j}{\min_i l_i^{(j)}\lambda_j}\geq\frac{1}{2}\, .
    \end{equation*}
    Then
    \begin{equation*}
        \lim_{j\to \infty} \frac{\Tr(-\Delta_{R_j}^{\beta_j}-\lambda_j)_\limminus^\gamma}{L_{\gamma, d}^{\rm sc}|R_j|\lambda_j^{\gamma+d/2}} = 0\, .
    \end{equation*}
\end{lemma}
\begin{remark}
    The assumption that $\liminf_{j\to \infty} \beta_j/(\min_i l_i^{(j)}\lambda_j)\geq 1/2$ is crucial. Indeed, we shall see later on that we can construct sequences of cuboids $\{R_j\}_{j\geq 1}$ that satisfy all the assumptions of the lemma but $\lim_{j\to \infty} \beta_j/(\min_i l_i^{(j)}\lambda_j)< 1/2$ for which
    \begin{equation*}
        \lim_{j\to \infty} \frac{\Tr(-\Delta_{R_j}^{\beta_j}-\lambda_j)_\limminus^\gamma}{L_{\gamma, d}^{\rm sc}|R_j|\lambda_j^{\gamma+d/2}} = \infty\, , 
    \end{equation*}
    this is done in the proof of Theorem~\ref{thm: asymptotics of M} \ref{itm: almost Neumann}.
\end{remark}

\begin{proof}[Proof of Lemma~\ref{lem: super-critical collapse}]
    Without loss of generality we may assume that $\min_i l_i^{(j)}= l_1^{(j)}$ for each $j\geq 1$. By the same rescaling argument as in the proof of Lemma~\ref{lem: collapsing Weyl law} we may assume without loss of generality that $|R_j|=1$ for each $j\geq 1$. Define $\tilde \beta_j = \min\{\beta_j, \sqrt{\lambda_j}\}$. Since Riesz means are non-negative and increase as the Robin parameter is decreasing we have 
    \begin{equation*}
        0 \leq \frac{\Tr(-\Delta_{R_j}^{\beta_j}-\lambda_j)_\limminus^\gamma}{L_{\gamma, d}^{\rm sc}\lambda_j^{\gamma+d/2}} \leq \frac{\Tr(-\Delta_{R_j}^{\tilde\beta_j}-\lambda_j)_\limminus^\gamma}{L_{\gamma, d}^{\rm sc}\lambda_j^{\gamma+d/2}}
    \end{equation*}
    for each $j \geq 1$ and by construction
    \begin{equation*}
        \liminf_{j\to \infty} \frac{\tilde \beta_j}{l_1^{(j)}\lambda_j} \geq\frac{1}{2} \qquad \mbox{and}\qquad \lim_{j\to \infty}l_1^{(j)}\tilde \beta_j =0\, .
    \end{equation*}
    By scaling of the Robin eigenvalues and Lemma~\ref{lem: lambda1 asymptotics}, we obtain 
    \begin{equation*}
        \lambda_1(-\Delta_{R_j}^{\tilde\beta_j}) = \sum_{i=1}^d \frac{\lambda_1(-\Delta_{(0, 1)}^{l_i^{(j)}\tilde\beta_j})}{(l_{i}^{(j)})^{2}}> \frac{\lambda_1(-\Delta_{(0, 1)}^{l_1^{(j)}\tilde\beta_j})}{(l_{1}^{(j)})^{2}} = \frac{2\tilde\beta_j}{l_1^{(j)}}(1 + O(l_1^{(j)}\tilde\beta_j))\, 
    \end{equation*}
    as $j \to \infty$. Consequently, by our assumptions, it holds that $\lambda_1(-\Delta_{R_j}^{\tilde\beta_j})> \lambda_j$ for all large enough $j$ and thus for all such $j$ we have 
    $$
        \Tr(-\Delta_{R_j}^{\beta_j}-\lambda_j)_\limminus^\gamma= \Tr(-\Delta_{R_j}^{\tilde \beta_j}-\lambda_j)_\limminus^\gamma=0\, .
    $$
    This completes the proof of Lemma~\ref{lem: super-critical collapse}.
\end{proof}

\section{Semiclassical inequalities for Robin Laplace operators}
\label{sec: uniform inequalities}

In this section we turn our attention to uniform inequalities for $\Tr(-\Delta_{R}^{\beta\sqrt{\lambda}}-\lambda)_\limminus^\gamma$. Specifically, we are interested in when these Riesz means are bounded from above by the Weyl term $L_{\gamma, d}^{\rm sc}|R|\lambda^{\gamma+d/2}$. Our main goal is to prove Theorem~\ref{thm: Robin BLY intro}, before we do this we introduce some additional notation which will be of use here and in what follows.

For $\gamma \geq 0$, $d\geq 1$ and $\beta>0$, we define
\begin{equation*}
    r_{\gamma, d}(\beta) := \sup\biggl\{ \frac{\Tr(-\Delta_{R}^{\beta\sqrt{\lambda}}-\lambda)_\limminus^\gamma}{L_{\gamma, d}^{\rm sc}|R|\lambda^{\gamma+d/2}}: R \subset \R^d \mbox{ a cuboid and }\lambda >0\biggr\}\, .
\end{equation*}
Since the Riesz means are non-increasing functions of $\beta$ it follows that $\beta\mapsto r_{\gamma, d}(\beta)$ is non-increasing. For $\gamma>0$, $r_{\gamma, d}(\beta)$ is a supremum of a family of continuous functions and therefore $\beta \mapsto r_{\gamma, d}(\beta)$ is lower semicontinuous. A consequence of Weyl's law is that $r_{\gamma, d}(\beta)\geq 1$. 

In terms of $r_{\gamma, d}(\beta)$, Theorem~\ref{thm: Robin BLY intro} states that there exists a $\beta$ sufficiently large such that $r_{\gamma, d}(\beta)=1$. The next result is a slight strengthened version of Theorem~\ref{thm: Robin BLY intro}.
\begin{theorem}\label{thm: Robin BLY}
    Let $d\geq 1, \gamma>0$. There exists a $\beta(\gamma, d)>0$ such that $r_{\gamma, d}(\beta)=1$ if and only if $\beta \geq \beta(\gamma, d)$. Moreover, if $R \subset \R^d$ is a cuboid, $\beta >\beta(\gamma, d)$, and $\lambda>0$, then
    \begin{equation*}
        \Tr(-\Delta_R^{\beta \sqrt{\lambda}}-\lambda)_\limminus^\gamma < L_{\gamma, d}^{\rm sc}|R|\lambda^{\gamma+d/2}\, .
    \end{equation*}
\end{theorem}
\begin{remark}
    A couple of remarks are in order.
    \begin{enumerate}
    \item The conclusion of Theorem~\ref{thm: Robin BLY} fails for $d=1$ and $\gamma=0$. Indeed, $\lambda_k(-\Delta_{(0, 1)}^\beta) < \lambda_k(-\Delta_{(0, 1)}^{\rm D})= \pi^2k^2$ for any $\beta>0$ and $k\geq 1$. Therefore, for $\lambda \in [\lambda_k(-\Delta_{(0, 1)}^{\beta}), \pi^2 k^2)$ we have $\Tr(-\Delta_{(0, 1)}^{\beta}-\lambda)_\limminus^0 = k >\sqrt{\lambda}/\pi=L_{0, 1}^{\rm sc}\sqrt{\lambda}$.
    \item From the two-term Weyl asymptotics for $\Tr(-\Delta_{R}^{\beta\sqrt{\lambda}}-\lambda)_\limminus^\gamma$ it is clear that we must have $\beta(\gamma, d)\geq \beta_W(\gamma, d-1)$ as otherwise the inequality would fail as $\lambda \to \infty$. As discussed in the introduction, the conjecture that this is the only obstruction to the validity of the inequality, i.e.\ that $\beta(\gamma, d) = \beta_W(\gamma, d-1)$, turns out to be too naive (see Section~\ref{sec: the critical beta are different}). 
    \end{enumerate}
\end{remark}

Before turning to our proof of Theorem~\ref{thm: Robin BLY} we shall deduce some basic properties for the $r_{\gamma, d}(\beta)$. In  particular, we shall show that $r_{\gamma, d}(\beta)$ is always finite.

\begin{lemma}\label{lem: r finite}
    For any $\gamma\geq 0$, $d\geq 1$ and $\beta>0$, we have $r_{\gamma, d}(\beta)<\infty$.
\end{lemma}

\begin{proof}
   By Lemma~\ref{lem: Neumann Riesz mean bound}, 
   \begin{equation*}
        \Tr(-\Delta_R^{\beta \sqrt{\lambda}}-\lambda)_\limminus^{\gamma}\lesssim_d |R|\lambda^{\gamma+d/2}(1+(l_1 \sqrt{\lambda})^{1-d})\, , 
   \end{equation*}
   for any $\lambda >0$ and cuboid $R\subset \R^d$ with shortest side length $l_1$.
   To prove that $r_{\gamma, d}(\beta)<\infty$ it therefore remains to consider the case when $l_1 \sqrt{\lambda}$ is small.

   By Lemma~\ref{lem: lambda1 asymptotics}
   \begin{equation*}
        \lambda_1(-\Delta_{R}^{\beta \sqrt{\lambda}})\geq \lambda_1(-\Delta_{(0, l_1)}^{\beta\sqrt{\lambda}}) = l_1^{-2}\lambda_1(-\Delta_{(0, 1)}^{l_1\beta\sqrt{\lambda}}) = \frac{2\beta}{ l_1 \sqrt{\lambda} } \, \lambda \, (1+o(1))
   \end{equation*}
   as $l_1 \sqrt{\lambda} \to 0$. Therefore, for $l_1\sqrt{\lambda}$ small enough, we have $\lambda_1(-\Delta_{R}^{\beta \sqrt{\lambda}})>\lambda$ and hence $\Tr(-\Delta_{R}^{\beta \sqrt{\lambda}}-\lambda)_\limminus^\gamma=0$. This completes the proof.
\end{proof}

\begin{lemma}\label{lem: r bounds}
    For any $d\geq 2$, $\gamma \geq 0$, $\beta >0$ and $m \in \{1, \ldots, d-1\}$, 
    \begin{equation*}
        r_{\gamma+(d-m)/2, m}(\beta)\leq r_{\gamma, d}(\beta) \leq r_{\gamma, d-m}(\beta)r_{\gamma+(d-m)/2, m}(\beta)\, .
    \end{equation*}
    Moreover, the mapping $\gamma \mapsto r_{\gamma, d}(\beta)$ is non-increasing.
\end{lemma}

Lemma~\ref{lem: r bounds} has the following consequence for the numbers $\beta(\gamma, d)$ from Theorem~\ref{thm: Robin BLY}.
\begin{corollary}\label{cor: beta ineqs}
        The numbers $\beta(\gamma, d)$ defined in Theorem~\ref{thm: Robin BLY} satisfy
        \begin{equation*}
            \beta(\gamma+ (d-m)/2, m) \leq \beta(\gamma, d) \leq \beta(\gamma, m)
        \end{equation*}
        for any $d\geq 2$, $\gamma>0$ and $m \in \{1, \ldots, d-1\}$. Moreover, the mapping $\gamma \mapsto \beta(\gamma, d)$ is non-increasing.
\end{corollary}

\begin{proof}[Proof of Lemma~\ref{lem: r bounds}]
    Let $\{\lambda_j\}_{j\geq 1}$ be a non-negative sequence with $\lim_{j\to \infty}\lambda_j=\infty$ and define $R_j = (\lambda_j^{-1/2}R)\times (0, 1)^{d-m}$ for a fixed cuboid $R\in \R^{m}$. By Lemma~\ref{lem: collapsing Weyl law}, 
    \begin{equation*}
    \begin{aligned}
        r_{\gamma, d}(\beta) &\geq \limsup_{j \to \infty}\frac{\Tr(-\Delta_{R_j}^{\beta\sqrt{\lambda_j}}-\lambda_j)_\limminus^\gamma}{L_{\gamma, d}^{\rm sc}|R_j|\lambda_j^{\gamma+d/2}}\\
        &= \frac{\Tr(-\Delta_{R}^{\beta}-1)_\limminus^{\gamma+(d-m)/2}}{L_{\gamma+(d-m)/2, m}^{\rm sc}|R|} \\
        &= \frac{\Tr(-\Delta_{\lambda^{-1/2}R}^{\beta\sqrt{\lambda}}-\lambda)_\limminus^{\gamma+(d-m)/2}}{L_{\gamma+(d-m)/2, m}^{\rm sc}|\lambda^{-1/2}R|\lambda^{\gamma + (d-m)/2 + m/2}}\,  .
    \end{aligned}
    \end{equation*}
    where the last equality holds for any arbitrarily chosen $\lambda>0$ by scaling of the Robin eigenvalues. Since $R \subset \R^{m}$ and $\lambda>0$ are arbitrary, we conclude
    \begin{equation*}
        r_{\gamma, d}(\beta)\geq r_{\gamma+(d-m)/2, m}(\beta)\, , 
    \end{equation*}
    which completes the proof of the first inequality.

    To prove the second inequality claimed in the lemma we use the product structure of $R$. For any $m \in \{1, \ldots, d-1\}$ we can write $R= R'\times R''$ with $R'\subset \R^{d-m}$ and $R''\subset \R^{m}$ both being cuboids. Let $\tilde\lambda_k = (\lambda-\lambda_k(-\Delta_{R''}^{\beta\sqrt{\lambda}}))_\limplus$ and $\tilde\beta_k = \beta \sqrt{\lambda/\tilde\lambda_k}>\beta$ for any $k $ such that $\tilde \lambda_k>0$. Then
    \begin{align*}
        \Tr(-\Delta_{R}^{\beta\sqrt{\lambda}}-\lambda)_\limminus^\gamma
        &= \sum_{k\geq 1}\Tr(-\Delta_{R'}^{\beta \sqrt{\lambda}}-(\lambda-\lambda_k(-\Delta_{R''}^{\beta \sqrt{\lambda}}))_\limplus)_\limminus^\gamma\\
        &=
        \sum_{k: \tilde \lambda_k>0}\Tr(-\Delta_{R'}^{\tilde\beta_k \sqrt{\tilde\lambda_k}}-\tilde \lambda_k)_\limminus^\gamma\\
        &\leq
        \sum_{k: \tilde \lambda_k>0}r_{\gamma, d-m}(\tilde\beta_k)L_{\gamma, d-m}^{\rm sc}|R'|\tilde\lambda_k^{\gamma+(d-m)/2}\\
        &\leq
        r_{\gamma, d-m}(\beta)L_{\gamma, d-m}^{\rm sc}|R'|\Tr(-\Delta_{R''}^{\beta\sqrt{\lambda}}-\lambda)_\limminus^{\gamma+(d-m)/2}\\
        &\leq
        r_{\gamma, d-m}(\beta)r_{\gamma+(d-m)/2, m}(\beta)L_{\gamma, d}^{\rm sc}|R|\lambda^{\gamma+d/2}\, , 
    \end{align*}
    where we used that $\beta \mapsto r_{\gamma, d}(\beta)$ is non-increasing and that $L_{d-m, \gamma}^{\rm sc}L_{m, \gamma+(d-m)/2}^{\rm sc}=L_{\gamma, d}^{\rm sc}$. Since $R \subset \mathbb{R}^d$ and $\lambda>0$ are arbitrary this proves the claimed upper bound for $r_{\gamma, d}(\beta)$.

    That $\gamma \mapsto r_{\gamma, d}(\beta)$ is non-increasing follows from the Aizenman--Lieb identity. Indeed, for any $\gamma' >\gamma \geq 0$ we have
    {\allowdisplaybreaks
    \begin{align*}
        \Tr(-\Delta_{R}^{\beta \sqrt{\lambda}}-\lambda)_\limminus^{\gamma'}
        &= \frac{1}{B(1+\gamma, \gamma'-\gamma)}\int_0^\lambda(\lambda-\tau)^{\gamma'-\gamma-1}\Tr(-\Delta_{R}^{\beta\sqrt{\lambda}}-\tau)_\limminus^\gamma\, d\tau\\
        &=\frac{1}{B(1+\gamma, \gamma'-\gamma)}\int_0^\lambda(\lambda-\tau)^{\gamma'-\gamma-1}\Tr(-\Delta_{R}^{\beta\sqrt{\lambda/\tau}\sqrt{\tau}}-\tau)_\limminus^\gamma\, d\tau\\
        &\leq
        \frac{L_{\gamma, d}^{\rm sc}|R|}{B(1+\gamma, \gamma'-\gamma)}\int_0^\lambda r_{\gamma, d}(\beta\sqrt{\lambda/\tau})(\lambda-\tau)^{\gamma'-\gamma-1}\tau^{\gamma+d/2}\, d\tau\\
        &\leq
        \frac{r_{\gamma, d}(\beta)L_{\gamma, d}^{\rm sc}|R|}{B(1+\gamma, \gamma'-\gamma)}\int_0^\lambda (\lambda-\tau)^{\gamma'-\gamma-1}\tau^{\gamma+d/2}\, d\tau\\
        &=r_{\gamma, d}(\beta)L_{\gamma', d}^{\rm sc}|R|\lambda^{\gamma'+d/2}\, . 
    \end{align*}
    }
    Here, $B(x, y) = \frac{\Gamma(x)\Gamma(y)}{\Gamma(x+y)}$ denotes the Euler Beta function and we used again that $\beta \mapsto r_{\gamma, d}(\beta)$ is non-increasing. This completes the proof of Lemma~\ref{lem: r bounds}.
    \end{proof}

\begin{proof}[Proof of Theorem~\ref{thm: Robin BLY}]
    We split the proof into three steps. First, we show that for any $\gamma>0$ there exists a $\beta(\gamma, 1)>0$ at which $r_{\gamma, 1}(\beta)$ becomes $1$. The second step uses the established existence of $\beta(\gamma, 1)$ to prove that $\beta(\gamma, d)$ is well defined for any $d$. In the third and final step we argue that the claimed strict inequality holds when $\beta >\beta(\gamma, d)$.

    \medskip
    \noindent{\it Step 1: Existence of $\beta(\gamma, 1).$} We begin by proving that for $\gamma>0$ there exists $\beta(\gamma, 1)>0$ so that $r_{\gamma, 1}(\beta)=1$ if and only if $\beta \geq \beta(\gamma, 1)$. Since $\beta \mapsto r_{\gamma, 1}(\beta)$ is non-increasing and lower semicontinuous (i.e.\ right continuous) it suffices to show that there exist $\beta_0$, $\beta_1>0$ with $0<\beta_0<\beta_1$ so that $r_{\gamma, 1}(\beta_0)>1$ and $r_{\gamma, 1}(\beta_1)=1$.

    Let $\gamma>0$ be fixed in the following. The existence of $\beta_0>0$ such that $r_{\gamma, 1}(\beta_0)>1$ follows from the two-term Weyl law. Indeed, for any $0<\beta_0<\beta_W(\gamma, 0)$ and $\varepsilon\in (0, 1)$, Corollary~\ref{cor: sc two-term Weyl} implies that for $\lambda$ sufficiently large, 
    \begin{equation*}
        \frac{\Tr(-\Delta_{(0, 1)}^{\beta_0 \sqrt{\lambda}}-\lambda)_\limminus^\gamma}{L_{\gamma, 1}^{\rm sc}\lambda^{\gamma+1/2}}\geq 1+(1-\varepsilon)\frac{L_{\gamma, 0}(\beta_0)}{2 L_{\gamma, 1}^{\rm sc}}\lambda^{-1/2}>1\, .
    \end{equation*}
    In particular, we conclude that $r_{\gamma, 1}(\beta_0)>1$.

    We now turn to proving that there exists $\beta_1>0$ such that $r_{\gamma, 1}(\beta_1)=1$. By scaling of the Robin Laplacian, 
    \begin{equation*}
        \frac{\Tr(-\Delta^{\beta \sqrt{\lambda}}_{(0, l)}-\lambda)^\gamma_\limminus}{L_{\gamma, 1}^{\rm sc}l\lambda^{\gamma+1/2}} = \frac{\Tr(-\Delta_{(0, 1)}^{\beta \sqrt{l^2\lambda}}-l^2\lambda)_\limminus^\gamma}{L_{\gamma, 1}^{\rm sc}(l^2\lambda)^{\gamma+1/2}}\, .
    \end{equation*}
    Therefore, it is sufficient to consider $R=(0, 1)$.

    For small $\lambda$, we can show the desired inequality as follows. Since by Lemma~\ref{lem: lambda1 asymptotics}, $\lambda_1(-\Delta_{(0, 1)}^{\beta \sqrt{\lambda}}) = 2\beta \sqrt{\lambda} + O(\lambda)$ as $\lambda \to 0$, it follows that for any $\beta >0$ there exists a $\Lambda(\beta)\in (0, \pi^2)$ so that $\lambda_1(-\Delta_{(0, 1)}^{\beta\sqrt{\lambda}})>\lambda$ for $\lambda \in (0, \Lambda(\beta))$. Consequently, for any $\gamma \geq 0$, 
    \begin{equation*}
        \Tr(-\Delta_{(0, 1)}^{\beta \sqrt{\lambda}}-\lambda)_\limminus^\gamma =0 \quad \mbox{for all }\lambda \leq \Lambda(\beta)\, .
    \end{equation*}
    By monotonicity of the Riesz mean with respect to $\beta$ for fixed $\lambda$ it follows that $\beta \mapsto \Lambda(\beta)$ is non-decreasing. Thus for any $\beta \geq \beta'>0$ we have that
    \begin{equation*}
        \Tr(-\Delta_{(0, 1)}^{\beta \sqrt{\lambda}}-\lambda)^\gamma_\limminus = 0 < L_{\gamma, 1}^{\rm sc}\lambda^{\gamma+1/2} \quad \mbox{for all }\lambda \leq \Lambda(\beta')\, .
    \end{equation*}

    For large $\lambda$ we can argue again with the two-term Weyl law. By Corollary~\ref{cor: sc two-term Weyl}, 
    \begin{equation*}
        \lim_{\lambda \to \infty} \frac{\Tr(-\Delta_{(0, 1)}^{\beta \sqrt{\lambda}}-\lambda)^\gamma_\limminus- L_{\gamma, 1}^{\rm sc}\lambda^{\gamma+1/2}}{\lambda^{\gamma}} = \frac{L_{\gamma, 0}(\beta)}{2}\, 
    \end{equation*} 
    where $L_{\gamma, 0}(\beta) < 0$ for $\beta > \beta_W(\gamma, 0)$. Consequently, for any $\beta>\beta_W(\gamma, 0)$ there exists a $\Lambda'(\beta, \gamma)>0$ such that
    \begin{equation*}
        \Tr(-\Delta_{(0, 1)}^{\beta \sqrt{\lambda}}-\lambda)^\gamma_\limminus \leq L_{\gamma, 1}^{\rm sc}\lambda^{\gamma+1/2}+ \frac{L_{\gamma, 0}(\beta)}{4}\lambda^{\gamma} < L_{\gamma, 1}^{\rm sc}\lambda^{\gamma+1/2}\, , 
    \end{equation*}
    for all $\lambda \geq \Lambda'(\beta, \gamma)$.

    The monotonicity of the Riesz mean in $\beta$ for fixed $\lambda$, implies that that for any $\varepsilon>0$ and all $\beta \geq \beta_W(\gamma, 0)+\varepsilon$ it holds that
    \begin{equation*}
        \Tr(-\Delta_{(0, 1)}^{\beta \sqrt{\lambda}}-\lambda)^\gamma_\limminus \leq L_{\gamma, 1}^{\rm sc}\lambda^{\gamma+1/2}
    \end{equation*}
    for all $\lambda \in (0, \Lambda(\beta_W(\gamma, 0)+\varepsilon)] \cup [\Lambda'(\beta_W(\gamma, 0)+\varepsilon, \gamma), \infty)$. It remains to show that the inequality also holds in the intermediate regime $\lambda \in [\Lambda(\beta_W(\gamma, 0)+\varepsilon), \Lambda'(\beta_W(\gamma, 0)+\varepsilon, \gamma)]$ for $\beta$ large enough. To that end, we note that for fixed $\lambda>0$ function $\beta \mapsto \Tr(-\Delta_{(0, 1)}^{\beta \sqrt{\lambda}}-\lambda)_\limminus^\gamma$ is non-increasing and satisfies
    \begin{equation}\label{eq: pointwise conv Riesz mean interval}
        \lim_{\beta \to \infty}\Tr(-\Delta_{(0, 1)}^{\beta\sqrt{\lambda}}-\lambda)^\gamma_\limminus = \Tr(-\Delta_{(0, 1)}^{\rm D}-\lambda)^\gamma_\limminus \, .
    \end{equation}
    Furthermore, it is a consequence of the validity of P\'olya's conjecture for $-\Delta_{(0, 1)}^{\rm D}$ and the Aizenman--Lieb identity (see, e.g., \cite[Theorem 4.5]{FrankLarson_CPAM26}) that for any fixed $\gamma>0$, 
    \begin{equation}\label{eq: strict Berezin Dirichlet}
        \Tr(-\Delta_{(0, 1)}^{\rm D}-\lambda)_\limminus^\gamma< L_{\gamma, 1}^{\rm sc}\lambda^{\gamma+1/2}\, \qquad \text{for all } \lambda >0 .
    \end{equation}
    Since $\Tr(-\Delta^{\beta\sqrt{\lambda}}_{(0, 1)}-\lambda)_\limminus^0 \leq \sqrt{\lambda}/\pi +1$, Lemma~\ref{lem: lambdak derivative bounds} implies that for any fixed $\gamma>0$ the functions $f_{\gamma, \beta}(\lambda)=\Tr(-\Delta^{\beta\sqrt{\lambda}}_{(0, 1)}-\lambda)_\limminus^\gamma$, $\beta \in [0, \infty)$, form a family of uniformly H\"older continuous functions on any compact subset of $(0, \infty)$. Therefore, the Arzel\`a--Ascoli theorem implies that the pointwise convergence in \eqref{eq: pointwise conv Riesz mean interval} is uniform on $[\Lambda(\beta_W(\gamma, 0)+\varepsilon), \Lambda'(\beta_W(\gamma, 0)+\varepsilon, \gamma)]$. Consequently, \eqref{eq: strict Berezin Dirichlet} implies that there exists a $\beta_1>0$ large enough so that for all $\beta \geq \beta_1$ 
    \begin{equation*}
        \Tr(-\Delta_{(0, 1)}^{\beta \sqrt{\lambda}}-\lambda)^\gamma_\limminus \leq L_{\gamma, 1}^{\rm sc}\lambda^{\gamma+1/2} \quad \mbox{for all } \lambda \in [\Lambda(\beta_W(\gamma, 0)+\varepsilon), \Lambda'(\beta_W(\gamma, 0)+\varepsilon, \gamma)]\, .
    \end{equation*}
    
    Therefore, we have established the desired existence of $0<\beta_0<\beta_1$ for which $r_{\gamma, 1}(\beta_0)>1$ and $r_{\gamma, 1}(\beta_1)=1$. As argued at the beginning of the proof, this shows that $\beta(\gamma, 1)$ is well-defined.

    \medskip

    \noindent{\it Step 2: Existence of $\beta(\gamma, d).$} Fix $d\geq 2$ and $\gamma>0$. As in Step 1, the monotonicity and lower semicontinuity of $\beta \mapsto r_{\gamma, d}(\beta)$ implies that it is sufficient to show that there exists $0<\beta_0<\beta_1$ for which $r_{\gamma, d}(\beta_0)>1$ and $r_{\gamma, d}(\beta_1)=1$.

    As in the one-dimensional case we can deduce the existence of the desired $\beta_0$ from the two-term Weyl asymptotics of $\Tr(-\Delta_R^{\beta\sqrt{\lambda}}-\lambda)_\limminus^\gamma$ in Corollary~\ref{cor: sc two-term Weyl} by considering an arbitrary cuboid $R$ and and $\beta_0 \in(0, \beta_W(\gamma, d-1))$ with $\lambda$ sufficiently large.

    To deduce the existence of $\beta_1$ we rely on the existence of $\beta(\gamma, 1)$, for any $\gamma>0$, and Lemma~\ref{lem: r bounds}. By repeated use of the upper bound in Lemma~\ref{lem: r bounds} it follows that for any $d\geq 2, \gamma > 0, \beta>0$
    \begin{equation*}
        r_{\gamma, d}(\beta) 
        \leq \prod_{m=0}^{d-1}r_{\gamma+m/2, 1}(\beta)\, .
    \end{equation*}
    We conclude that $r_{\gamma, d}(\beta)=1$ if $\beta \geq \max_{m=0, \ldots, d-1}\{\beta(\gamma+m/2, 1)\} = \beta(\gamma, 1)$. Thus the existence of $\beta(\gamma, d)$ follows since we may take $\beta_1 = \beta(\gamma, 1)$.

    \medskip

    \noindent{\it Step 3: Strict inequality for $\beta >\beta(\gamma, d).$} If $\lambda_1(-\Delta_{R}^{\beta \sqrt{\lambda}})-\lambda \geq 0$ then $\Tr(-\Delta_{R}^{\beta\sqrt{\lambda}}-\lambda)_\limminus^\gamma=0< L_{\gamma, d}^{\rm sc}|R|\lambda^{\gamma+d/2}$. If instead $\lambda_1(-\Delta_{R}^{\beta \sqrt{\lambda}})-\lambda <0$, then the sum defining the Riesz mean contains at least one non-zero term. In this case, the claimed strict inequality thus follows from the validity of the inequality at $\beta = \beta(\gamma, d)$ and the fact that the functions $\beta \mapsto \lambda_k(-\Delta_{R}^{\beta})$ are strictly increasing with respect to $\beta$. This completes the proof of Theorem~\ref{thm: Robin BLY}.
\end{proof}

\section{The shape optimization problem}
\label{sec: shape optimization}

We have now developed all the tools needed for our analysis of the shape optimization problem. This is the content of the current section and our main goal is to prove Theorem~\ref{thm: main theorem intro}. However, before we turn to studying the asymptotic geometric behavior of sequences of maximizing cuboids we establish a couple of preliminary results concerning the underlying shape optimization problem.

The first result establishes the existence of optimizers for any fixed values of $\beta$ and $\lambda$.
\begin{lemma}\label{lem: existence of optimizer}
    Let $d\geq 2, \gamma > 0, \lambda >0, $ and $\beta>0$. There exists a cuboid $R^* \subset \R^d$ with $|R^*|=1$ such that
    \begin{equation*}
        \Tr(-\Delta_{R^*}^\beta-\lambda)_\limminus^\gamma = M_{\gamma, d}(\lambda, \beta)\, .
    \end{equation*}
\end{lemma}

\begin{proof}
    Without loss of generality we may assume that $\lambda$ is so large that $M_{\gamma, d}(\lambda, \beta)>0$. Indeed, if $M_{\gamma, d}(\lambda, \beta)=0$, any cuboid is a maximizer and the statement of the lemma is trivially true.

    Let $\{R_j\}_{j\geq 1}$ be a maximizing sequence of unit measure cuboids for $M_{\gamma, d}(\lambda, \beta)$. By the invariance under orthogonal transformations we may without loss of generality assume that $R_j = \prod_{i=1}^d (0, l_i^{(j)})$ with $0<l_{1}^{(j)} \leq \ldots \leq l_d^{(j)}$ for each $j\geq 1$. We aim to show that the sequence of vectors $\{(l_1^{(j)}, \ldots, l_d^{(j)})\}_{j\geq 1}$ is uniformly bounded. By the measure constraint, it suffices to prove that $\liminf_{j\to \infty}l_1^{(j)}>0$. To prove this, we argue by contradiction. 
    
    Assume that $\liminf_{j\to \infty}l_1^{(j)}=0$. Along this subsequence it holds that $\lambda_1(-\Delta_{R_j}^\beta) \to \infty$. Indeed, from scaling of the Robin eigenvalues and Lemma~\ref{lem: lambda1 asymptotics}
    \begin{align*}
        \lambda_1(-\Delta_{R_j}^\beta) &= \sum_{i=1}^d \lambda_1(-\Delta_{(0, l_i^{(j)})}^\beta) \\
        &\geq \lambda_1(-\Delta_{(0, l_1^{(j)})}^\beta) \\
        &= (l_1^{(j)})^{-2}\lambda_1(-\Delta_{(0, 1)}^{l_1^{(j)}\beta}) = 2\beta(l_1^{(j)})^{-1} + O(\beta^2)\, , 
    \end{align*}
    this proves the claim. Consequently, there are arbitrarily large $j$ for which $-\Delta_{R_j}^\beta$ has no eigenvalues less than $\lambda$, and thus $\Tr(-\Delta_{R_j}^\beta-\lambda)_\limminus^\gamma =0$. Since we assumed that $M_{\gamma, d}(\lambda, \beta)>0$ this contradicts the maximizing property of $\{R_j\}_{j\geq 1}$. We conclude that for maximizing sequences $\{R_j\}_{j\geq 1}$ we must have $\liminf_{j \to \infty}l_1^{(j)}>0$. 

    By compactness, we can now extract a converging sequence of cuboids from our maximizing sequence, slightly abusing notation we keep denoting this subsequence by $\{R_j\}_{j\geq 1}$. We denote its limit by $R^*$ and aim to show that $R^*$ realizes the supremum defining $M_{\gamma, d}(\lambda, \beta)$. Let
    \begin{equation*}
        l_i^* = \lim_{j\to \infty} l_i^{(j)}
    \end{equation*}
    be the side lengths of $R^*$. Note that $|R_j|=1$ for each $j\geq 1$ implies that $|R^*| =1$. Each eigenvalue of $-\Delta_{R^*}^\beta$ is naturally identified as a limit of eigenvalues of $-\Delta_{R_j}^\beta$. Indeed, each eigenvalue on the cuboid $R=\prod_{i=1}^d (0, l_i)$ is naturally identified with a $d$-tuple of natural numbers $\bar k=(k_1, \ldots, k_d)$ by
    \begin{equation*}
        \lambda_{\bar k}(-\Delta_R^\beta) = \sum_{i=1}^d \lambda_{k_i}(-\Delta_{(0, l_i)}^\beta)= \sum_{i=1}^d l_i^{-2}\lambda_{k_i}(-\Delta_{(0, 1)}^{l_i\beta})\, .
    \end{equation*}
    Thus, in particular, the continuity of $\beta \mapsto \lambda_{k}(-\Delta_{(0, 1)}^\beta)$ for each $k\geq 1$ implies that for any fixed $\bar k =(k_1, \ldots, k_d)\in \N^d$ we have
    \begin{equation}\label{eq: eigenvalue convergence}
        \lim_{j\to \infty} \lambda_{\bar k}(-\Delta_{R_j}^\beta) = \lambda_{\bar k}(-\Delta_{R^*}^\beta)\, .
    \end{equation}
    
    Note that for any fixed $\lambda$ the bound in Lemma~\ref{lem: Neumann Riesz mean bound} implies that the sum defining the Riesz means $\Tr(-\Delta_{R_j}^\beta-\lambda)_\limminus^\gamma$ and $\Tr(-\Delta_{R^*}^\beta-\lambda)_\limminus^\gamma$ contain a uniformly bounded number of non-zero terms. Thus, combining \eqref{eq: eigenvalue convergence} with the fact that  $\mu \mapsto (\lambda-\mu)_\limplus^\gamma$ is upper semicontinuous (even continuous if $\gamma>0$) we find that
    \begin{equation*}
        M_{\gamma, d}(\lambda, \beta)=\lim_{j\to \infty}\Tr(-\Delta_{R_j}^\beta-\lambda)_\limminus^\gamma \leq \Tr(-\Delta_{R^*}^\beta-\lambda)_\limminus^\gamma\, .
    \end{equation*}
    Since $R^*$ is admissible in the supremum defining $M_{\gamma, d}(\lambda, \beta)$ equality must hold in this inequality and we conclude that $R^*$ is a maximizer.
\end{proof}

\subsection{The asymptotic behavior of \texorpdfstring{$M_{\gamma, d}$}{Mdgamma}}

In this subsection we turn our attention to the asymptotic behavior of $M_{\gamma, d}(\beta, \lambda)$. 
The main aim of this subsection is to prove the following theorem.
\begin{theorem}\label{thm: asymptotics of M}
    Let $d\geq 2, \gamma \geq 0$. Let $\{\lambda_j\}_{j\geq 1}$ and $\{\beta_j\}_{j\geq 1}$ be sequences of positive numbers with $\lim_{j\to \infty}\lambda_j=\infty$.
    \begin{enumerate}[label=(\roman*)]
        \item\label{itm: almost Neumann} If $\lim_{j\to \infty}\beta_j/\sqrt{\lambda_j} =0$, then
        \begin{equation*}
            \lim_{j\to \infty}\frac{M_{\gamma, d}(\lambda_j, \beta_j)}{L_{\gamma, d}^{\rm sc}\lambda_j^{\gamma+d/2}} = \infty\, .
        \end{equation*}
        \item\label{itm: positive limit beta} If $\lim_{j\to \infty}\beta_j/\sqrt{\lambda_j} =\beta'>0$, then
        \begin{equation*}
            \lim_{j\to \infty}\frac{M_{\gamma, d}(\lambda_j, \beta_j)}{L_{\gamma, d}^{\rm sc}\lambda_j^{\gamma+d/2}} = r_{\gamma+1/2, d-1}(\beta')\, .
        \end{equation*}
        \item\label{itm: infinite limit beta} If $\lim_{j\to \infty}\beta_j/\sqrt{\lambda_j} =\infty$, then
        \begin{equation*}
            \lim_{j\to \infty}\frac{M_{\gamma, d}(\lambda_j, \beta_j)}{L_{\gamma, d}^{\rm sc}\lambda_j^{\gamma+d/2}} = 1\, .
        \end{equation*}
    \end{enumerate}
\end{theorem}
\begin{remark}
    Note that since $r_{\gamma+1/2, d-1}(\beta')=1$ if and only if $\beta' \geq \beta(\gamma+1/2, d-1)$ the asymptotics of $M_{\gamma, d}(\lambda_j, \beta_j)$ match the Weyl asymptotics for any fixed cuboid if and only if $\beta' \geq \beta(\gamma+1/2, d-1)$. 
\end{remark}

\begin{proof}[Proof of Theorem~\ref{thm: asymptotics of M}]
We split the proof according to the asymptotic behavior of $\beta_j/\sqrt{\lambda_j}$. We begin with proving~\ref{itm: almost Neumann} and then proceed to prove~\ref{itm: positive limit beta} and~\ref{itm: infinite limit beta}.

\medskip
\noindent{\it Part 1: Proof of~\ref{itm: almost Neumann}.} In this regime, we prove the claim by constructing a suitable sequence of unit measure cuboids $R_j$ for which
\begin{equation*}
    \lim_{j\to \infty}\frac{\Tr(-\Delta_{R_j}^{\beta_j}-\lambda_j)_\limminus^\gamma}{\lambda_j^{\gamma+d/2}}= \infty\, .
\end{equation*}
Since $M_{\gamma, d}(\lambda_j, \beta_j)\geq \Tr(-\Delta_{R_j}^{\beta_j}-\lambda_j)_\limminus^\gamma$ for each $j$ this implies the claim in~\ref{itm: almost Neumann}.

Consider the sequence of cuboids defined by $R_j = (0, l_j) \times (0, l_j^{-1/(d-1)})^{d-1}$ with $l_j =(2+\varepsilon)\beta_j\lambda_j^{-1}$ for any fixed $\varepsilon>0$. Note that with this choice of $R_j$, we have $|R_j|=1$ and since $\lim_{j\to \infty}\beta_j/\sqrt{\lambda_j}=0$ it holds that
\begin{equation*}
    \lim_{j\to \infty} l_j \sqrt{\lambda_j} =0\, , \quad \lim_{j\to \infty}\beta_j l_j = 0\, , \quad\lim_{j\to \infty} \lambda_j l_j^{-2/(d-1)}=\infty\, , \quad \mbox{and }\lim_{j\to \infty} \frac{\beta_j}{l_j\lambda_j} = \frac{1}{2+\varepsilon}< \frac{1}{2}\, .
\end{equation*}
In particular, as $\varepsilon>0$ is arbitrary, this sequence falls just outside of when Lemma~\ref{lem: super-critical collapse} applies.

By Lemma~\ref{lem: lambda1 asymptotics} and the choice of $l_j$ we have that 
\begin{equation*}
    \lambda_1(-\Delta_{(0, l_j)}^{\beta_j})= \frac{2\beta_j}{l_j}(1+ o(1)) = \frac{2\lambda_j}{2+\varepsilon} (1+o(1))
\end{equation*}
as $j \to \infty$. Thus 
\begin{equation*}
    \lambda_j-\lambda_1(-\Delta_{(0, l_j)}^{\beta_j}) \geq \frac{\varepsilon \lambda_j}{2+2\varepsilon}
\end{equation*}
for any $j$ sufficiently large. Therefore, setting $R_j'= (0, l_j^{-1/(d-1)})^{d-1} = l_j^{-1/(d-1)}Q$ where $Q=(0, 1)^{d-1}\subset \R^{d-1}$ we have by Weyl's law for $-\Delta_Q^{\rm D}$
{\allowdisplaybreaks
\begin{align*}
    \lim_{j\to \infty}\frac{\Tr(-\Delta_{R_j}^{\beta_j}-\lambda_j)_\limminus^\gamma}{\lambda_j^{\gamma+d/2}}
    &=
    \lim_{j\to \infty}\frac{1}{\lambda_j^{\gamma+d/2}}\sum_{k\geq 1}\Tr(-\Delta_{R_j'}^{\beta_j}-(\lambda_j-\lambda_k(-\Delta_{(0, l_j)}^{\beta_j})))_\limminus^\gamma\\
    &\geq 
    \lim_{j\to \infty}\frac{1}{\lambda_j^{\gamma+d/2}}\Tr\Bigl(-\Delta_{R_j'}^{\rm D}-\frac{\varepsilon \lambda_j}{2+2\varepsilon}\Bigr)_\limminus^\gamma\\
    &=
    \lim_{j\to \infty}\frac{l_j^{2\gamma/(d-1)}}{\lambda_j^{\gamma+d/2}}\Tr\Bigl(-\Delta_{Q}^{\rm D}-\frac{\varepsilon }{2+2\varepsilon}\lambda_jl_j^{-2/(d-1)}\Bigr)_\limminus^\gamma\\
    &=
    \lim_{j\to \infty}L_{\gamma, d-1}^{\mathrm sc}\Bigl(\frac{\varepsilon }{2+2\varepsilon}\Bigr)^{\gamma+(d-1)/2}(l_j\sqrt{\lambda_j})^{-1}\\
    &=\infty\, .
\end{align*}
}
This completes the proof of Theorem~\ref{thm: asymptotics of M} \ref{itm: almost Neumann}.

\medskip
\noindent{\it Part 2: Proof of~\ref{itm: positive limit beta} and~\ref{itm: infinite limit beta}.} By Lemma~\ref{lem: existence of optimizer} there exists a sequence of unit measure $d$-dimensional cuboids $\{R_j\}_{j\geq 1}$ such that $M_{\gamma, d}(\lambda_j, \beta_j) = \Tr(-\Delta_{R_j}^{\beta_j}-\lambda_j)_\limminus^\gamma$ for each $j$. Without loss of generality we may assume that the side lengths $l_1^{(j)}, \ldots, l_d^{(j)}$ of $R_j$ satisfy that $l_1^{(j)}\leq l_2^{(j)}\leq \ldots \leq l_d^{(j)}$.

If $\lim_{j \to \infty} l_1^{(j)}\sqrt{\lambda_j} = \infty$, then by Lemma~\ref{lem: quantitative Weyl law} it holds that
\begin{equation*}
    \limsup_{j\to \infty} \frac{\Tr(-\Delta_{R_j}^{\beta_j}-\lambda_j)_\limminus^\gamma}{L_{\gamma, d}^{\rm sc}\lambda_j^{\gamma+d/2}} = 1\, .
\end{equation*}
If $\lim_{j\to \infty} l_1^{(j)}\sqrt{\lambda_j}=0$, then since $\liminf_{j\geq 1} \frac{\beta_j}{l_1^{(j)}\lambda_j} = \infty$ as $\lim_{j\to \infty}\beta_j/\sqrt{\lambda_j}>0$ it follows from Lemma~\ref{lem: super-critical collapse} that
\begin{equation*}
    \limsup_{j\to \infty} \frac{\Tr(-\Delta_{R_j}^{\beta_j}-\lambda_j)_\limminus^\gamma}{L_{\gamma, d}^{\rm sc}\lambda_j^{\gamma+d/2}} = 0\, .
\end{equation*}
If instead $0<\lim_{j\to \infty}l_1^{(j)}\sqrt{\lambda_j}<\infty$, then Lemma~\ref{lem: collapsing Weyl law} implies that there exist $m \in \{1, \ldots, d-1\}$ and an $m$-dimensional cuboid $R'$ such that
\begin{equation*}
    \limsup_{j\to \infty} \frac{\Tr(-\Delta_{R_j}^{\beta_j}-\lambda_j)_\limminus^\gamma}{L_{\gamma, d}^{\rm sc}\lambda_j^{\gamma+d/2}} = 
    \begin{cases}
    \displaystyle \frac{\Tr(-\Delta_{R'}^{\beta'}-1)_\limminus^{\gamma+ (d-m)/2}}{L_{\gamma+(d-m)/2, m}^{\rm sc}|R'|} & \mbox{in case \ref{itm: positive limit beta}}\, , \\[12pt]
    \displaystyle \frac{\Tr(-\Delta_{R'}^{\rm D}-1)_\limminus^{\gamma+ (d-m)/2}}{L_{\gamma+(d-m)/2, m}^{\rm sc}|R'|} & \mbox{in case \ref{itm: infinite limit beta}}\, .
    \end{cases}
\end{equation*}

By the validity of P\'olya's conjecture for $-\Delta_{R'}^{\rm D}$ we conclude that under the assumptions of~\ref{itm: infinite limit beta} it holds that
\begin{equation*}
    \frac{\Tr(-\Delta_{R'}^{\rm D}-1)_\limminus^{\gamma+ (d-m)/2}}{L_{\gamma+(d-m)/2, m}^{\rm sc}|R'|}\leq 1\, .
\end{equation*}
Therefore, in the setting of~\ref{itm: infinite limit beta} we conclude that
\begin{equation*}
     \limsup_{j\to \infty} \frac{\Tr(-\Delta_{R_j}^{\beta_j}-\lambda_j)_\limminus^\gamma}{L_{\gamma, d}^{\rm sc}\lambda_j^{\gamma+d/2}} \leq 1\, .
\end{equation*}

In the setting of~\ref{itm: positive limit beta}, using the definition of $r_{\gamma, d}(\beta')$ and the fact that each of these numbers are not smaller than $1$, we conclude that
\begin{equation*}
    \limsup_{j\to \infty}\frac{\Tr(-\Delta_{R_j}^{\beta_j}-\lambda_j)_\limminus^\gamma}{L_{\gamma, d}^{\rm sc}\lambda_j^{\gamma+d/2}} \leq \max_{m=1, \ldots, d-1} \;r_{\gamma+(d-m)/2, m}(\beta') = r_{\gamma+1/2, d-1}(\beta')\, , 
\end{equation*}
where the final equality follows from Lemma~\ref{lem: r bounds}. 

To complete the proof of~\ref{itm: positive limit beta} and~\ref{itm: infinite limit beta} it remains to prove a matching lower bound. For the lower bound we construct suitable trial sequences of cuboids. 

In the setting of~\ref{itm: infinite limit beta} it suffices to take the constant sequence of $d$-dimensional unit measure cuboids and use Lemma~\ref{lem: quantitative Weyl law}. The same trial sequence works in the setting of~\ref{itm: positive limit beta} whenever $r_{\gamma, d}(\beta')=1$.

To construct a trial sequence in the setting of~\ref{itm: positive limit beta}, we can use the same construction as in the proof of Lemma~\ref{lem: r bounds}. For $\beta'>0$ we may choose any fixed cuboid $R \subset \R^{d-1}$ to define a trial sequence of $d$-dimensional cuboids by $R_j = (\lambda_j^{-1/2}R)\times (0, |R|^{-1}\lambda_j^{(d-1)/2})$. Note that by construction $|R_j|=1$ for any $j\geq 1$. By Lemma~\ref{lem: collapsing Weyl law} applied along a subsequence for which the $\liminf$ is attained, we conclude that  
\begin{equation*}
    \liminf_{j\to \infty} \frac{M_{\gamma, d}(\lambda_j, \beta_j)}{L_{\gamma, d}^{\rm sc}\lambda_j^{\gamma+d/2}}
    \geq
    \liminf_{j\to \infty} \frac{\Tr(-\Delta_{R_j}^{\beta_j}-\lambda_j)_\limminus^\gamma}{L_{\gamma, d}^{\rm sc}\lambda_j^{\gamma+d/2}}
    = 
    \frac{\Tr(-\Delta_{R}^{\beta'}-1)^{\gamma+1/2}}{L_{\gamma+1/2, d-1}^{\rm sc}|R|}\, .
\end{equation*}
As $R$ was arbitrary we can make the right-hand side as close to $r_{\gamma+1/2, d-1}(\beta')$ as we wish. Combining the obtained upper and lower bounds completes the proof of Theorem~\ref{thm: asymptotics of M} \ref{itm: positive limit beta} and \ref{itm: infinite limit beta}, and therefore the proof of Theorem~\ref{thm: asymptotics of M}.
\end{proof}

\subsection{The asymptotic behavior of almost maximizers for Riesz means}

In this subsection we finally get to the proof of our main result concerning the asymptotic behavior of cuboids which (almost) maximize Riesz means in the limit of the spectral parameter tending to infinity. Our goal is to prove the following theorem of which Theorem~\ref{thm: main theorem intro} is a special case.
\begin{theorem}\label{thm: geometric convergence general}
    Fix $d\geq 2$, $\gamma>0$. Let $\{\lambda_j\}_{j\geq 1}$, $\{\beta_j\}_{j\geq 1}$ be sequences of positive numbers, and $\{R_j\}_{j\geq 1}$ be a sequence of cuboids in $\R^d$. Assume that $|R_j|=1$ for each $j\geq 1$,  
    \begin{equation*}
        \lim_{j\to \infty}\lambda_j =\infty \, , \quad \mbox{and}\quad \lim_{j\to \infty} \frac{\Tr(-\Delta_{R_j}^{\beta_j}-\lambda_j)_\limminus^\gamma-M_{\gamma, d}(\lambda_j, \beta_j)}{\lambda_j^{\gamma+(d-1)/2}}=0\, .
    \end{equation*}
    \begin{enumerate}[label=(\roman*)]
        \item\label{itm: geometric convergence subcritical} If $\limsup_{j \to \infty} \beta_j/\sqrt{\lambda_j}<\beta(\gamma+1/2, d-1)$, then the sequence $\{R_j\}_{j\geq 1}$ has no converging subsequences.
        \item\label{itm: geometric convergence supercritical} If $\liminf_{j \to \infty} \beta_j/\sqrt{\lambda_j}>\beta(\gamma+1/2, d-1)$, then the sequence $\{R_j\}_{j\geq 1}$ converges to the unit cube as $j\to \infty$.
    \end{enumerate}
\end{theorem}

\begin{proof}
    We split the proof according to the two claims. Without loss of generality we shall always assume that $R_j$ has side lengths $l_1^{(j)}, \ldots, l_d^{(j)}$ which are ordered so that $l_1^{(j)}\leq l_2^{(j)}\leq \ldots \leq l_d^{(j)}$ for each $j\geq 1$.

    \medskip
    \noindent{\it Part 1: Proof of~\ref{itm: geometric convergence subcritical}.} The claim is equivalent to showing that
    \begin{equation*}
        \limsup_{j\to \infty} l_1^{(j)} = 0\, .
    \end{equation*}
    We argue by contradiction, showing that $\limsup_{j\to \infty}l_1^{(j)}= l>0$ contradicts the almost maximizing property of $\{R_j\}_{j\geq 1}$. By passing to a subsequence we may assume without loss of generality that $\lim_{j\to \infty}l_1^{(j)}=l$. 
    
    Let $\beta' = \limsup_{j \to \infty} \beta_j/\sqrt{\lambda_j}$. By Theorem~\ref{thm: asymptotics of M} \ref{itm: positive limit beta}, the assumptions on $\{R_j\}_{j\geq 1}$ yield
    \begin{equation}\label{eq: supercritical asymptotics}
        \liminf_{j\to \infty}\frac{\Tr(-\Delta_{R_j}^{\beta_j}-\lambda_j)_\limminus^{\gamma}}{L_{\gamma, d}^{\rm sc}\lambda_j^{\gamma+d/2}} \geq r_{\gamma+1/2, d-1}(\beta') >1\, , 
    \end{equation}
    where the last inequality is due to $\beta'< \beta(\gamma+1/2, d-1)$.
    On the other hand, because $\lim_{j\to \infty}l_1^{(j)}=l$, and $|R_j|=1$, we have that $\Haus^{d-1}(\partial R_j)$ remains uniformly bounded from above. Since $\lim_{j\to \infty}\lambda_j =\infty$, it thus follows from Lemma~\ref{lem: quantitative Weyl law} that
    \begin{equation*}
        \lim_{j\to \infty}\frac{\Tr(-\Delta_{R_j}^{\beta_j}-\lambda_j)_\limminus^{\gamma}}{L_{\gamma, d}^{\rm sc}\lambda_j^{\gamma+d/2}}=1\, , 
    \end{equation*}
    which contradicts \eqref{eq: supercritical asymptotics}. Therefore, we conclude that for any $\{R_j\}_{j\geq 1}$ as in the statement of the theorem, we must have $\limsup_{j\to \infty}l_1^{(j)}=0$ and therefore the sequence of cuboids cannot contain any convergent subsequences. This completes the proof of Theorem~\ref{thm: geometric convergence general} \ref{itm: geometric convergence subcritical}.

    \medskip
    \noindent{\it Part 2: Proof of \ref{itm: geometric convergence supercritical}.} By Theorem~\ref{thm: asymptotics of M}, the assumptions on $\{R_j\}_{j\geq 1}$ imply that
    \begin{equation}\label{eq: subcritical asymptotics}
        \lim_{j\to \infty}\frac{\Tr(-\Delta_{R_j}^{\beta_j}-\lambda_j)_\limminus^{\gamma}}{L_{\gamma, d}^{\rm sc}\lambda_j^{\gamma+d/2}}=1\, .
    \end{equation}
    
    We begin by arguing that \eqref{eq: subcritical asymptotics} together with the assumption that $\liminf_{j\to \infty}\beta_j/\sqrt{\lambda_j}>\beta(\gamma+1/2, d-1)$ implies that
    \begin{equation*}
        \lim_{j\to \infty} l_1^{(j)}\sqrt{\lambda_j} = \infty\, .
    \end{equation*}
    We argue by contradiction, assume that $\liminf_{j\to \infty}l_1^{(j)}\sqrt{\lambda_j}< \infty$. By passing to a subsequence we may assume without loss of generality that the limit $\lim_{j\to \infty}l_1^{(j)}\sqrt{\lambda_j}$ exists and is finite and that $\lim_{j\to \infty}\beta_j/\sqrt{\lambda_j}$ either exists or is infinite.

    If $\lim_{j\to \infty}l_1^{(j)}\sqrt{\lambda_j}=0$, then since $|R_j|=1, \lim_{j\to \infty}\lambda_j = \infty, $ and $\liminf_{j\to \infty}\beta_j/\sqrt{\lambda_j}>0$ we can apply Lemma~\ref{lem: super-critical collapse} concluding that
    \begin{equation*}
        \lim_{j\to \infty}\frac{\Tr(-\Delta_{R_j}^{\beta_j}-\lambda_j)_\limminus^{\gamma}}{L_{\gamma, d}^{\rm sc}\lambda_j^{\gamma+d/2}}=0\, , 
    \end{equation*}
    which contradicts \eqref{eq: subcritical asymptotics}. If instead $0<\lim_{j\to \infty}l_1^{(j)}\sqrt{\lambda_j}<\infty$, then we can apply Lemma~\ref{lem: collapsing Weyl law} to deduce that there exist an integer $m \in \{1, \ldots, d-1\}$ and a cuboid $R' \subset \R^m$ such that
    \begin{equation*}
        \lim_{j\to \infty}\frac{\Tr(-\Delta_{R_j}^{\beta_j}-\lambda_j)_\limminus^{\gamma}}{L_{\gamma, d}^{\rm sc}\lambda_j^{\gamma+d/2}}=\begin{cases}
        \displaystyle\frac{\Tr(-\Delta_{R'}^{\beta'}-1)_\limminus^{\gamma+(d-m)/2}}{L_{\gamma+(d-m)/2, m}^{\rm sc}|R'|}& \mbox{if }\lim_{j\to \infty}\beta_j/\sqrt{\lambda_j} = \beta'\, , \\[12pt]
        \displaystyle\frac{\Tr(-\Delta_{R'}^{\rm D}-1)_\limminus^{\gamma+(d-m)/2}}{L_{\gamma+(d-m)/2, m}^{\rm sc}|R'|}& \mbox{if }\lim_{j\to \infty}\beta_j/\sqrt{\lambda_j} = \infty\, .
        \end{cases}
    \end{equation*}
    By Corollary~\ref{cor: beta ineqs}, $\beta(\gamma+1/2, d-1) \geq \beta(\gamma+(d-m)/2, m)$ for each $m \in \{1, \ldots, d-1\}$, and therefore Theorem~\ref{thm: Robin BLY} together with monotonicity of $\beta \mapsto \Tr(-\Delta_{R'}^{\beta}-1)_\limminus^{\gamma}$ tells us that
    \begin{equation*}
        \lim_{j\to \infty}\frac{\Tr(-\Delta_{R_j}^{\beta_j}-\lambda_j)_\limminus^{\gamma}}{L_{\gamma, d}^{\rm sc}\lambda_j^{\gamma+d/2}}<1\, , 
    \end{equation*}
    which again contradicts \eqref{eq: subcritical asymptotics}. Therefore, we conclude that $\lim_{j\to \infty}l_1^{(j)}\sqrt{\lambda_j}=\infty$. 

    Since the unit cube $Q = \prod_{i=1}^d(0, 1)$ is admissible in the optimization problem defining $M_{\gamma, d}(\lambda, \beta)$, we have by Theorem~\ref{thm: two-term Weyl} that
    \begin{equation}\label{eq: M two-term below}
    \begin{aligned}
        M_{\gamma, d}(\lambda_j, \beta_j) &\geq \Tr(-\Delta_{Q}^{\beta_j}-\lambda_j)_\limminus^\gamma\\
        &= L_{\gamma, d}^{\rm sc}\lambda_j^{\gamma+d/2}+ \frac{1}{4}L_{\gamma, d-1}(\beta_j/\sqrt{\lambda_j})\Haus^{d-1}(\partial Q)\lambda_j^{\gamma+(d-1)/2} + o(\lambda_j^{\gamma+(d-1)/2})\, , 
    \end{aligned}
    \end{equation}
    as $j \to \infty$.

    Similarly, our assumptions on $\{R_j\}_{j\geq 1}$ together with Theorem~\ref{thm: two-term Weyl} yields that
    \begin{equation}\label{eq: M two-term above}
    \begin{aligned}
        M_{\gamma, d}(\lambda_j, \beta_j) &= \Tr(-\Delta_{R_j}^{\beta_j}-\lambda_j)_\limminus^\gamma+o(\lambda_j^{\gamma+(d-1)/2})\\
        &\leq L_{\gamma, d}^{\rm sc}\lambda_j^{\gamma+d/2}+ \frac{1}{4}L_{\gamma, d-1}(\beta_j/\sqrt{\lambda_j})\Haus^{d-1}(\partial R_j)\lambda_j^{\gamma+(d-1)/2}\\
        &\quad + C_{\gamma, d}\Haus^{d-1}(\partial R_j)\lambda_j^{\gamma+(d-1)/2}((l_1^{(j)}\sqrt{\lambda_j})^{-\kappa_{\gamma, d}} +(l_1^{(j)}\sqrt{\lambda_j})^{1-d})\\
        &\quad + o(\lambda_j^{\gamma+(d-1)/2})\\
        &=L_{\gamma, d}^{\rm sc}\lambda_j^{\gamma+d/2}+ \frac{1}{4}L_{\gamma, d-1}(\beta_j/\sqrt{\lambda_j})\Haus^{d-1}(\partial R_j)\lambda_j^{\gamma+(d-1)/2}\\
        &\quad +o(\Haus^{d-1}(\partial R_j)\lambda_j^{\gamma+(d-1)/2})\, .
    \end{aligned}
    \end{equation}
    In the third step, we used that $\lim_{j\to \infty}l_1^{(j)}\sqrt{\lambda_j}=\infty$ and the bound $\Haus^{d-1}(\partial R_j) \gtrsim_d 1$. 
    
    Since $\beta(\gamma+1/2, d-1)\geq \beta_W(\gamma+1/2, d-2)= \beta_W(\gamma, d-1)$ the assumption that $\liminf_{j\to \infty}\beta_j/\sqrt{\lambda_j}>\beta(\gamma+1/2, d-1)$ together with the fact that $\beta \mapsto L_{\gamma, d-1}(\beta)$ is smooth and decreasing ensures that $\limsup_{j\to \infty}L_{\gamma, d-1}(\beta_j/\sqrt{\lambda_j})<0$. Combining this observation with \eqref{eq: M two-term below} and \eqref{eq: M two-term above}, we conclude that
    \begin{equation*}
        \Haus^{d-1}(\partial Q) \geq \limsup_{j\to \infty}\Haus^{d-1}(R_j)(1+o(1))\, .
    \end{equation*}
    Since the perimeter among unit measure cuboids is uniquely minimized by the cube, we must have $\lim_{j\to \infty}\Haus^{d-1}(R_j)= \Haus^{d-1}(\partial Q)$ and the sequence $\{R_j\}_{j\geq 1}$ converges to the unit cube. This completes the proof of Theorem~\ref{thm: geometric convergence general}.
\end{proof}

\section{Comparison of \texorpdfstring{$\beta_W$}{bW} and \texorpdfstring{$\beta^*$}{b*}}
\label{sec: the critical beta are different}

We close the paper by analyzing the two values of $\beta$ which play a central role in our analysis, namely, the critical Robin parameter $\beta(\gamma, d)$ from Theorem~\ref{thm: Robin BLY} and the point where the second term in the two-term Weyl formula changes sign, $\beta_W(\gamma, d)$. Our main purpose is to demonstrate that, as we claimed in the introduction, there are $d, \gamma, $ and $\beta$ for which the heuristic argument for the asymptotic behavior of maximizers of $M_{\gamma, d}(\lambda, \beta \sqrt{\lambda})$ based on two-term asymptotics fails. In other words, we aim to show that there exist pairs $d\geq 1$ and $\gamma>0$ for which $\beta(\gamma, d)>\beta_W(\gamma, d-1) $. In fact, we believe that this inequality always holds, although we are currently unable to prove this conjecture.
\begin{conjecture}\label{conj: betas}
    For any $d\geq 1$ and $\gamma > 0$, 
    $$\beta(\gamma, d)>\beta_W(\gamma, d-1) \, .$$
\end{conjecture}

Recall that $\beta_W(\gamma, d-1) = \beta_W(\gamma+(d-1)/2, 0)$ and by Corollary~\ref{cor: beta ineqs}
\begin{equation*}
    \beta(\gamma, d) \geq \beta(\gamma+(d-1)/2, 1)  \, .
\end{equation*}
Therefore, if we find some $\gamma>0$ such that $\beta(\gamma, 1)>\beta_W(\gamma, 0)$, we also have that
for each $d\leq 2\gamma+1$
\begin{equation*}
    \beta(\gamma-(d-1)/2, d) \geq\beta(\gamma, 1) > \beta_W(\gamma, 0) = \beta_W(\gamma-(d-1)/2, d-1)\, .
\end{equation*}
Hence, to prove the above conjecture, it is enough to consider $d=1$.

The main support of Conjecture \ref{conj: betas} comes from the belief that for $\gamma>0$ and as $\lambda \to \infty$, 
\begin{equation}\label{eq: conj oscillatory term}
    \Tr(-\Delta_{(0, 1)}^{\beta \sqrt{\lambda}}-\lambda)_\limminus^\gamma = L_{\gamma, 1}^{\rm sc}\lambda^{\gamma+1/2}+ \frac{1}{2}L_{\gamma, 0}(\beta) \lambda^\gamma + G_{\gamma, \beta}(\sqrt{\lambda})\lambda^{\gamma/2} + o(\lambda^{\gamma/2})\, , 
\end{equation}
where $G_{\gamma, \beta}$ is a (non-trivial) $\pi$-periodic function whose integral over a period is zero. In particular, for $\beta = \beta_W(\gamma, 0)$ the second term vanishes and the existence of an oscillatory term of the suggested form would imply that the leading Weyl term can not be an upper bound for the Riesz mean for all $\lambda>0$. However, at the moment we only know how to justify the validity of such an asymptotic expansion in the limiting cases of Dirichlet or Neumann boundary conditions. Numerical support of \eqref{eq: conj oscillatory term} is presented in Figure~\ref{fig:asymptotic oscillations} in the next subsection.

In order to justify our claim that $\beta(\gamma, 1) >\beta_W(\gamma, 0)$ for some range of $\gamma$, let us first record a few facts about $\beta_W(\gamma, 0)$. 

\begin{lemma} \label{lem: betaW_properties}\ 
\begin{enumerate}[label=(\roman*)]
    \item\label{itm: betaW property 1} The map $\gamma \mapsto \beta_W(\gamma, 0)$ is smooth and decreasing.
    \item\label{itm: betaW property 2} $\beta_W(0, 0) = 1$ and $\beta_W(1/2, 0)= 3/4$.
    \item\label{itm: betaW property 3} It holds that
    \begin{align*}
        \lim_{\gamma \to \infty}\beta_W(\gamma, 0)\sqrt{\gamma} = c_*\, , 
    \end{align*}
    where $c_*$ is the unique solution of $2e^{x^2}\operatorname{erfc}(x)=1$.
\end{enumerate}
\end{lemma}
\begin{remark}
    Numerically one can verify that $c_*  \approx 0.769$.
\end{remark}

\begin{proof}[Proof of Lemma~\ref{lem: betaW_properties}]
The properties of $\gamma \mapsto \beta_W(\gamma, 0)$ listed in \ref{itm: betaW property 1} follow from the implicit function theorem combined with the smoothness of the map
$$
    (\beta, \gamma) \mapsto \frac{4}{\pi}\int_0^1(1-s^2)^{\gamma}\frac{\beta}{\beta^2+s^2}\, ds -1= \frac{8\gamma}{\pi}\int_0^1x(1-x^2)^{\gamma-1}\arctan\Bigl(\frac{x}{\beta}\Bigr)\, dx -1
$$
and the fact that this function is decreasing in both variables. Here we used the alternative expression for the integral proved in Lemma~\ref{lem: properties of L}.

The values of $\beta_W(\gamma, 0)$ listed in \ref{itm: betaW property 2} follow from that the integral defining $L_{\gamma, 0}(\beta)$ can be explicitly computed in these particular $\gamma$. Indeed, 
\begin{equation*}
    \frac{4}{\pi}\int_0^1(1-s^2)^{0}\frac{\beta}{\beta^2+s^2}\, ds  = \frac{4}{\pi}\arctan\Bigl(\frac{1}{\beta}\Bigr)
\end{equation*}
and
\begin{equation*}
    \frac{4}{\pi}\int_0^1(1-s^2)^{1/2}\frac{\beta}{\beta^2+s^2}\, ds  = 2\sqrt{1+\beta^2}-2\beta\, .
\end{equation*}
Using this one readily verifies the claimed values of $\beta_W(\gamma, 0)$.

It remains to prove \ref{itm: betaW property 3}. By Lemma~\ref{lem: properties of L} we have that $\beta_W(\gamma, 0)$ satisfies
\begin{equation*}
    \frac{8\gamma}{\pi}\int_0^1x(1-x^2)^{\gamma-1}\arctan\Bigl(\frac{x}{\beta_W(\gamma, 0)}\Bigr)\, dx = 1\, .
\end{equation*}
By changing variables in the integral to $x=t/\sqrt{\gamma-1}$ the equation becomes
\begin{equation*}
    \frac{8}{\pi}\frac{\gamma}{\gamma-1}\int_0^\infty t\Bigl(1-\frac{t^2}{\gamma-1}\Bigr)_\limplus^{\gamma-1}\arctan\Bigl(\frac{t}{\beta_W(\gamma, 0)\sqrt{\gamma-1}}\Bigr)\, dt = 1\, .
\end{equation*}
Note that $\gamma \mapsto (1-t^2/(\gamma-1))_\limplus^{\gamma-1}$ is monotone increasing and $\lim_{\gamma\to \infty}(1-t^2/(\gamma-1))_\limplus^{\gamma-1}= e^{-t^2}$ for each $t>0$. It follows from the monotone convergence theorem that
\begin{align*}
    \frac{8}{\pi}\frac{\gamma}{\gamma-1}&\int_0^\infty t\Bigl(1-\frac{t^2}{\gamma-1}\Bigr)_\limplus^{\gamma-1}\arctan\Bigl(\frac{t}{\beta_W(\gamma, 0)\sqrt{\gamma-1}}\Bigr)\, dt\\
    &= 
    \frac{8}{\pi}\int_0^\infty te^{-t^2}\arctan\Bigl(\frac{t}{\beta_W(\gamma, 0)\sqrt{\gamma} \, (1+o(1))}\Bigr)\, dt + o(1)
\end{align*}
as $\gamma \to \infty$. By an integration by parts, a substitution of variables, and using \cite[eq. 7.7.1]{NIST}, 
\begin{align*}
    \frac{8}{\pi}\int_0^\infty te^{-t^2}\arctan\Bigl(\frac{t}{B}\Bigr)\, dt
    =
    \frac{4}{B\pi}\int_0^\infty \frac{e^{-t^2}}{1+\bigl(\frac{t}{B}\bigr)^2}\, dt
    = 2 e^{B^2}\operatorname{erfc}(B)\, , 
\end{align*}
for any $B>0$. The integral representation of $e^{B^2}\operatorname{erfc}(B)$ used in the final step also implies that this is a smooth strictly decreasing function of $B>0$. 

We conclude that as $\gamma \to \infty$, 
\begin{equation*}
    2 e^{ \beta_W(\gamma, 0)^2 \gamma}\operatorname{erfc}(\beta_W(\gamma, 0)\sqrt{\gamma}) = 1 + o(1)\, .
\end{equation*}
The smoothness and monotonicity of $(0, \infty)\ni x \mapsto e^{x^2}\operatorname{erfc}(x)$ implies that as claimed 
\begin{equation*}
    \lim_{\gamma \to \infty} \beta_W(\gamma, 0) \sqrt{\gamma} = c^*\, .
\end{equation*}
This completes the proof of the lemma.
\end{proof}

Let us now take a closer look at $\beta(\gamma, 1)$. We define for $k \in \N$ and $\gamma > 0$
\begin{equation*}
    \beta^{(k)}(\gamma):= \inf\Bigl\{\beta>0 : \Tr(-\Delta_{(0, 1)}^{\beta \sqrt{\lambda}}-\lambda)_\limminus^\gamma \leq L_{\gamma, 1}^{\rm sc}\lambda^{\gamma+1/2}, \forall \lambda \in (\pi^2 (k-1)^2, \pi^2k^2]\Bigr\}\, .
\end{equation*}
Note that $\beta^{(k)}(\gamma)$ is positive and finite for each $k, \gamma$. Indeed, this follows from the continuity and monotonicity of $\beta \mapsto \Tr(-\Delta_{(0, 1)}^{\beta \sqrt{\lambda}}-\lambda)_\limminus^\gamma$ and the facts that as $\beta \to 0^\limplus$ the Riesz mean converges to that of the Neumann problem for which we know that the inequality fails, and conversely as $\beta \to \infty$ the Riesz mean converges to that of the Dirichlet problem for which the inequality holds. 
Clearly, $\beta(\gamma, 1) = \sup_{k\geq 1}\beta^{(k)}(\gamma)$. 

The advantage in studying $\beta^{(k)}(\gamma)$ is two-fold. First, restricting attention to a finite interval of $\lambda$ enables performing reliable numerics. Second, on each of the intervals we know exactly which eigenvalues give a non-trivial contribution to the Riesz mean. We believe that a difficulty in settling Conjecture~\ref{conj: betas} lies in that the quantity
\begin{equation*}
    \inf\bigl\{\lambda >0: \Tr(-\Delta_{(0, 1)}^{\beta(\gamma, 1)\sqrt{\lambda} }-\lambda)_\limminus^\gamma = L_{\gamma, 1}^{\rm sc}\lambda^{\gamma+d/2}\bigr\}
\end{equation*}
tends to infinity as $\gamma \to \infty$. As a consequence, we expect that one needs to understand a growing number of $\beta^{(k)}(\gamma)$ as $\gamma$ becomes large to settle the conjecture.

By Lemma~\ref{lem: betaW_properties} we have that $\beta_W(\gamma, 0)\leq 1$ and thus the following lemma shows that at least for small $\gamma$, the inequality $\beta(\gamma, 1)>\beta_W(\gamma, 0)$ must hold. 
\begin{lemma}\label{lem: betak limit}
    For each $k\geq 1$, 
    \begin{equation*}
        \lim_{\gamma \to 0^\limplus} \beta^{(k)}(\gamma) = \infty\, .
    \end{equation*}
    In particular, $\lim_{\gamma \to 0^\limplus}\beta(\gamma, 1) = \infty$.
\end{lemma}

\begin{proof}
    For any fixed $\lambda, \beta>0$ such that $\lambda$ is not an eigenvalue of $-\Delta_{(0, 1)}^{\beta \sqrt{\lambda}}$ it holds that
    \begin{equation*}
        \lim_{\gamma \to 0^\limplus} \Tr(-\Delta_{(0, 1)}^{\beta \sqrt{\lambda}}-\lambda)_\limminus^\gamma = \Tr(-\Delta_{(0, 1)}^{\beta \sqrt{\lambda}}-\lambda)_\limminus^0\, .
    \end{equation*}
    For any $k\in \N$ and each $\lambda \in [\lambda_k(-\Delta_{(0, 1)}^{\beta \pi k}), \pi^2k^2)$ we have that
    \begin{equation*}
        \Tr(-\Delta_{(0, 1)}^{\beta \sqrt{\lambda}}-\lambda)_\limminus^0 = k> \frac{1}{\pi}\lambda^{1/2}= L_{0, 1}^{\rm sc}\lambda^{1/2}\, . 
    \end{equation*} 
    Since $\gamma \mapsto L_{\gamma, 1}^{\rm sc}$ is continuous, we conclude that for any fixed $\beta$ there exists a $\gamma>0$ so small that $\beta^{(k)}(\gamma)\geq \beta$. This implies the desired conclusion.
\end{proof}

\subsection{Numerical results}

In Figures~\ref{fig:betaW_betak}--\ref{fig:asymptotic oscillations} we display numerical evidence that supports Conjecture~\ref{conj: betas}.

In Figure~\ref{fig:betaW_betak}, we show numerically obtained values for $\beta_W(\gamma, 0)$ and $\beta^{(k)}(\gamma)$ for $k=1, 2, 3$ in the range $\gamma \in [0, 20]$. Observe that for any $\gamma$ in the plotted range, at least one of the curves $\beta^{(k)}(\gamma)$ lies above $\beta_W(\gamma, 0)$ which numerically verifies Conjecture~\ref{conj: betas} up to $\gamma=20$ for $d=1$. It is also seen that $\beta^{(1)}(\gamma)$ exceeds $\beta^{(2)}(\gamma)$ and $\beta^{(3)}(\gamma)$ at first for small $\gamma$, but is surpassed by $\beta^{(2)}(\gamma)$ around $\gamma=2.5$. Similarly, $\beta^{(2)}(\gamma)$ is then surpassed by $\beta^{(3)}(\gamma)$ around $\gamma=14$. We expect that the $k$'s such that $\beta^{(k)}(\gamma)=\beta(\gamma, 1)$ grow as $\gamma \to \infty$.

\begin{figure}[t]
    \centering
    \begin{subfigure}[t]{0.48\textwidth}
        \centering
        \includegraphics[width=\textwidth]{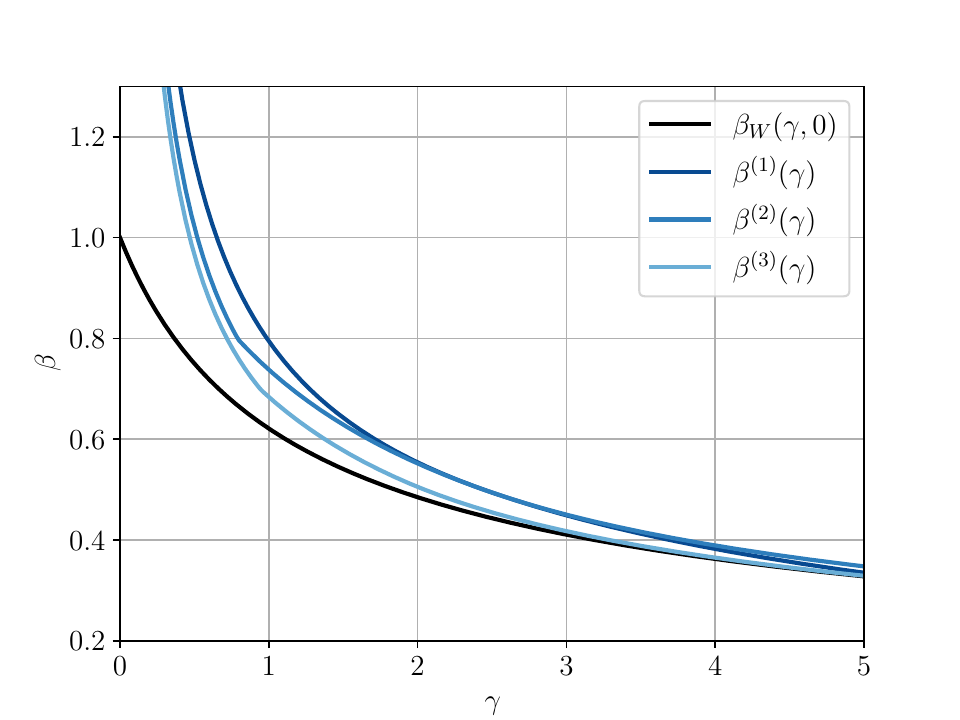}
        \caption{$\gamma \in [0, 5]$}
    \end{subfigure}\hfill
    \begin{subfigure}[t]{0.48\textwidth}
        \centering
        \includegraphics[width=\textwidth]{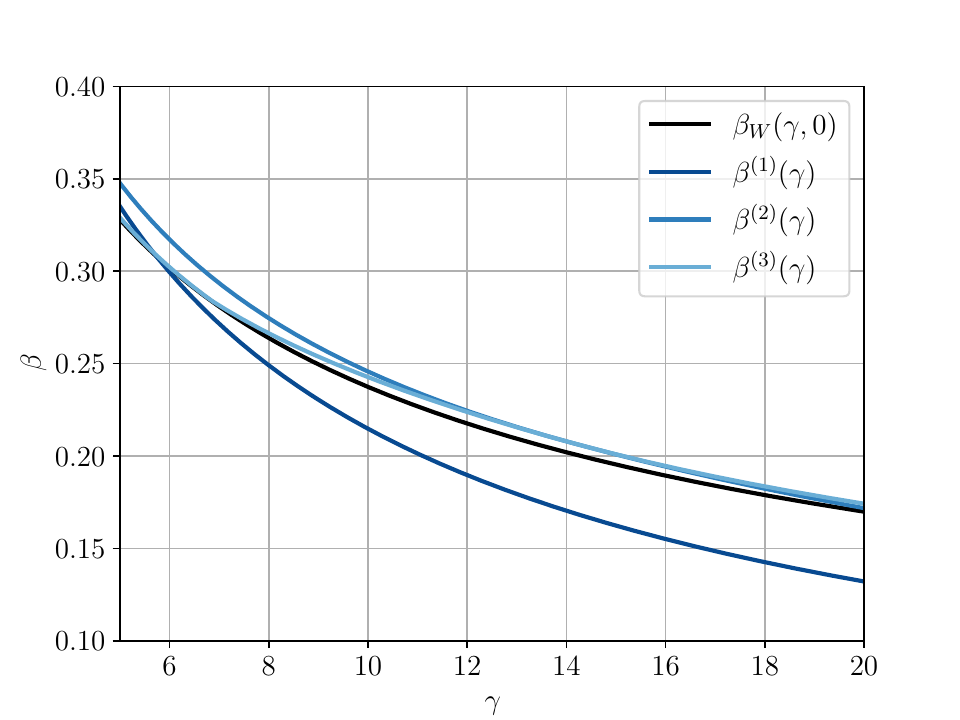}
        \caption{$\gamma \in [5, 20]$}
    \end{subfigure}
    \vspace{-5pt}
    \caption{$\beta_W(\gamma, 0)$ and $\beta^{(k)}(\gamma)$ for $k=1, 2, 3$.} \label{fig:betaW_betak} 
\end{figure}

Figure \ref{fig:thresholdnumbers_gamma} shows the difference
\begin{equation}\label{eq: normalized Weyl difference}
    \lambda^{-\gamma} \Bigl(L_{\gamma, 1}^{\rm sc} \lambda^{\gamma+1/2} -  \mathrm{Tr}(-\Delta_{(0, 1)}^{\beta \sqrt{\lambda}}  - \lambda )_\limminus^\gamma\Bigr)
\end{equation}
plotted against $\lambda$ for different combinations of $\beta$ and $\gamma$. By Corollary~\ref{cor: sc two-term Weyl}, the quantity in \eqref{eq: normalized Weyl difference}
converges to $-\frac{1}{2}L_{\gamma, 0}(\beta)$ as $\lambda \to \infty$. This convergence can be seen in Figure~\ref{fig:thresholdnumbers_gamma}, in particular, at $\beta = \beta_W(\gamma, 0)$ the difference appears to converge to zero as expected. Furthermore, the inequality $\mathrm{Tr}(-\Delta_{(0, 1)}^{\beta \sqrt{\lambda}}  - \lambda )_\limminus^\gamma \leq L_{\gamma, 1}^{\rm sc} \lambda^{\gamma+1/2}$ fails for any $\lambda$ such that the plotted curve drops below zero. In support of Conjecture~\ref{conj: betas}, this can be seen to happen in both plots when $\beta=\beta_W(\gamma, 0)$ but seizes to happen as $\beta$ is increased. The critical value $\beta(\gamma, 1)$ is defined so that the corresponding curve touches but never crosses zero. The plotted curves (specifically, the red curves) indicate that $\beta(\gamma, 1) = \beta^{(1)}(\gamma)$ when $\gamma=1$ and $\beta(\gamma, 1) = \beta^{(2)}(\gamma)$ when $\gamma=10$.

\begin{figure}[t]
    \centering
    \begin{subfigure}[t]{0.48\textwidth}
        \centering
        \includegraphics[width=\linewidth]{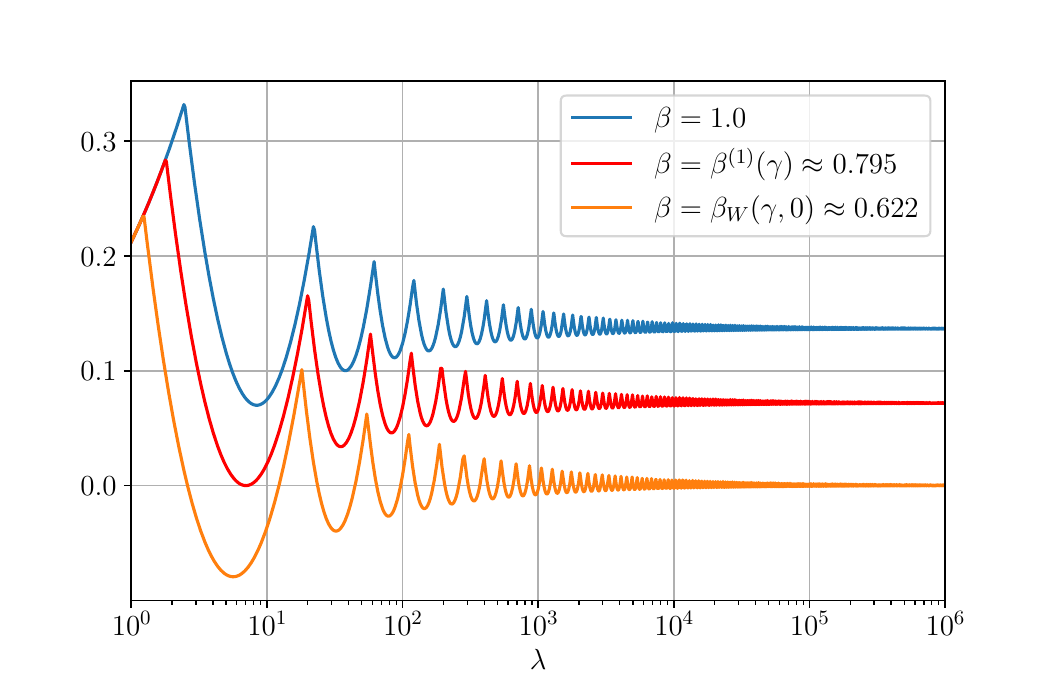}
    \caption{$\gamma = 1$}
    \end{subfigure}
    \hfill
    \begin{subfigure}[t]{0.48\textwidth}
        \centering
        \includegraphics[width=\linewidth]{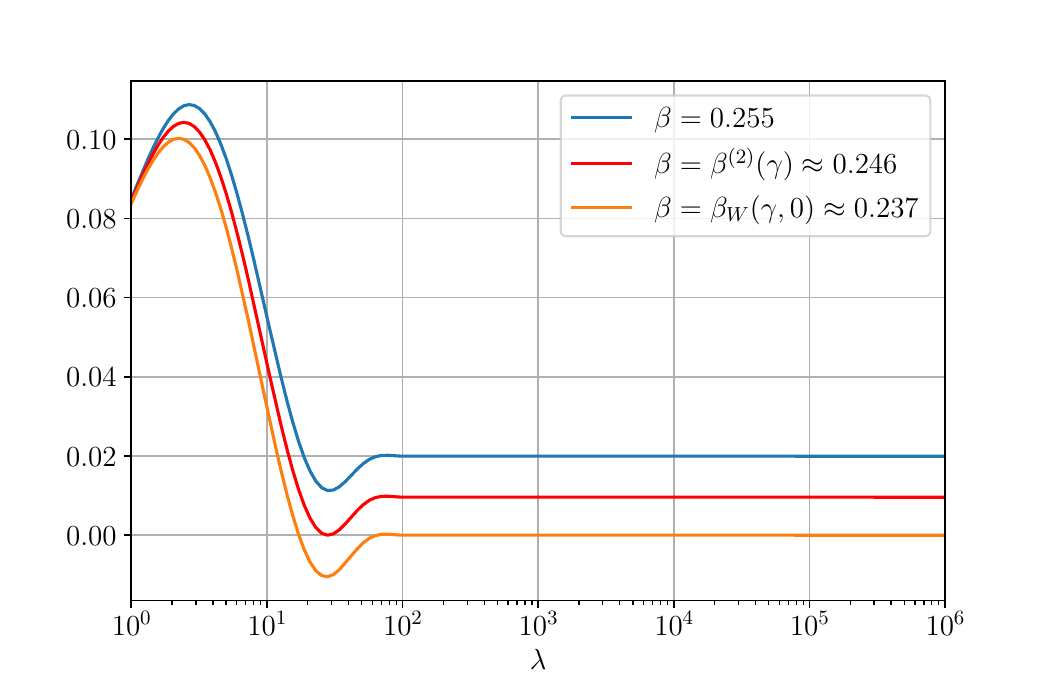}
    \caption{$\gamma = 10$}
    \end{subfigure}
    \vspace{-5pt}
    \caption{The quantity in \eqref{eq: normalized Weyl difference} shown for several combinations of $\gamma$ and~$\beta$.}  \label{fig:thresholdnumbers_gamma}
\end{figure}

\begin{figure}
    \centering
    \begin{subfigure}[t]{0.48\textwidth}
        \centering
        \includegraphics[width=\linewidth]{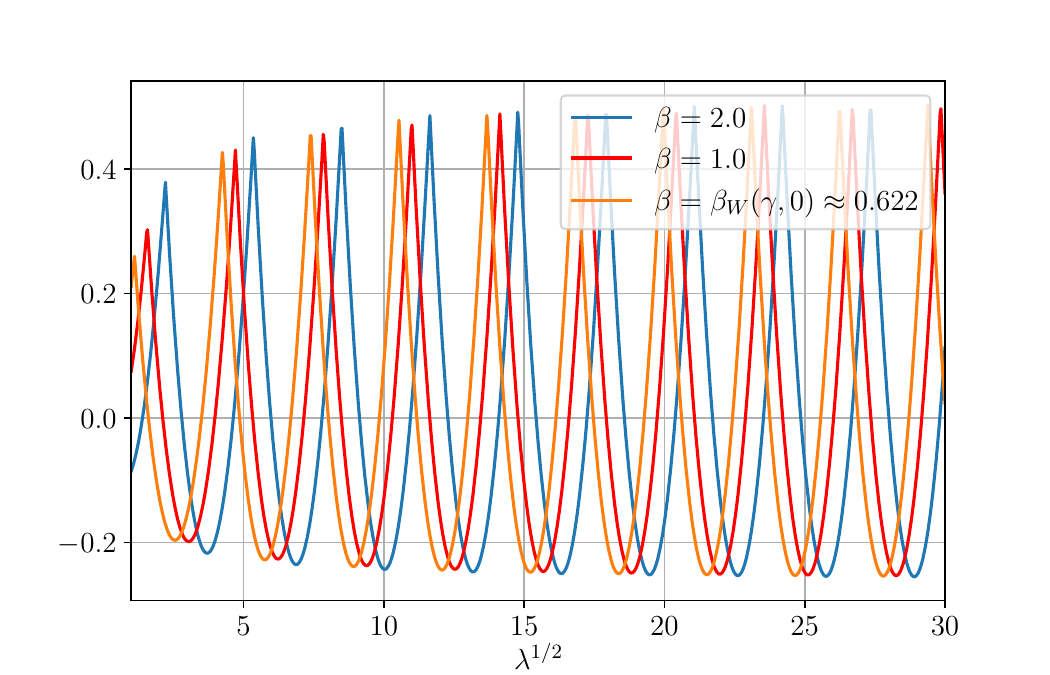}
    \caption{$\gamma = 1$}
    
    \end{subfigure}
    \hfill
    \begin{subfigure}[t]{0.48\textwidth}
        \centering
        \includegraphics[width=\linewidth]{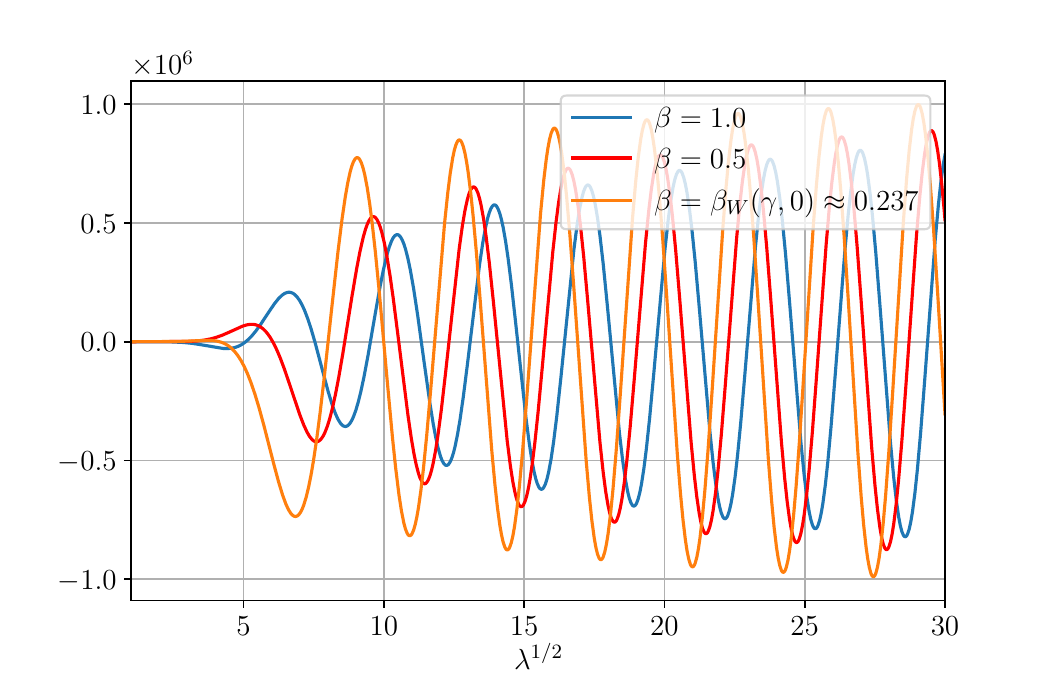}
    \caption{$\gamma = 10$}

    \end{subfigure}
    \vspace{-5pt}
    \caption{The quantity \eqref{eq: normalized Weyl difference2} shown for different combinations of $\gamma$ and $\beta$, illustrating the suggested periodic oscillations in $\sqrt{\lambda}$.} \label{fig:asymptotic oscillations}
\end{figure}

Finally, Figure \ref{fig:asymptotic oscillations} shows the quantity
\begin{equation}\label{eq: normalized Weyl difference2}
    \lambda^{-\gamma/2} \Bigl(L_{\gamma, 1}^{\rm sc} \lambda^{\gamma+1/2} + \frac{1}{2} L_{\gamma, 0}(\beta) \lambda^{\gamma}-  \mathrm{Tr}(-\Delta_{(0, 1)}^{\beta \sqrt{\lambda}}  - \lambda )_\limminus^\gamma\Bigr)
\end{equation}
plotted against $\lambda$ for several combinations of $\beta$ and $\gamma$. In line with the refined asymptotic expansion suggested in \eqref{eq: conj oscillatory term} the curves in the plot appear to oscillate around zero as $\lambda$ becomes large. When $\beta=\beta_W(\gamma, 0)$ the inequality $\mathrm{Tr}(-\Delta_{(0, 1)}^{\beta \sqrt{\lambda}}  - \lambda )_\limminus^\gamma \leq L_{\gamma, 1}^{\rm sc} \lambda^{\gamma+1/2}$ again fails for any $\lambda$ such that the plotted curve drops below zero. The plots show that for both $\gamma=1$ and $\gamma=10$ this failure appears to happen periodically in $\sqrt{\lambda}$.

\FloatBarrier


\bibliographystyle{amsplain}

\def\myarXiv#1#2{\href{http://arxiv.org/abs/#1}{\texttt{arXiv:#1\, [#2]}}}

\end{document}